\documentclass{article}

\usepackage[utf8]{inputenc}
\usepackage[T1]{fontenc}
\usepackage{microtype}
\usepackage{amsmath,amssymb,mathtools}
\usepackage{graphicx}
\usepackage{latexsym}
\usepackage{pxfonts,kpfonts}
\usepackage{fullpage}
\usepackage[colorinlistoftodos]{todonotes}
\usepackage{bbm}            
\usepackage{xypic}
\usepackage[all]{xy}

\usepackage[inline]{enumitem}
\setlist[itemize]{label={$\blacktriangleright$}}

\usepackage{tikz-cd}

\usepackage[numbered]{bookmark}
\usepackage{hyperref}
\hypersetup{
  pdfencoding=auto,      
  colorlinks,
  urlcolor = {blue!40!black},
  linkcolor = {red!40!black},
  citecolor = {green!40!black},
  pdfauthor = {Ricardo Campos, Julien Ducoulombier, and Najib Idrissi},
  pdftitle = {Boardman--Vogt resolutions and bar/cobar constructions of (co)operadic (co)bimodules}
}

\usepackage{amsthm}
\hypersetup{colorlinks}
\usepackage[all]{hypcap}

\theoremstyle{definition}
\newtheorem{defi}{Definition}[section]
\newtheorem{const}[defi]{Construction}
\newtheorem{expl}[defi]{Example}

\newtheorem{rmk}[defi]{Remark}

\newtheorem{notat}[defi]{Notation}

\theoremstyle{plain}
\newtheorem{mainthm}{Main Theorem}
\newtheorem{pro}[defi]{Proposition}
\newtheorem{thm}[defi]{Theorem}

\newtheorem{lmm}[defi]{Lemma}
\newtheorem{cordef}[defi]{Corollary-Definition}

\theoremstyle{remark}

\renewcommand{\SS}{{\mathbb S}}
\newcommand{\R}{{\mathbb R}}
\newcommand{\Q}{{\mathbb Q}}
\newcommand{\Map}{{\mathrm{Map}}}
\renewcommand{\hom}{{\mathrm{hom}}}
\newcommand{\Nat}{{\mathrm{Nat}}}
\newcommand{\Emb}{{\mathrm{Emb}}}
\newcommand{\Imm}{{\mathrm{Imm}}}
\newcommand{\Topo}{{\mathrm{Top}}}
\newcommand{\AllTop}{\mathrm{AllTop}}
\newcommand{\Embbar}{\overline{\Emb}}
\newcommand{\Immbar}{\overline{\Imm}}
\newcommand{\Seq}{{\mathrm{Seq}}}
\newcommand{\Operad}{{\mathrm{Operad}}}
\newcommand{\Bimod}{{\mathrm{Bimod}}}
\newcommand{\Ibimod}{{\mathrm{Ibimod}}}
\newcommand{\Comm}{{Comm}}

\newcommand{\TT}{{\mathrm{T}}}
\newcommand{\MM}{{\mathsf{M}}}

\newcommand{\calO}{{\mathcal O}}
\newcommand{\calB}{{\mathcal B}}
\newcommand{\calI}{{\mathcal I}}
\newcommand{\calF}{{\mathcal F}}
\newcommand{\calU}{{\mathcal U}}
\newcommand{\calIF}{{\mathcal{IF}}}
\newcommand{\calBF}{{\mathcal{BF}}}

\newcommand{\hocolim}{\operatorname{hocolim}}

\title{Delooping the functor calculus tower}
\date{}
\author{Julien Ducoulombier \and Victor Turchin}

\begin{document}

\maketitle

\begin{abstract}
We study a connection between  mapping spaces of bimodules and  of infinitesimal bimodules over an operad. As  main application and motivation of our work, we produce an explicit delooping of the
manifold calculus tower associated to the space of smooth maps $D^m\to D^n$ of discs, $n\geq m$, avoiding any given multisingularity and coinciding with the standard inclusion near the boundary $\partial D^m$. In particular, we give a new proof of the delooping of the space of disc embeddings in terms of little discs operads maps with the advantage that it can be applied to more general mapping spaces. 
 As a spin-off result we discover a homotopy recurrence relation on the components of the 
little discs operads.
\end{abstract}

\tableofcontents

\section{Introduction}\label{s:intro}

\subsection{Functor calculus on a closed disc for non-singular mapping spaces}\label{ss11}

The calculus of functors on manifolds was invented by T.~Goodwillie and M.~Weiss in order to study 
spaces of smooth embeddings~\cite{Weiss99,Weiss99.2}. The approach is universal and in particular can be applied to the study of more general   spaces of maps between smooth manifolds avoiding any given type $\SS$ of multisingularity. The idea of this method goes back to Smale\rq{}s study of immersions~\cite{Smale59} and 
to the Gromov $h$-principle~\cite{Gromov86}, that suggest replacing the space of maps avoiding any given singularity with the space of sections
of a jet bundle. In case the singularity condition depends on more than one point, one can consider similar spaces of sections of multijet bundles 
over configuration spaces of points in the source manifold. By doing so one should take into account that  points in configurations can collide
or be forgotten. The manifold calculus  keeps track of all these data in a homotopy invariant way.  The $k$-th approximation 
in this method is obtained by restricting the number of points in configurations to be~$\leq k$. \vspace{5pt}

For $m\leq n$, and any given multisingularity $\SS$, specified by a multijet condition,
see~\cite{Vassiliev92}, consider the space $\Map_\partial^\SS(D^m,D^n)$ of smooth maps $D^m\to D^n$
avoiding $\SS$ and coinciding with the standard inclusion of discs $i\colon D^m\subset D^n$ near the boundary $\partial D^m$. The multisingularity $\SS$ must be {\it closed} in the sense that if $f$ is a limit of 
$\SS$-singular maps, then $f$ should also be $\SS$-singular. Examples of such mapping spaces are
the spaces $\Emb_\partial(D^m,D^n)$ and $\Imm_\partial(D^m,D^n)$ of embeddings and
immersions, respectively. As another example, one can consider the space $\Imm_\partial^{(\ell)}(D^m,D^n)$
of non-$\ell$-equal immersions, $\ell\geq 2$, i.e. immersions $f\colon D^m\looparrowright D^n$
for which any subset of $\ell$ points in $D^m$ has more than one point in the image. One obviously has
$\Imm_\partial^{(2)}(D^m,D^n)=\Emb_\partial(D^m,D^n)$.  Self-tangency would be  another
example of a possible forbidden multi-singularity, or one can consider any  mixed
condition on self-intersection and singularity type at intersection points. \vspace{5pt}

Let $\calO_\partial(D^m)$ be the category of open subsets of $D^m$  containing the boundary $\partial D^m$. For any contravariant functor $F\colon \calO_\partial(D^m)\to\Topo$ to topological 
spaces, that  sends isotopy equivalences to weak homotopy equivalences, the functor calculus assigns a {\it Taylor tower} of polynomial approximations to $F$:
\begin{equation}
\xymatrix{
&F\ar[dl] \ar[d] \ar[dr] \ar[drr]\\
T_0F&T_1F\ar[l]&T_2F\ar[l]&T_3F\ar[l]&\cdots\ar[l]
}
\label{eq_tower}
\end{equation}
We say that the tower~\eqref{eq_tower} {\it  converges to $F$} (or simply {\it converges}) if the natural map
$F\to T_\infty F$ to the limit of the tower is a weak equivalence of functors (by this we mean objectwise weak
equivalence). In practice we are usually concerned with the convergence of the Taylor tower~\eqref{eq_tower}
evaluated on $U=D^m$. \vspace{5pt}

 To study the space $\Map_\partial^\SS(D^m,D^n)$, we consider the functor  $\Map_\partial^\SS(-,D^n)$
 assigning to $U\in\calO_\partial(D^m)$ 
 $$
 U\mapsto \Map_\partial^\SS(U,D^n)
 $$
 the space of $\SS$-non-singular maps $U\to D^n$ coinciding with $i\colon D^m\subset D^n$ near
 $\partial D^m$. One of the main results of the embedding calculus~\cite{Goodwillie15} is that it can be
 always applied when codimension is at least three:
 \begin{equation}\label{eq:emb_conv}
 T_\infty\Emb_\partial(D^m,D^n)\simeq \Emb_\partial(D^m,D^n), \quad n-m\geq 3.
 \end{equation}
 Note that even though the multi-singularity condition is expressed using only two points, we
 still need to go to the limit of the tower, which takes into account configurations of arbitrary large number of points, in order to recover the initial embedding space.\vspace{5pt}
 
Unfortunately, for other types of $\SS$ the question of convergence of the Goodwillie-Weiss tower has not yet been studied. In particular one does not have  any results for non-$\ell$-equal immersions
with $\ell\geq 3$. One should mention that for general mapping spaces, the methods 
of the embedding calculus do not work and the question of convergence  appears to be very hard. 
Besides the Goodwillie-Weiss calculus method, one also has the Vassiliev theory of discriminants that
was used to study such spaces~\cite{Vassiliev92}. However, for embedding spaces the discriminant method
has a more restricted range where it is applicable $n\geq 2m+2$ compared to the calculus approach,
  which works for
the range  $n\geq m+3$. Thus one still anticipates that for general mapping spaces avoiding any given
multisingularity $\SS$, the calculus method works and in particular can be applied in cases where the
discriminant theory cannot. In particular we hope that the delooping result that we produce in this
paper will encourage more
 studies in this direction.

\subsection{Action of the little discs operad and a few more examples of mapping spaces}\label{ss12}

The spaces $\Map_\partial^\SS(D^m,D^n)$  
  are naturally algebras over the little $m$-discs operad $\calB_m$. Recall that an element $b\in\calB_m(k)$ is a configuration of $k$ discs $D^m_i$, $i=1\ldots k$,
with disjoint interiors in the unit disc $D^m$, where each disc is the image of a linear map $L_i\colon D^m\hookrightarrow D^m$,
which is a composition of translation and rescaling. Given such $b$ and $f_i\in \Map_\partial^\SS(D^m,D^n)$, $i=1\ldots k$, the action in question is defined as follows:
$b(f_1,\ldots,f_k)\in \Map_\partial^\SS(D^m,D^n)$ is the map $D^m\to D^n$, which is the standard
inclusion on $D^m\setminus\cup_{i=1}^k D^m_i$, and is $\hat L_i\circ f_i\circ L_i^{-1}$ on $D_i^m$, where
$\hat L_i\colon \R^n\to\R^n$ is the obvious extension of $L_i$, so that it is also a composition of translation and rescaling. \vspace{5pt}

A typical example of a $\calB_m$-algebra is an iterated $m$-loop space. Moreover, by the celebrated 
May-Boardman-Vogt recognition principle, under the condition $\pi_0X$ is a group, any $\calB_m$-algebra $X$ is weakly equivalent to an $m$-loop space $\Omega^mY$~\cite{May72,Boardman73}.  In this paper we give an explicit
$m$-delooping of the tower $T_\bullet \Map_\partial^\SS(D^m,D^n)$. Before doing this, consider several other examples to which our delooping construction applies.\vspace{5pt}

  Define $\Emb_\partial^{fr}(D^m,D^n)$ as the space of framed embeddings $D^m\hookrightarrow D^n$, i.e. embeddings with trivialization of the normal bundle 
standard near $\partial D^m$. Define also $\Embbar_\partial(D^m,D^n)$ as the homotopy fiber of the
inclusion $\Emb_\partial(D^m,D^n)\hookrightarrow\Imm_\partial(D^m,D^n)$ over $i\colon D^m\subset D^n$. As another example
we consider $\Embbar_\partial(D^m,D^n)^\Q$ -- rationalization of $\Embbar_\partial(D^m,D^n)$. The common feature of these three examples $\Emb_\partial^{fr}(D^m,D^n)$, $\Embbar_\partial(D^m,D^n)$,
$\Embbar_\partial(D^m,D^n)^\Q$ is that they are all algebras over $\calB_{m+1}$. For the 
first example, this has been shown by R.~Budney~\cite{Budney07}. For the second one, 
see~\cite{Sakai14,Turchin10}. The third example was considered in~\cite{FresseTWbig}:
as a consequence of the above, it is an algebra over $\calB_{m+1}^\Q$ and thus over $\calB_{m+1}$ by restriction.  The idea of this  little discs action in one higher dimension is that embeddings can be shrunken 
and pulled one through another as in the proof of the commutativity of 
 the monoid $\pi_0 \Emb_\partial(D^1,D^3)$ of isotopy classes of classical long knots. 

Finally, define $\Immbar{}_\partial^{(\ell)}(D^m,D^n)$ as the homotopy fiber of the
inclusion $\Imm_\partial^{(\ell)}(D^m,D^n)\hookrightarrow\Imm_\partial(D^m,D^n)$ over $i\colon D^m\subset D^n$. When $\ell\geq 3$, this space is not a $\calB_{m+1}$-algebra as pulling 
non-$\ell$-equal immersions one through another might create self-intersections of higher degree, but it is still a 
$\calB_m$-algebra. The pair of spaces $(\Immbar_\partial^{(\ell)}(D^m,D^n),\Embbar_\partial(D^m,D^n))$ is actually an algebra over the {\it extended Swiss cheese operad}
$\mathcal{ESC}_{m,m+1}$ considered in~\cite{Willwacher17}. Informally this means that both non-$\ell$-equal
immersions can be shrunken and pulled through embeddings and embeddings can be shrunken and pulled
through immersions. It would be interesting to find a relative delooping of this pair of spaces that would account to this extended Swiss cheese action. The relative delooping with respect to the action of the usual Swiss cheese operad~\cite{Voronov99} has been obtained (modulo convergence of the tower and our Main Theorem~\ref{th:delooping1}) by the first author in~\cite{Ducoulombier16}. \vspace{5pt}

Recent developments in the manifold calculus of functors allows one to express the tower~\eqref{eq_tower}
in terms of derived mapping spaces of (truncated) infinitesimal bimodules over 
$\calB_m$~\cite{Arone14,Turchin13}. The main result of this paper -- we show that in certain cases,
as in all examples  above, these towers admit an explicit $m$-th or $(m+1)$-th delooping in terms 
of derived mapping spaces of (truncated) bimodules over $\calB_m$ or of (truncated) operads.
This approach will translate many difficult geometrical problems to a not necessarily easy, but definitely
interesting 
algebraic framework. Also the obtained deloopings  are more highly connected than the initial spaces, which will
allow the use of   rational homotopy theory to study them. As a particular example, the delooping
of $\Embbar_\partial(D^m,D^n)$, $n-m\geq 3$, produced earlier by Boavida de Brito and 
Weiss in~\cite{Weiss15} and the deloopings of $T_\infty\Embbar_\partial(D^m,D^n)^\Q$ 
and  of $T_k\Embbar_\partial(D^m,D^n)^\Q$, that follow from our work, see~\eqref{eq:Qdeloop}-\eqref{eq:Qdeloop2}, were recently used in~\cite{FresseTWbig} to  produce
a complete rational understanding of the spaces $\Embbar_\partial(D^m,D^n)$, $n-m\geq 3$, and
 the towers $T_k\Embbar_\partial(D^m,D^n)$, $n-m\geq 2$.\vspace{5pt}

In case the tower~\eqref{eq_tower} does not converge, one can still consider the induced 
map $\pi_0F(D^m)\to \pi_0T_kF(D^m)$. According to our delooping result for
 $F=\Map_\partial^\SS(-,D^n)$ (or any other functor considered in this subsection), this map
 produces an invariant of isotopy classes of $\SS$-non-singular maps that takes values in
 an  (abelian) group. Such invariants can be of interest. For example, in the classical case of knots in a 
 three-dimensional space $\Emb_\partial(D^1,D^3)$, this map
 was shown to be an integral additive Vassiliev invariant of order $\leq k-1$, and it is conjectured 
 to be the universal one of this type~\cite{Budney17}.

\subsection{(Truncated) operads, bimodules, infinitesimal bimodules}\label{ss13}

In this subsection we recall some standard definitions from the theory of operads and fix notation. 
This will be necessary to formulate our main results. 

\subsubsection{$\Sigma$-sequences, operads, bimodules, infinitesimal bimodules.}\label{ss131}
By a {\it $\Sigma$-sequence} we mean a family of topological spaces $M=\{M(n),\, n\geq 0\}$, endowed
with a right action of the symmetric group $M(n)\times\Sigma_n\to M(n)$.  We denote the category of $\Sigma$-sequences by $\Sigma\Seq$. This category is endowed with a monoidal structure 
$(\Sigma\Seq,\circ,\mathbb{1})$, where $\circ$ is the composition product~\cite{Fresse09},
and $$\mathbb{1}(n)=\begin{cases} *,& n=1;\\ \emptyset,& n\neq 1.\end{cases}$$  An {\it operad} is a monoid
with respect to this structure. Explicitly, the structure of an operad is determined by the operadic compositions:
\begin{equation}\label{A2}
\circ_{i}:O(n)\times O(m)\longrightarrow O(n+m-1),\hspace{15pt} \text{with }1\leq i\leq n, 
\end{equation}
and the unit element $\ast_1\in O(1)$, satisfying compatibility with the action of the symmetric group, associativity, commutativity and unit axioms. A map between two operads should respect the operadic compositions.  We denote by $\Operad$ the categories of operads. 

\begin{expl}\label{ex:operad1}
Our two main examples of  operads are the little discs operad $\calB_m$
and the Fulton-MacPherson operad $\calF_m$, which is equivalent to $\calB_m$, see~\cite{Salvatore99}.
We assume that the reader is familiar with these two examples. Main properties of $\calF_m$ are recalled
at the beginning of Section~\ref{s:FM}.
\end{expl}

\begin{expl}\label{ex:operad2}
We can also consider rationalizations $\calB_n^\Q$, $\calF_n^\Q$ of $\calB_n$ and $\calF_n$, 
respectively, see~\cite{FresseTWbig}. One has natural maps $\calB_n\to \calB_n^\Q$,
$\calF_n\to\calF_n^\Q$.
\end{expl}

\begin{expl}\label{ex:operad3}
Framed little discs operad $\calB_n^{fr}$ and framed Fulton-MacPherson operad $\calF_n^{fr}$~\cite{Salvatore99}.
One has natural inclusions $\calB_n\to\calB_n^{fr}$, $\calF_n\to \calF_n^{fr}$. 
\end{expl}

A {\it left module}, {\it right module}, and {\it bimodule over an operad $O$} is a symmetric sequence $M$ endowed with the 
structure of a left module, right module, or bimodule, respectively, over the monoid ($=$ operad) $O$.  Explicitly, the
structure of a bimodule is given by a family of maps
\begin{equation}\label{A3}
\begin{array}{llr}\vspace{7pt}
\gamma_{r}: & M(n)\times O(m_{1})\times \cdots\times O(m_{n})  \longrightarrow  M(m_{1}+\cdots + m_{n}), & \text{right action},\\ \vspace{7pt}
\gamma_{l}: & O(n)\times M(m_{1})\times \cdots\times M(m_{n})  \longrightarrow  M(m_{1}+\cdots + m_{n}),& \text{left action},
\end{array}
\end{equation}
satisfying compatibility with the action of the symmetric group, associativity and unity axioms (see \cite{Arone14,Fresse09}). In particular, the spaces $O(0)$ and $M(0)$ are $O$-algebras and there is a map of algebras $\gamma_{0}:O(0)\rightarrow M(0)$. A map between $O$-bimodules should respect these operations. We denote by $\Bimod_{O}$ the category of $O$-bimodules. Thanks to the unit in $O(1)$, the right action can equivalently be defined by a family of continuous maps
$$
\circ^{i}:M(n)\times O(m)\longrightarrow M(n+m-1),\hspace{15pt}\text{with }1\leq i\leq n.
$$
For the rest of the paper, we also use the following notation: 
$$
\begin{array}{ll}\vspace{7pt}
x\circ^{i}y=\circ^{i}(x\,;\,y) & \text{for } x\in M(n) \text{ and } y\in O(m),  \\ 
x(y_{1},\ldots,y_{n})=\gamma_{l}(x\,;\,y_{1}\,;\ldots;\,y_{n}) & \text{for } x\in O(n) \,\,\text{ and } y_{i}\in M(m_{i}).
\end{array} 
$$

\begin{expl}\label{ex:bimod1}
Given a map of operads $O\to P$, the target operad $P$ becomes a bimodule over $O$. For example, one has inclusions of operads $\calB_m\to \calB_n$ and $\calF_m\to\calF_n$, which
are induced by the coordinate inclusion $\R^m\subset\R^n$, $n\geq m$. Composing these maps with those from 
Examples~\ref{ex:operad2} and~\ref{ex:operad3}, we get operad maps $\calB_m\to\calB_n^\Q$,
$\calB_m\to\calB_n^{fr}$, and $\calF_m\to\calF_n^\Q$, $\calF_m\to\calF_n^{fr}$, $n\geq m$. As a consequence for $n\geq m$, the sequences $\calB_n$, $\calB_n^\Q$, $\calB_n^{fr}$ are bimodules
over $\calB_m$ and $\calF_n$, $\calF_n^\Q$, $\calF_n^{fr}$ are bimodules over $\calF_m$. 
\end{expl}

\begin{expl}\label{ex:bimod2}
For $n\geq 1$, $\ell\geq 2$, consider the sequences of spaces $\{\calB_n^{(\ell)}(j),\, j\geq 0\}$,
where $\calB_n^{(\ell)}(j)$ is the configuration space of $j$ discs in a unit disc $D^n$, defined as images of maps $L_i\colon D^n\to D^n$, each one being a composition of translation and rescaling, satisfying the 
{\it non-$\ell$-overlapping condition}:  no $\ell$ of them share a point in their interiors.  One obviously has
$\calB_n^{(2)}=\calB_n$.
The sequence
$\calB_n^{(\ell)}$ is a bimodule over $\calB_n$, see~\cite{Turchin14}, and thus a bimodule over $\calB_m$ 
by restriction. It  is called {\it bimodule of non-$\ell$-overlapping discs}. 
\end{expl}

\begin{expl}\label{ex:bimod3}
For any $m\leq n$ and any multisingularity $\SS$, the sequence
\[
\Map^\SS\left(\sqcup_\bullet D^m,D^n\right) =
\left\{ \Map^\SS\left(\sqcup_j D^m,D^n\right), \, j\geq 0\right\}
\]
is a bimodule over $\calB_m$ by taking pre- and post-composition.
Here $\Map^\SS\left(\sqcup_j D^m,D^n\right)$ denotes the space of smooth $\SS$-non-singular
maps $\sqcup_j D^m\to D^n$.  In fact it is
a $\calB_n$-$\calB_m$ bimodule, i.e. it has a left action of $\calB_n$ and a right action of $\calB_m$, 
which commute with each other. But we are interested only in its $\calB_m$-bimodule restriction.
\end{expl}

Finally, recall the notion of {\it infinitesimal bimodule}, which is less standard.  
To the best of our knowledge it appeared first in~\cite{Merkulov09} with this name, see also~\cite{Arone14}. In
the literature it is sometimes called {\it weak}, {\it abelian}, or {\it linear bimodule}~\cite{Dwyer12,Tourtchine10}. 
An \textit{infinitesimal bimodule} over $O$, or \textit{$O$-Ibimodule}, is a  sequence $N\in \Sigma\Seq$ endowed with operations 
\begin{equation}\label{C3}
\begin{array}{llr}\vspace{7pt}
\circ_{i}:O(n)\times N(m) 
\rightarrow N(n+m-1) & \text{for } 1\leq i\leq n,  & \text{infinitesimal left action,}\\ 
\circ^{i}:N(m)\times O(n)
\rightarrow N(n+m-1) & \text{for } 1\leq i\leq n, & \text{infinitesimal right action,}
\end{array} 
\end{equation}
satisfying unit, associativity, commutativity and compatibility with the symmetric group axioms, see \cite{Arone14}. We denote by $\Ibimod_{O}$ the category of infinitesimal bimodules. The infinitesimal right
action is equivalent to the usual right action, but it is not the case for the left action. In fact  
the existence of a left action does not imply the existence of  an infinitesimal left action, nor does the 
existence of  an infinitesimal left action imply the existence of a left action. 

\begin{expl}\label{ex:ibimod1}
Given a map of operads $\eta\colon O\to P$, the sequence $P$ inherits a structure of an infinitesimal bimodule over $O$: for $x\in O$, $y\in P$, one defines $x\circ_i y := \eta(x)\circ_i y$ and $y\circ^i x :=
y\circ_i \eta(x)$. For example, $\calB_n$, $\calB_n^\Q$, $\calB_n^{fr}$ are infinitesimal bimodules over $\calB_m$ and $\calF_n$, $\calF_n^\Q$, $\calF_n^{fr}$ are infinitesimal bimodules over $\calF_m$.
\end{expl}

\begin{expl}\label{ex:ibimod2}
As we mentioned earlier,  the structure of a bimodule and that of an infinitesimal bimodule do not imply
one another. However, if $M$ is a bimodule over an operad $O$  and one has a map of $O$-bimodules $\eta:O\rightarrow M$, then $M$ is also an infinitesimal bimodule over $O$. Since the right operations and the right infinitesimal operations are the same, we just need to define the left infinitesimal operations:
$$
\begin{array}{rcl}\vspace{4pt}
\circ_{i}:O(n)\times M(m) & \longrightarrow & M(n+m-1); \\ 
(\,x\,;\,y) & \longmapsto & \gamma_{l}(\,x\,;\,\underset{i-1}{\underbrace{\eta(\ast_{1}),\ldots,\eta(\ast_{1})}},y,\underset{n-i}{\underbrace{\eta(\ast_{1}),\ldots,\eta(\ast_{1})}}\,).
\end{array} 
$$ 
For example, one has the obvious inclusion $\calB_n\to\calB_n^{(\ell)}$, making $\calB_n^{(\ell)}$  into an infinitesimal bimodule over $\calB_n$ and also over $\calB_m$ by restriction. As another example,
one also has the inclusion $\calB_m\to \Map^\SS\left(\sqcup_\bullet D^m,D^n\right)$, making
$\Map^\SS\left(\sqcup_\bullet D^m,D^n\right)$ into  a  $\calB_m$-Ibimodule.
\end{expl}

Note  that Example~\ref{ex:ibimod1} is a particular case of Example~\ref{ex:ibimod2}.

\begin{expl}\label{ex:ibimod3}
For $m\leq n$, a multisingularity $\SS$, and a compact subset $K\subset D^n$ in the interior of the unit
 disc, define the space $\Map^\SS\left(\left(\sqcup_j D^m\right)\sqcup K,D^n\right)$ of maps
 \[
 f\colon \left(\sqcup_j D^m\right)\sqcup K\to D^n,
 \]
 such that $f|_{\sqcup_j D^m}$ is smooth $\SS$-non-singular, $f|_K$ is a composition of translation and rescaling and $f(K)$ lies in the interior of $D^n$ and is disjoint from $f(\sqcup_j D^m)$. Then 
 the sequence of spaces
 \[
 \Map^\SS\left(\left(\sqcup_\bullet D^m\right)\sqcup K,D^n\right) =\left\{\Map^\SS\left(\left(\sqcup_j D^m\right)\sqcup K,D^n\right),\, j\geq 0\right\}
 \]
 has a structure of a $\calB_m$-Ibimodule defined similarly by pre- and post-composition. 
 One should also notice that contrary to the previous examples, this sequence is not a $\calB_m$-bimodule.
 \end{expl}

\subsubsection{Truncated objects.}\label{sss132}
For any $k\geq 0$, we also consider the categories $\TT_k\Sigma\Seq$, $\TT_k\Operad$, $\TT_k\Bimod_O$,
$\TT_k\Ibimod$ of {\it $k$-truncated sequences, $k$-truncated operads, $k$-truncated bimodules,}
and {\it $k$-truncated infinitesimal bimodules}, respectively. A {\it $k$-truncated object} is a finite
sequence of spaces $\{M(j),\, 0\leq j\leq k\}$ with all the corresponding operations: $\Sigma$-action, unit $*_1$, compositions~\eqref{A2}, \eqref{A3}, \eqref{C3}, in the range where applicable, and satisfying the same compatibility axioms: associativity, commutativity, unity, and $\Sigma$-compatibility.  For example,
for $k$-truncated operads, we require $\Sigma$-action, unit $*_1\in O(1)$, and compositions~\eqref{A2}
with $n\leq k$, $m\leq k$, $n+m-1\leq k$. For $k$-truncated bimodules, besides the $\Sigma$-action,
we require right action~\eqref{A3} with $n\leq k$, $m_1+\ldots +m_n\leq k$, and left action~\eqref{A3}
with $m_1+\ldots+m_n\leq k$. For $k$-truncated Ibimodules, besides the $\Sigma$-action we 
require compositions~\eqref{C3} with $m\leq k$ and $n+m-1\leq k$. In particular, infinitesimal left
compositions $\circ_i\colon O(n)\times N(m)\to N(n+m-1)$ with $m=0$ and $n=k+1$ are allowed.\vspace{5pt}

One has obvious truncation functors, that abusing notation we always denote by $\TT_k$:
$$
\begin{array}{ll}\vspace{5pt}
\TT_k\colon \Sigma\Seq\to \TT_k\Sigma\Seq, & \TT_k\colon\Operad\to \TT_k\Operad, \\ 
\TT_k\colon\Bimod_O\to \TT_k\Bimod_O, &  \TT_k\colon \Ibimod_O\to \TT_k\Ibimod_O.
\end{array} 
$$

 \subsection{Taylor tower on a closed disc and the little discs operad: 
 Main Theorem~\ref{th:delooping1}}\label{ss14}
 Recent developments in the manifold calculus allows one to describe the Taylor tower
 in terms of derived mapping spaces of truncated right modules or infinitesimal bimodules~\cite{Arone14,
 Boavida13,Turchin13}. In case of the closed disc and a contravariant $F\colon\calO_\partial(D^m)\to\Topo$
 of certain \lq\lq{}context-free\rq\rq{} nature, one has
 \begin{equation}\label{eq:tower_ibimod}
\begin{array}{rcl}\vspace{4pt} 
T_\infty F(D^m) & \cong & \Ibimod_{\calB_m}^h(\calB_m,Ib(F));\\
T_kF(D^m) & \cong & \TT_k\Ibimod_{\calB_m}^h(\TT_k\calB_m,\TT_kIb(F)),\, k\geq 0,
\end{array} 
\end{equation}
 where $\Ibimod_{\calB_m}^h(-,-)$ and $\TT_k\Ibimod_{\calB_m}^h(-,-)$ denote the derived mapping spaces of infinitesimal bimodules and of $k$-truncated infinitesimal bimodules, respectively; $Ib(F)$ is a 
 $\calB_m$-Ibimodule naturally assigned to $F$. The table below relates the functors considered in
 Subsections~\ref{ss11}-\ref{ss12} and the corresponding to them Ibimodules described in Subsection~\ref{ss13}. 
 \begin{center}
 \scalebox{.88}{\begin{tabular}{c|c|c|c|c|c|c}
$F$&$\Embbar_\partial(-,D^n)$&$\Embbar_\partial(-,D^n)^\Q$&$\Emb_\partial^{fr}(-,D^n)$&$\Immbar{}_\partial^{(\ell)}(-,D^n)$&$\Map_\partial^\SS(-,D^n)$&$\Map_\partial^\SS(-,D^n\setminus K)$\\
\hline
$Ib(F)$&$\calB_n$&$\calB_n^\Q$&$\calB_n^{fr}$&$\calB_n^{(\ell)} $&$\Map^\SS\left(\sqcup_\bullet D^m, D^n\right)$&$\Map^\SS\left(\left(\sqcup_\bullet D^m\right)\sqcup K, D^n\right)$\\
\end{tabular}}
\end{center}
\vspace{.2cm}

In particular, for $n-m\geq 3$ the convergence of the tower~\eqref{eq:emb_conv} together
with~\eqref{eq:tower_ibimod} allows one to describe the spaces $\Emb_\partial(D^m,D^n)$, $\Embbar_\partial(D^m,D^n)$,
and $\Emb_\partial^{fr}(D^m,D^n)$ as spaces of derived maps of 
infinitesimal bimodules. For example, one has
\begin{equation}\label{eq:embbar_infbim}
\Embbar_\partial(D^m,D^n)\simeq \Ibimod_{\calB_m}^h(\calB_m,\calB_n),\hspace{15pt} \text{with } n-m\geq 3.
\end{equation}

This equivalence is a generalization of
Sinha\rq{}s cosimplicial model~\cite{Sinha06}  for the space $\Embbar_\partial(D^1,D^n)$ of knots, with $n\geq 4$. Indeed, in the latter case $\calB_1$ can be replaced by the associative operad. 
The right-hand side of~\eqref{eq:embbar_infbim}
becomes the homotopy totalization of a cosimplicial object as an infinitesimal bimodule
over the associative operad is the same thing as a cosimplicial object.\vspace{5pt}

In fact  equivalence~\eqref{eq:tower_ibimod} was proved in~\cite{Arone14} only for $F=
\Embbar_\partial(-,D^n)$, but the argument works for all functors of context-free nature as all the 
examples above.
Note that for all the functors $F$ from the table above except the last one, the space $F(D^m)$ is a $\calB_m$-algebra,
see Subsection~\ref{ss12}. Note also that all the corresponding infinitesimal bimodules $Ib(F)$ appear
as $\calB_m$-bimodules endowed with a map from $\calB_m$, see Example~\ref{ex:ibimod2}. Moreover,
the first three ones are operads, see Example~\ref{ex:ibimod1}, and exactly for the first three functors $F(D^m)$
is a $\calB_{m+1}$-algebra. As our main result, the theorem below, shows that this connection is not random. 

\begin{mainthm}\label{th:delooping1}
Let $\eta\colon\calB_m\to M$ be a morphism of $\calB_m$-bimodules, and also assume
that $M(0)=*$. Then one has an equivalence of towers:
\begin{equation}\label{eq:deloopingBm1}
\TT_k\Ibimod_{\calB_m}^h\left(\TT_k\calB_m,\TT_kM\right)\simeq 
\Omega^m \TT_k\Bimod_{\calB_m}^h\left(\TT_k\calB_m,\TT_kM\right),\, k\geq 0,
\end{equation}
implying at the limit $k=\infty$:
\begin{equation}\label{eq:deloopingBm2}
\Ibimod_{\calB_m}^h\left(\calB_m,M\right)\simeq 
\Omega^m \Bimod_{\calB_m}^h\left(\calB_m,M\right).
\end{equation}
\end{mainthm}
Here $\Bimod^h(-,-)$ and $\TT_k\Bimod^h(-,-)$ denote the derived mapping spaces of bimodules and of $k$-truncated bimodules, respectively. The basepoint for the loop spaces above is $\TT_k\eta$ and $\eta$, respectively. By a {\it tower} we mean
a sequence of spaces $X_\bullet$ together with morphisms
\[
X_0\leftarrow X_1\leftarrow X_2\leftarrow X_3\leftarrow\cdots,
\]
which are usually assumed to be fibrations. A morphism $X_\bullet\to Y_\bullet$ of towers is a sequence
of maps $X_k\to Y_k$ which make all the corresponding squares commute. Two towers are said to be equivalent if there is a zigzag of morphisms between them, which are all objectwise weak homotopy equivalences.\vspace{5pt}

Main Theorem~\ref{th:delooping1} has the following immediate corollary.

\begin{thm}\label{th:delooping2}
Let $\eta\colon \calB_m\to P$ be a map of operads, where $P(0)=*$ and $P(1)\simeq *$. Then
one has an equivalence of towers:
\begin{equation}\label{eq:deloopingBm3}
\TT_k\Ibimod_{\calB_m}^h\left(\TT_k\calB_m,\TT_k P\right)\simeq 
\Omega^{m+1} \TT_k\Operad^h\left(\TT_k\calB_m,\TT_k P\right),\, k\geq 1,
\end{equation}
implying at the limit $k=\infty$:
\begin{equation}\label{eq:deloopingBm4}
\Ibimod_{\calB_m}^h\left(\calB_m,P\right)\simeq 
\Omega^{m+1} \Operad^h\left(\calB_m,P\right).
\end{equation}
\end{thm}
As before for the basepoint in the loop spaces one takes $\TT_k\eta$ and $\eta$, respectively.
This theorem follows from Main Theorem~\ref{th:delooping1} and also the fact that for any map of 
topological operads $O\to P$, with $P(1)\simeq *$, one has an equivalence of towers
\[
\TT_k\Bimod_O^h(\TT_kO,\TT_kP)\simeq \Omega \TT_k\Operad^h(\TT_kO,\TT_kP),
\]
see~\cite{Ducoulombier17}. In the case of non-$\Sigma$ operads and $k=\infty$, this has been
shown earlier by Dwyer and Hess~\cite{Dwyer12}.\vspace{5pt}

For $m=1$, Theorems~\ref{th:delooping1} 
and~\ref{th:delooping2} were proved earlier by Dwyer-Hess~\cite{Dwyer12} and the second author~\cite{Tourtchine10}. To be precise this was proved for the associative operad,
which is equivalent to~$\calB_1$. Recently another proof appeared in~\cite{BataninDL17}.
This result is sometimes referred as the {\it concrete} topological version of Deligne\rq{}s Hochschild
cohomology conjecture. For  Deligne\rq{}s conjecture, now theorem, and its topological version 
see~\cite{McClure04.2,McClure04,Voronov00} and references in within. \vspace{5pt}

Several years ago Dwyer and Hess announced that they proved~\eqref{eq:deloopingBm2} and~\eqref{eq:deloopingBm4}. Their approach used the fact that $\calB_m$
is a homotopy Boardmann-Vogt tensor product of $m$ copies of $\calB_1$~\cite{Fiedorowicz15}, which allowed them to peel off the $m$ deloppings one after another. 
 As we explain in the introduction, this delooping result has many applications and is
 centrally important for the manifold functor calculus. Having a different proof
 would be helpful in  understanding this connection between algebra and geometry. Also their approach
 could be useful in describing the partial deloopings of the tower.
 We hope that they will fix  technical difficulties  and their proof will finally appear. \vspace{5pt}
 
 In our approach we get deloopings~\eqref{eq:deloopingBm1} and~\eqref{eq:deloopingBm2} in 
one step by taking explicit cofibrant replacements of the source objects which we replace to be 
the Fulton-MacPherson operad $\calF_m$ and its truncations instead of $\calB_m$. An advantage of our approach is that we get
delooping of all the stages of the tower and not only of their limits. \vspace{5pt}

Theorem~\ref{th:delooping2} produces an $(m+1)$-delooping of $T_\bullet\Embbar_\partial(D^m,D^n)$
and $T_\bullet\Embbar_\partial(D^m,D^n)^\Q$. In particular, one has
\begin{gather}\label{eq:del_boavida}
\Embbar_\partial(D^m,D^n)\simeq \Omega^{m+1}\Operad^h(\calB_m,\calB_n), \,\,\,  n\geq m+ 3;\\ 
T_k\Embbar_\partial(D^m,D^n)\simeq \Omega^{m+1}\TT_k\Operad^h(\TT_k\calB_m,\TT_k\calB_n),
\, n\geq m.
\end{gather}
These results were first proved by Boavida de Brita and Weiss~\cite{Weiss15,Weiss15_2}. Their method did not 
use~\eqref{eq:embbar_infbim} and also could not be applied in a general situation and in particular to any other examples 
considered above. Specifically it cannot be used to get the following deloopings that follow from our work:
\begin{gather}\label{eq:Qdeloop}
T_\infty\Embbar_\partial(D^m,D^n)^\Q\simeq \Omega^{m+1}\Operad^h(\calB_m,\calB_n^\Q), \,\,\, n\geq m;\\ 
T_k\Embbar_\partial(D^m,D^n)^\Q\simeq \Omega^{m+1}\TT_k\Operad^h(\TT_k\calB_m,\TT_k\calB_n^\Q).
\, n\geq m,\label{eq:Qdeloop2}
\end{gather}

Nonetheless, Theorem~\ref{th:delooping2}  cannot be applied to produce an $(m+1)$-delooping of 
the space $\Emb_\partial^{fr}(D^m,D^n)$ or of
the polynomial approximations
$T_\bullet\Emb_\partial^{fr}(D^m,D^n)$. Indeed, $\calB_n^{fr}(1)$ is the orthogonal group $O(n)$, which 
is not contractible. In fact it is easy to show that the inclusion $\calB_n\to \calB_n^{fr}$ induces an equivalence of towers
\[
\TT_k\Operad^h(\TT_k\calB_m,\TT_k \calB_n)\simeq \TT_k\Operad^h(\TT_k\calB_m,\TT_k \calB_n^{fr}),\, k\geq 1.
\]
Thus, $\Operad^h(\calB_m,\calB_n^{fr})$ cannot be an $(m+1)$-delooping of $\Emb_\partial^{fr}(D^m,D^n)$, $n-m\geq 3$. It has recently been showen by T.~Willwacher and the authors in~\cite{Ducoulombier18} that
\[
 \Emb_\partial^{fr}(D^m,D^n)\simeq \Omega^{m+1}\big(\, \Operad^h(\calB_m,\calB_n)// SO(n)\,\big), \,\, n\geq m+3;
\]
\[
T_k \Emb_\partial^{fr}(D^m,D^n)\simeq \Omega^{m+1}\big(\, \TT_k\Operad^h(\TT_k\calB_m,\TT_k \calB_n)// SO(n)\,\big).
\]
For other approaches of delooping the  spaces of disc embeddings and related problems,
see also~\cite{Budney12,Moriya16,Mostovoy02,Sakai14}.

\subsection{Generalized delooping results: Main Theorems~\ref{th:main} and~\ref{th:main3}}\label{ss15}

To prove~\eqref{eq:deloopingBm4} we replace the little discs operad by the Fulton-MacPherson operad $\calF_m$, which, if we ignore its arity zero operation, is cofibrant. Each component $\calF_m(k)$ is a manifold
with corners whose interior is the configuration space $C(k,\R^m)$ of $k$ distinct points in $\R^m$ 
quotiented out by translations and rescalings. On the other hand, again ignoring degeneracies, $\calF_m$ 
has a natural cofibrant replacement $\calIF_m$ as $\calF_m$-Ibimodule, whose components $\calIF_m(k)$ are also
manifolds with corners with interior $C(k,\R^m)$, see~\cite{Turchin13} and also Subsection~\ref{ss:IF}. Thus this quotient by translations and rescalings kills exactly
$(m+1)$ degrees of freedom  and suggests the delooping~\eqref{eq:deloopingBm4}.  Similarly, for~\eqref{eq:deloopingBm2},   the spaces
$C(k,\R^m)$ quotiented only by translations must admit a natural compactification $\calBF_m(k)$, so that
the sequence $\calBF_m$ forms a cofibrant replacement of $\calF_m$ as $\calF_m$-bimodule,
see Subsection~\ref{ss:BF}. Now, translations of $\R^m$ have $m$ degrees of freedom, which explains 
that we only get $m$-th delooping in~\eqref{eq:deloopingBm2}. This idea that lost degrees of freedom correspond to deloopings has been successfully put to use by the second author in~\cite{Tourtchine10} for the case 
$m=1$ of the associative operad.  Our work and the techniques that we use are inspired
 from~\cite{Tourtchine10}.  However, instead of these geometrical cofibrant replacements 
 of $\calF_m$ in the categories of bimodules and infinitesimal bimodules, we use combinatorial ones of the Boardman-Vogt type.  We sketch briefly the geometrical approach in Subsection~\ref{ss:FM5}. 
 A crucial thing is the construction of the
 delooping  map between the towers, which  appears more natural in this geometrical approach.
    The geometrical approach required more 
work, so instead we used the Boardman-Vogt type resolutions which are easier to define. In fact we conjecture that they are homeomorphic to the geometric ones, see Section~\ref{s:FM}.  As an outcome we were  able to prove a more general result,  
  which implies Main Theorem~\ref{th:delooping1}.

\begin{mainthm}\label{th:main}
Let $O$ be a coherent, $\Sigma$-cofibrant, well-pointed, and weakly doubly reduced ($O(0)=*$, $O(1)\simeq *$)  operad, and let $\eta\colon O\to M$ be an $O$-bimodules map, with $M(0)=*$. 
Then one has an equivalence of towers:
\begin{equation}\label{eq:equiv_general1}
\TT_k\Ibimod_O^h\left(\TT_kO,\TT_kM\right)\simeq \Map_*\left(\Sigma O(2), \TT_k\Bimod_O^h\left(\TT_kO,\TT_kM
\right)\right),\, k\geq 0,
\end{equation}
implying at the limit $k=\infty$:
\begin{equation}\label{eq:equiv_general2}
\Ibimod_O^h\left(O,M\right)\simeq \Map_*\left(\Sigma O(2), \Bimod_O^h\left(O,M
\right)\right).
\end{equation}
\end{mainthm}
Here $\Map_*(-,-)$ denotes the space of pointed maps, and $\TT_k\Bimod_O^h\left(\TT_kO,\TT_kM
\right)$, $\Bimod_O^h\left(O,M
\right)$ are pointed in $\TT_k\eta$ and $\eta$, respectively. The property of being {\it coherent}
 is expressed in terms of a  certain homotopy recurrence relation on the components of $O$,
see Definition~\ref{d:coherent}.  To have equivalence of towers~\eqref{eq:equiv_general1}  only up to stage~$k$,
it is enough for an operad to be {\it $k$-coherent}, meaning that the recurrence relation 
works only up to arity~$k$, see Definition~\ref{d:coherent}. 

Theorems~\ref{th:coherent},~\ref{th:FM} and Lemma~\ref{l:coherent} imply that any 
operad equivalent to the little discs operad~$\calB_d$, $0\leq d\leq \infty$, is coherent. It would be interesting to find other coherent operads, as so far these are the only 
 examples that we know. In fact we showed that for any weakly doubly reduced, $\Sigma$-cofibrant,
 and well-pointed (but not necessairily coherent) operad $O$, one has a naturally defined map from the
 right-hand side to the left-hand side of~\eqref{eq:equiv_general1}. Then we were able to determine 
 homotopy conditions of {\it coherence} on~$O$ that ensure that this map is a weak equivalence.  It would be interesting to understand
 what this map measures, what is exactly  its nature, and whether there are similar algebraic settings  
 producing analogous maps. \vspace{5pt}
 
 As an attempt to understand better this phenomenon, in the very last Subsection~\ref{ss:last} we formulate
 and sketch a proof of a more general statement -- Main Theorem~\ref{th:main3}, which shed some light
 at least on the algebraic side of this problem. The property of being coherent is really  a condition
 on an infinitesimal bimodule rather than on an operad. In case $N$ is a coherent $O$-Ibimodule
 endowed with a map $N\to O$,  Main Theorem~\ref{th:main3} describes the space
 $\Ibimod_O^h\left(N,M\right)$ similarly to the right-hand side of~\eqref{eq:equiv_general2} as a
 space of based maps from a space $C_N$ depending on $O$ and $N$ to the space $ \Bimod_O^h\left(O,M
\right)$. We do not see any immediate geometrical application of this more general result, but it is
interesting from the algebraic viewpoint. This more general result  must have much more examples since coherent infinitesimal 
bimodules must be easier to construct than coherent operads.
 
%

\subsubsection*{Acknowledgements.}
The second author is greatful to G.~Arone, M.~Kontsevich, and P.~Lambrechts, discussions with whom brought him to the problem solved in the paper
(Main Theorem~\ref{th:delooping1} and Theorem~\ref{th:delooping2}). The authors are indebted to
 B.~Fresse for answering numerous questions on the homotopy theory. The authors also thank G.~ Arone,
 T.~Banach, C.~Berger, P.~Boavida de Brito, P.~Gaucher, K.~Hess, D.~Nardin, D.~Sinha, M.~Weiss, and
  D.~Yetter for communication. Finally, the authors acknowledge
University of Paris~13 and Kansas State University for generous support that allowed a one month visit by Ducoulombier to KSU. The first author is partially supported by the grant ERC-2015-StG 678156 GRAPHCPX while the second author is partially supported by
the Simons Foundation
collaboration grant,
award ID: 519474. 

\section{Homotopy theory}\label{s:ht}
In this section we build up the necessary homotopy background. We describe the model structure
for the categories of $\Sigma$ and $\Lambda$-sequences, operads, bimodules, infinitesimal bimodules,
and the truncated versions of all these structures. In fact we  consider two model structures
for these algebraic objects: projective and Reedy. The projective one is obtained by transfering the projective model structure 
from $\Sigma$-sequences and the Reedy model structure is obtained by transfering the Reedy structure from $\Lambda$-sequences. There are a few reasons why we want to use both structures. The projective one is
more common, for example it is the one used in the manifold calculus in particular for equivalence~\eqref{eq:tower_ibimod}. Also in the projective structure the construction of the derived mapping spaces is more explicit: since all objects are fibrant we do not need to worry about the fibrant replacements, which on the contrary are difficult to make explicit in the Reedy structure.\footnote{In fact in~\cite[Section~3.1.1]{FressePC17}  B. Fresse and the authors produce an explicit Reedy fibrant replacement in the categories of  (infinitesimal) bimodules. However, if we compare explicit derived mapping spaces, in the projective structure they still appear to be smaller.} The advantage of the
Reedy structure is that  it simplifies the proof of Main Theorem~\ref{th:main}. Using it makes the proof
more elegant in the case when the operad $O$ is doubly reduced ($O(0)=O(1)=*$). Also the generalization from the strictly to weakly doubly reduced  case ($O(0)=*$, $O(1)\simeq *$) relies on certain homotopy properties of the Reedy model structure  (Theorems~\ref{th:zigzag}, \ref{th:bim_reedy_model}~(iii)) that
we were not able to check for the projective structure. Having both structures available could be handy for a possible future work. A detailed study  of the projective and Reedy model structures for (infinitesimal) bimodules is done in the  work~\cite{FressePC17} by B.~Fresse and the authors.

\subsection{$\Sigma$ and $\Lambda$ sequences}\label{ss:ht1}

Following~\cite{Fresse17}, we denote by $\Lambda$ the category whose objects are finite sets
$\underline{n}=\{1,\ldots,n\}$, $n\geq 0$, and morphisms are injective maps between them. 
The category $\Sigma\subset \Lambda$ is the subcategory of isomorphisms of $\Lambda$. By a 
$\Sigma$-sequence, respectively $\Lambda$-sequence, we understand a functor
$\Sigma^{op}\to\Topo$, respectively $\Lambda^{op}\to \Topo$. The categories of $\Sigma$ and
$\Lambda$-sequences are denoted by $\Sigma\Seq=\Topo^{\Sigma^{op}}$ and $\Lambda\Seq=
\Topo^{\Lambda^{op}}$, respectively. Let $\Sigma_{>0}\subset \Sigma$ and $\Lambda_{>0}
\subset\Lambda$ be the full subcategories of non-empty sets. We similarly define the categories
$\Sigma_{>0}\Seq$, $\Lambda_{>0}\Seq$ of $\Sigma_{>0}$ and $\Lambda_{>0}$-sequences. 
One has obvious adjunctions
\[
(-)_{>0}\colon\Sigma\Seq \leftrightarrows \Sigma_{>0}\Seq\colon(-)_+;
\]
\[
(-)_{>0}\colon\Lambda\Seq \leftrightarrows \Lambda_{>0}\Seq\colon(-)_+;
\]
where $M_{>0}$ is obtained from $M$ by forgetting its arity zero component, and $N_+$ is obtained from $N$ by defining $N_+(0)=*$ and keeping all the other components the same.\vspace{5pt} 

Similarly to Subsection~\ref{sss132} we also consider the categories $\TT_k\Lambda$,
$\TT_k\Sigma$, $\TT_k\Lambda_{>0}$, $\TT_k\Sigma_{>0}$ as full subcategories of $\Sigma$ and 
$\Lambda$ of objects with $\leq k$ elements, and we consider the diagram categories of $k$-truncated sequences $\TT_k\Lambda\Seq$,
$\TT_k\Sigma\Seq$, etc. \vspace{5pt}

The categories $\Sigma\Seq$, $\Sigma_{>0}\Seq$, $\TT_k\Sigma\Seq$, $\TT_k\Sigma_{>0}\Seq$
being diagram categories are endowed with the so called {\it projective model structure}
\cite[Section 11.6]{Hirschhorn03}, \cite[Section II.8.1]{Fresse17}. For this model structure, a map 
$M\to N$ is a weak equivalence, respectively a fibration, if it is an objectwise weak homotopy equivalence,
respectively an objectwise Serre fibration. This model structure is cofibrantly generated, and all the 
objects are fibrant in it.\vspace{5pt}

The categories $\Lambda\Seq$, $\Lambda_{>0}\Seq$, $\TT_k\Lambda\Seq$, $\TT_k\Lambda_{>0}\Seq$
are endowed with the so called {\it Reedy model structure}. The idea is that they are also
diagram categories with the source category of generalized Reedy type~\cite{Berger11},
\cite[Section II.8.3]{Fresse17}. For a (possibly truncated) $\Lambda$-sequence $X$, we denote by 
$\MM(X)$ the (truncated) $\Sigma$-sequence defined as
\begin{equation}\label{eq:match1}
\MM(X)(r)=\lim_{ {u\in Mor_\Lambda(\underline{i},\underline{r})}\atop {i<r} }X(i)
\end{equation}
and called {\it matching object} of $X$. By \cite[Proposition II.8.3.2]{Fresse17}, it can equivalently be defined as
\begin{equation}\label{eq:match2}
\MM(X)(r)=\lim_{ {u\in Mor_{\Lambda^+}(\underline{i},\underline{r})}\atop {i<r} }X(i),
\end{equation}
where $\Lambda^+$ is the subcategory of $\Lambda$ consisting of order-preserving maps.
 Thus~\eqref{eq:match2} is a limit over a so called subcubical diagram (see Subsection~\ref{s:coherent}). For $\Lambda_{>0}$ and
  $\TT_k\Lambda_{>0}$-sequences we slightly modify~\eqref{eq:match1}, \eqref{eq:match2} by excluding 
  $i=0$ in the diagram of the limit. \vspace{5pt}
  
  According to~\cite{Berger11}, \cite[Section II.8.3]{Fresse17}, the categories $\Lambda\Seq$,
   $\Lambda_{>0}\Seq$, $\TT_k\Lambda\Seq$, $\TT_k\Lambda_{>0}\Seq$ are endowed with the cofibantly generated model structure for which weak equivalences are objectwise weak homotopy equivalences,
   fibrations are morphisms $M\to N$ for which any induced map 
   $M(r)\to \MM(M)(r)\times_{\MM(N)(r)} N(r)$ is a Serre fibration in every arity where defined. 
  As shown in \cite[Theorem~II.8.3.20]{Fresse17}, a morphism of $\Lambda$, $\Lambda_{>0}$, $\TT_k\Lambda$, $\TT_k\Lambda_{>0}$ sequences
   is a cofibration if and only if it is a cofibration as a morphism of  $\Sigma$, $\Sigma_{>0}$,
   $\TT_k\Sigma$, $\TT_k\Sigma_{>0}$-sequences, respectively.

   \subsection{Operads}\label{ss:ht2}
   
  Abusing notation we also denote by $\Lambda$ the operad defined as
   \begin{equation}\label{eq:lambda_op}
   \Lambda(r)=\begin{cases} *,& r=0\text{ or }1;\\
   \emptyset,& r\geq 2.
   \end{cases}
   \end{equation}
   One can easily see that a right module over this operad is the same thing as a $\Lambda$-sequence.
   Note that a left $\Lambda$-module is simply a $\Sigma$-sequence $M$ for which $M(0)$ is pointed.\vspace{5pt}
   
   Recall that an operad $O$ is called {\it reduced} if $O(0)=*$.  Since $O$ contains the operad $\Lambda$, 
   any right $O$ module is automatically a $\Lambda$-sequence. Thus  any bimodule or infinitisemal bimodule over $O$ including $O$ itself is also a $\Lambda$-sequence. The reduced operads will also be called $\Lambda$-{\it operads}. We denote by $\Lambda\Operad$ the category of reduced operads. 
   We also use notation $\Sigma\Operad$ for the category of all operads, and $\Sigma_{>0}\Operad$
   for the category of operads whose arity zero component is empty. \vspace{5pt}
   
   One has an adjunction
   \begin{equation}\label{eq:unitar_op}
   \tau\colon\Sigma\Operad\leftrightarrows\Lambda\Operad\colon\iota,
   \end{equation}
   where $\iota$ is the obvious inclusion and $\tau$ is the {\it unitarization} functor~\cite{FresseTWsmall},
   which collapses the arity zero component to a point and takes the other components to the quotient  induced by this collapse.\vspace{5pt}
   
   The category $\Sigma\Operad$ is endowed with the so called {\it projective model structure} transferred
   from $\Sigma\Seq$ along the adjunction
 \begin{equation}\label{eq:sigma_op_adj}
 \calF^\Sigma_{Op}\colon \Sigma\Seq\leftrightarrows\Sigma\Operad\colon \calU^\Sigma,
 \end{equation}
 where $\calU^\Sigma$ is the forgetful functor, and $\calF_{Op}^\Sigma$ is the free functor,
see~\cite{Berger03}. \lq\lq{}Transferred\rq\rq{} means that a morphism $P\to Q$ of operads is a weak equivalence, respectively a fibration, if and only if it is a weak equivalence, respectively a fibration, as a
morphism of $\Sigma$-sequences. \vspace{5pt}

The category $\Lambda\Operad$ is endowed with the so called {\it Reedy model structure}
transferred from $\Lambda_{>0}\Seq$ along the adjunction
 \begin{equation}\label{eq:lambda_op_adj}
 \calF^\Lambda_{Op}\colon \Lambda_{>0}\Seq\leftrightarrows\Lambda\Operad\colon \calU^\Lambda,
 \end{equation}
see \cite[Section II.8.4]{Fresse17}.  An important property of $\calF_{Op}^\Lambda$ is that
$\calF_{Op}^\Lambda(X)_{> 0}=\calF_{Op}^\Sigma(X)$: the free $\Lambda$-operad generated by a 
$\Lambda_{>0}$-sequence $X$ in positive arities is the free operad generated by the $\Sigma_{>0}$-
sequence $X$. Using this fact it has been shown in~\cite[Theorem II.8.4.12]{Fresse17} that a morphism 
$P\to Q$ of $\Lambda$-operads is a cofibration if and only if $P_{>0}\to Q_{>0}$ is a cofibration
of operads
in the projective model structure. \vspace{5pt}

The projective and Reedy model structures on the categories of (reduced) $k$-truncated operads
$\TT_k\Sigma\Operad$ and $\TT_k\Lambda\Operad$ are defined similarly by transferring the model
structure along the corresponding adjunction
\[
 \calF^{\TT_k\Sigma}_{Op}\colon \TT_k\Sigma\Seq\leftrightarrows\TT_k\Sigma\Operad\colon 
 \calU^{\TT_k\Sigma},
 \]
 \[
 \calF^{\TT_k\Lambda}_{Op}\colon \TT_k\Lambda_{>0}\Seq\leftrightarrows\TT_k\Lambda\Operad\colon 
 \calU^{\TT_k\Lambda}.
 \]
 It has been shown in~\cite{FresseTWsmall} that the adjunction~\eqref{eq:unitar_op} and  its truncated
 versions are Quillen. Moreover, for any pair of reduced ($k$-truncated) operads $P$ and $Q$, one
 has
 \begin{equation}\label{eq:der_op_map}
 \Sigma\Operad^h(\iota P,\iota Q)\simeq \Lambda\Operad^h(P,Q),
 \end{equation}
 respectively,
 \begin{equation}\label{eq:der_tr_op_map}
 \TT_k\Sigma\Operad^h(\iota P,\iota Q)\simeq \TT_k\Lambda\Operad^h(P,Q),
 \end{equation}
  see \cite[Theorems 1\&1\rq{}]{FresseTWsmall}. The equivalences
   \eqref{eq:der_op_map}-\eqref{eq:der_tr_op_map} are immediately derived from the fact
   that for a cofibrant replacement $WP$ of $\iota P$, the natural map $\tau WP\to P$ is an equivalence.\vspace{5pt}
   
   Because of the equivalences  \eqref{eq:der_op_map}-\eqref{eq:der_tr_op_map} we do not distinguish 
   between the two derived mapping spaces and simply write $\Operad^h(-,-)$, $\TT_k\Operad^h(-,-)$.\vspace{5pt}
   
   We notice for future use that the truncation functor
   $\TT_k\colon\Lambda\Operad\to\TT_k\Lambda\Operad$ preserves cofibrations because 
   $\TT_k\colon\Sigma_{>0}\Operad\to\TT_k\Sigma_{>0}\Operad$ does so; and it preserves fibrations
   because $\TT_k\colon\Lambda_{>0}\Seq\to \TT_k\Lambda_{>0}\Seq$ preserves fibrations.\vspace{5pt}
   
   We also need the following fact about $\Lambda$ operads. Recall that an operad $O$ is {\it weakly
   doubly reduced} if it is reduced and $O(1)\simeq *$, and it is {\it doubly reduced} or
   {\it strictly doubly reduced} if $O(0)=O(1)=*$. 
   
 \begin{thm}\label{th:zigzag}
   For any weakly doubly reduced operad $O$ there exists a  zigzag of weak equivalences of reduced operads
   \[
   O\xleftarrow{\simeq}WO\xrightarrow{\simeq}W_1O,
   \]
   where both $WO$ and $W_1O$ are Reedy cofibrant and $W_1O$ is doubly reduced.
   \end{thm}
   
   \begin{proof}[Idea of the proof]
   In case $O$ is $\Sigma$-cofibrant and well-pointed (the case that we need), one can take $WO$ to be
   the Boardmann-Vogt resolution of $O_{>0}$ to which we add a point in arity zero. The operad
   $W_1O$ is obtained from $WO$ by the {\it second unitarization}: collapsing the arity one component to a point and quotienting the other components according to the equivalence relation that this collapse produces. The complete proof is given in Appendix~\ref{s:A1}.
   \end{proof}
   
%

   \subsection{Bimodules}\label{ss:ht3}
   Let $O$ be a topological operad. In this subsection we denote by $\Sigma\Bimod_O$ the category
   of all $O$ bimodules. In case $O$ is reduced ($O(0)=*$), we denote by $\Lambda\Bimod_O$
   the category of reduced $O$ bimodules, i.e. bimodules $M$ with $M(0)=*$. One has a similar
   unitarization-inclusion adjunction
   \begin{equation}\label{eq:unitar_bim}
   \tau\colon\Sigma\Bimod_O\leftrightarrows\Lambda\Bimod_O\colon\iota,
   \end{equation}
   where $\iota$ is the inclusion functor, and $\tau$ is its adjoint -- it collapses the arity zero component 
   to a point, and adjusts the other components according to the equivalence relation induced by this collapse.
We also have the free-forgetful adjunctions
\begin{equation}\label{eq:sigma_bim_adj}
 \calF^\Sigma_{B}\colon \Sigma\Seq\leftrightarrows\Sigma\Bimod_O\colon \calU^\Sigma;
 \end{equation}
    \begin{equation}\label{eq:lambda_bim_adj}
 \calF^\Lambda_{B}\colon \Lambda_{>0}\Seq\leftrightarrows\Lambda\Bimod_O\colon \calU^\Lambda.
 \end{equation}
   Explicitly, $\calF_B^\Sigma(X)=O\circ X\circ O$, and $\calF_B^\Lambda(Y)=
   O\circ_\Lambda Y_+\circ_\Lambda O$. Here $\Lambda$ is considered as an operad~\eqref{eq:lambda_op} and we use the following standard notation that will be also useful for us
    in the sequel.
   
   \begin{notat}\label{not:circO}
   For an operad $P$, its right module $M$ and its left module $N$, we define $M\circ_P N$ as the
   coequalizer of
   $M\circ P\circ N\rightrightarrows M\circ N$,
     where the upper arrow is induced by the right $P$-action on $M$ (i.e. the operation $M\circ P\to M$) and the 
     lower arrow is induced by the left $P$-action on $N$ (i.e. the operation $P\circ N\to N$).
     \end{notat}
     
  \noindent  One can easily see that as a bimodule over $O_{>0}$,
   \[
   \calF_B^\Lambda(Y)_{>0}=O_{>0}\circ Y\circ O_{>0},
   \]
   it is a free $O_{>0}$-bimodule generated by $Y$. \vspace{5pt}
   
   The three adjunctions above have also their truncated counterparts:
   \begin{equation}\label{eq:unitar_tr_bim}
   \tau\colon\TT_k\Sigma\Bimod_O\leftrightarrows\TT_k\Lambda\Bimod_O\colon\iota,
   \end{equation}
   \begin{equation}\label{eq:sigma_tr_bim_adj}
 \calF^{\TT_k\Sigma}_{B}\colon \TT_k\Sigma\Seq\leftrightarrows\TT_k\Sigma\Bimod_O\colon \calU^{\TT_k\Sigma},
 \end{equation}
 \begin{equation}\label{eq:lambda_tr_bim_adj}
 \calF^{\TT_k\Lambda}_{B}\colon \TT_k\Lambda_{>0}\Seq\leftrightarrows\TT_k\Lambda\Bimod_O\colon \calU^{\TT_k\Lambda}.
 \end{equation}
   Here $\TT_k\Sigma\Bimod_O$ (respectively, $\TT_k\Lambda\Bimod_O$) are the categories
   of (reduced) $k$-truncated bimodules over~$O$. \vspace{5pt}
   
  In case $O$ is $\Sigma$-cofibrant and well-pointed, the categories $\Sigma\Bimod_O$
  and $\TT_k\Sigma\Bimod_O$ admit a cofibrantly generated model structure transferred from
   $\Sigma\Seq$, $\TT_k\Sigma\Seq$ along the adjunctions \eqref{eq:sigma_bim_adj} 
   and~\eqref{eq:sigma_tr_bim_adj} respectively, see~\cite{Ducoulombier17} and 
   also~\cite[Section~14.3]{Fresse09}. We use the following results from \cite{FressePC17}.
   
   \begin{thm}\label{th:bim_reedy_model} {\bf{(\cite[Theorems 3.1, 3.5, 3.6]{FressePC17})}}
   \begin{itemize}
   \item[(i)] For a $\Sigma$-cofibrant, well-pointed, and reduced topological operad $O$, the categories
   $\Lambda\Bimod_O$ and $\TT_k\Lambda\Bimod_O$, $k\geq 0$, admit a cofibrantly generated model
   structure transferred from $\Lambda_{>0}\Seq$ and $\TT_k\Lambda_{>0}\Seq$, respectively,
   along the adjunctions~\eqref{eq:lambda_bim_adj} and~\eqref{eq:lambda_tr_bim_adj}, respectively. 
   (We call them {\rm Reedy model structures}.)
   
   \item[(ii)] A morphism $M\to N$ of reduced ($k$-truncated) $O$-bimodules is a cofibration 
   in this model structure if and only if $M_{>0}\to N_{>0}$ is a cofibration of $O_{>0}$-bimodules
   in the projective model structure. 
   
   \item[(iii)] In case $O$ is a Reedy cofibrant operad, the model structure on $\Lambda\Bimod_O$ is left
   proper relative to the class of $\Sigma$-cofibrant reduced bimodules.
   \end{itemize}
   \end{thm}

   \begin{thm}\label{th:bimod_maps}
   {\bf (\cite[Proposition 3.10, Theorem 3.11]{FressePC17})}
   For a reduced, well-pointed, and $\Sigma$-cofibrant operad $O$, the adjunctions~\eqref{eq:unitar_bim}
   and~\eqref{eq:unitar_tr_bim} are Quillen adjunctions. Moreover, for any pair $M,\, N\in\Lambda\Bimod_O$
   (respectively, $M,\, N\in\TT_k\Lambda\Bimod_O$, $k\geq 0$), one has
   \begin{equation}\label{eq:der_bim_map}
   \Sigma\Bimod_O^h(\iota M,\iota N)\simeq \Lambda\Bimod_O^h(M,N),
 \end{equation}
 respectively,
   \begin{equation}\label{eq:der_tr_bim_map}
  \TT_k\Sigma\Bimod_O^h(\iota M,\iota N)\simeq \TT_k\Lambda\Bimod_O^h(M,N).
 \end{equation}
   \end{thm}
    
    Because of the equivalences~\eqref{eq:der_bim_map}-\eqref{eq:der_tr_bim_map}, we 
    do not distinguish between the two versions of derived mapping spaces (projective and Reedy),
    and simply write $\Bimod_O^h(-,-)$ and $\TT_k\Bimod_O^h(-,-)$. For the proof of the main theorem
    we use the Reedy version as it makes the proof  easier. Finally, in the $\Lambda$-case one has the expected comparison  theorem.
    
    \begin{thm}\label{th:bimod_ind_restr}
    {\bf (\cite[Theorem 3.7]{FressePC17})}
    For any weak equivalence $\phi\colon O_1\xrightarrow{\simeq} O_2$ of reduced, $\Sigma$-cofibrant, and 
    well-pointed
     operads, one has  Quillen equivalences
     \begin{equation}\label{eq:bimod_ind_restr}
     \phi_B^!\colon\Lambda\Bimod_{O_1}\leftrightarrows\Lambda\Bimod_{O_2}\colon\phi_B^*,
     \end{equation}
     \begin{equation}\label{eq:tr_bimod_ind_restr}
     \phi_B^!\colon\TT_k\Lambda\Bimod_{O_1}\leftrightarrows\TT_k\Lambda\Bimod_{O_2}\colon\phi_B^*,
     \end{equation}
     where $\phi_B^*$ is the restriction functor and $\phi_B^!$ is the induction one.
     \end{thm}
%
%
    
    \subsection{Infinitesimal bimodules}\label{ss:ht4}
    For a reduced topological operad $O$ whose components $O(k)$, $k\geq 0$, are cofibrant
    spaces, we consider two model structures on $\Ibimod_O$ and $\TT_k\Ibimod_O$: projective and Reedy.
     In order to
    distinguish between the two and also to be consistent with the notation in the previous subsections,
    the category $\Ibimod_O$ (respevtively, $\TT_k\Ibimod_O$) with the projective model structure
    is denoted by $\Sigma\Ibimod_O$ (respectively, $\TT_k\Sigma\Ibimod_O$) and the same category
    with the Reedy model structure is denoted by $\Lambda\Ibimod_O$ (respectively, 
    $\TT_k\Lambda\Ibimod_O$). \vspace{5pt}
    
    The projective one is transferred from $\Sigma\Seq$ (respectively, $\TT_k\Sigma\Seq$)
    along the free-forgetful adjunctions:
    \begin{equation}\label{eq:sigma_Ibim_adj}
 \calF^\Sigma_{Ib}\colon \Sigma\Seq\leftrightarrows\Sigma\Ibimod_O\colon \calU^\Sigma;
 \end{equation}
     \begin{equation}\label{eq:sigma_tr_Ibim_adj}
 \calF^{\TT_k\Sigma}_{Ib}\colon \TT_k\Sigma\Seq\leftrightarrows\TT_k\Sigma\Ibimod_O\colon \calU^{\TT_k\Sigma},\quad k\geq 0.
 \end{equation}
    For the category of right modules over an operad with cofibrant components,
    such model structure is constructed in~\cite[Proposition~14.1.A]{Fresse09}. Our case is very similar. 
    Indeed, the structure of a right module is the same thing as a functor $F(O)\to\Topo$
    from a certain topologically enriched category $F(O)$ to 
    $\Topo$~\cite[Definition 4.1, Proposition 4.3]{Arone14}. Thus the category of right modules
    is essentially a diagram category $\Topo^{F(O)}$, where the category $F(O)$ has cofibrant all morphism
    spaces. Similarly, the category $\Ibimod_O$ can be 
    described as a diagram category $\Topo^{\widetilde{\Gamma}(O)}$, where $\widetilde{\Gamma}(O)$
    is also a topologically enriched category assigned to $O$ with the same properties 
    \cite[Definition 4.7, Proposition 4.9]{Arone14}.  Thus we can again consider the projective model structure
    for the enriched version of the diagram categories~\cite[Theorem~4.1]{Batanin17}.\vspace{5pt}
    
    The Reedy model structure is transferred from $\Lambda\Seq$ (respectively, $\TT_k\Lambda\Seq$)
    along the adjunctions:
       \begin{equation}\label{eq:lambda_Ibim_adj}
 \calF^\Lambda_{Ib}\colon \Lambda\Seq\leftrightarrows\Lambda\Ibimod_O\colon \calU^\Lambda;
 \end{equation}
     \begin{equation}\label{eq:lambda_tr_Ibim_adj}
 \calF^{\TT_k\Lambda}_{Ib}\colon \TT_k\Lambda\Seq\leftrightarrows\TT_k\Lambda\Ibimod_O\colon \calU^{\TT_k\Lambda},\quad k\geq 0.
 \end{equation}
 
 One can easily check that for a $\Lambda$-sequence $X$ (respectively, $\TT_k\Lambda$-sequence $X$),
 the sequence $\calF_{Ib}^\Lambda(X)$ (respectively, $\calF_{Ib}^{\TT_k\Lambda}(X)$) as a ($k$-truncated)
 $O_{>0}$-Ibimodule is freely generated by the ($k$-truncated) $\Sigma$-sequence~$X$. We use the following results from~\cite{FressePC17}.
 
 \begin{thm}\label{th:Ibim_reedy_model}
 {\bf (\cite[Theorems 5.1 and 5.4]{FressePC17})}
 \begin{itemize}
 \item[(i)] For any reduced topological operad $O$ with cofibrant components, the category
 of $O$-Ibimodules (respectively, $k$-truncated $O$-Ibimodules, $k\geq 0$) admits a cofibrantly generated
 model structure transferred from $\Lambda\Seq$ (respectively, $\TT_k\Lambda\Seq$) along 
    \eqref{eq:lambda_Ibim_adj} (respectively, \eqref{eq:lambda_tr_Ibim_adj}).
    
    \item[(ii)] A map $\alpha\colon M\to N$ of ($k$-truncated) $O$-Ibimodules is a cofibration in this (Reedy)
    model structure if and only if $\alpha$ is a cofibration of $O_{>0}$-Ibimodules in the projective model structure.
    \end{itemize}
    \end{thm}
    
%
    
    \begin{thm}\label{th:id_ibim_equiv}
    {\bf (\cite[Theorem 5.9]{FressePC17})}
    \begin{itemize}
    \item[(i)] For any  reduced topological operad $O$ with cofibrant components, one has 
    Quillen equivalences:
    \begin{equation}\label{eq:id_ibim_equiv}
    id\colon\Sigma\Ibimod_O\leftrightarrows\Lambda\Ibimod_O\colon id;
    \end{equation}
    \begin{equation}\label{eq:id_tr_ibim_equiv}
    id\colon\TT_k\Sigma\Ibimod_O\leftrightarrows\TT_k\Lambda\Ibimod_O\colon id.
    \end{equation}
    
    \item[(ii)] As a consequence, for any pair of ($k$-truncated) $O$-Ibimodules $M$ and $N$, one has an equivalence of mapping spaces:
     \begin{equation}\label{eq:der_ibim_map}
   \Sigma\Ibimod_O^h(M, N)\simeq \Lambda\Ibimod_O^h(M,N),
 \end{equation}
     \begin{equation}\label{eq:der_tr_ibim_map}
   \TT_k\Sigma\Ibimod_O^h(M, N)\simeq \TT_k\Lambda\Ibimod_O^h(M,N),
 \end{equation}
    respectively.
    \end{itemize}
    \end{thm}
%
    
     Because of the equivalences~\eqref {eq:der_ibim_map}-\eqref{eq:der_tr_ibim_map}, we do not
     distinguish between the two mapping spaces and simply write
     $\Ibimod_O^h(-,-)$ and $\TT_k\Ibimod_O^h(-,-)$. For the proof of the main theorem we use
     the Reedy version of the derived mapping space.     Here is the comparison theorem in the $\Lambda$ case.
    
    \begin{thm}\label{th:ibimod_ind_restr}
    {\bf (\cite[Theorem 5.6]{FressePC17})}
    For any weak equivalence $\phi\colon O_1\xrightarrow{\simeq} O_2$ of reduced 
    well-pointed
     operads with cofibrant components, one has  Quillen equivalences
     \begin{equation}\label{eq:ibimod_ind_restr}
     \phi_{Ib}^!\colon\Lambda\Ibimod_{O_1}\leftrightarrows\Lambda\Ibimod_{O_2}\colon\phi_{Ib}^*,
     \end{equation}
     \begin{equation}\label{eq:tr_ibimod_ind_restr}
     \phi_{Ib}^!\colon\TT_k\Lambda\Ibimod_{O_1}\leftrightarrows\TT_k\Lambda\Ibimod_{O_2}\colon\phi_{Ib}^*,
     \end{equation}
     where $\phi_{Ib}^*$ is the restriction functor and $\phi_{Ib}^!$ is the induction one.
     \end{thm}

\subsection{Topological spaces: products, subspaces, quotients, mapping spaces}\label{ss:top}
In~\cite{Fresse17}, B.~Fresse uses the category of simplicial sets as the underlying category of 
$\Lambda$-operads. However, his methods are easily adaptable to other model 
categories~\cite[Section~II.8.4]{Fresse17}.  Throughout the paper by the category $\Topo$ of spaces we  understand the
category of $k$-spaces (compactly generated possibly non-Hausdorff spaces). The model category structure on $\Topo$ is
defined in~\cite[Section~2.4]{Hovey99}, where it is denoted by $\textbf{K}$. This category is cartesian closed~\cite[Appendix~A]{Lewis78}, \cite{Vogt71},
which implies that for any $X,Y,Z\in\Topo$, one has homeomorphisms:
\begin{equation}\label{eq:exp1}
\Map(X\times Y,Z)\cong \Map(X,\Map(Y,Z)),
\end{equation}
which is not in general true in the category $\AllTop$ of all topological spaces. We  freely use
the homeomorphism~\eqref{eq:exp1} and its based and/or equivariant versions. We also
use the homeomorphism (and its based/equivariant versions):
\begin{equation}\label{eq:map_prod}
\Map(X,\prod_i Y_i)\cong \prod_i \Map(X,Y_i),
\end{equation}
which holds in $\AllTop$ as well, and the fact that for any equivalence relation $\sim$ on any $k$-space 
$X$, the natural inclusion
\begin{equation}\label{eq:map_equiv}
\Map(X/{\sim},Y)\subset \Map(X,Y)
\end{equation}
is a homeomorphism on its image, see Lemma~\ref{l:q_map}. The latter fact is not true in 
general in $\AllTop$, see~\cite{mathoverflow}.

As a particular case of Lemma~\ref{l:q_map} or rather of its pointed version, in $\Topo$ for any inclusion $X_0\subset X$ and any pointed space $(Y,*)$, the space of pointed maps is homeomorphic to the
space of maps of pairs:
\[
\Map_*(X/X_0,Y)\cong \Map\bigl((X,X_0),(Y,*)\bigr),
\]
which is also used at several occasions in the paper.\vspace{5pt}

We warn the reader that here and everywhere throughout the paper the products, subspaces, and more generally limits, as well as mapping spaces should be taken with the kelleyfication of their standard product, subspace, or compact-open topologies. The quotients and any colimits of $k$-spaces are always $k$-spaces
and kelleyfication is not necessary. Basic properties of the category $\Topo$ of $k$-spaces are  recalled in Appendix~\ref{s:A0}.

     \section{Cofibrant replacements}\label{s:cof}
In order to construct the towers~\eqref{eq:equiv_general1} whose stages are derived mapping spaces, we need explicit cofibrant replacements for bimodules and infinitesimal bimodules. We use the
Boardman-Vogt  resolution, which has been introduced in~\cite{Boardman68} for topological operads
and adapted to bimodules by the first author in~\cite{Ducoulombier16,Ducoulombier17}.  In 
Subsection~\ref{ss:cof1} we recall this construction and then adjust it to the $\Lambda$ setting. In the second
Subsection~\ref{ss:cof2} we construct a  Boardman-Vogt type resolution for infinitesimal bimodules.\vspace{5pt}

  In order to fix  notation, a \textit{planar tree} $T$ is a finite planar tree with one output edge on the bottom and input edges on the top. The output and input edges are considered to be half-open, i.e. connected only to one vertex in the body of the tree.   The vertex connected to the output edge, called the {\it root} of $T$, is denoted by $r$. Each edge in the tree is endowed with an orientation from top to bottom. Let $T$ be a planar tree:
\begin{itemize}[leftmargin=*]
\item[$\blacktriangleright$] The set of its vertices and the set of its edges are denoted by $V(T)$ and $E(T)$ respectively. The set of its inner edges $E^{int}(T)$ is formed by the edges connecting two vertices. Each edge or vertex is joined to the root by a unique path composed of edges.  
\item[$\blacktriangleright$] According to the orientation of the tree, if $e$ is an inner edge, then its vertex  $t(e)$ toward the root is called the \textit{target vertex} whereas the other vertex $s(e)$ is called the \textit{source vertex}. 
\item[$\blacktriangleright$] The input edges are called \textit{leaves} and they  are ordered from  left to right. Let $in(T):=\{l_{1},\ldots, l_{|T|}\}$ denote the ordered set of leaves with $|T|$ the number of leaves. 
\item[$\blacktriangleright$]  The set of incoming edges of a vertex $v$ is ordered from left to right. This set is denoted by $in(v):=\{e_{1}(v),\ldots, e_{|v|}(v)\}$ with $|v|$ the number of incoming edges. The unique output edge of $v$ is denoted by $e_{0}(v)$.
\item[$\blacktriangleright$] The number $|v|$ of incoming edges for a vertrex $v$ will be called the {\it arity}
of $v$, whereas the total number of adjacent edges at $v$ will be called the {\it valence} of $v$.
\end{itemize}

A \textit{labelled planar tree} is a pair $(T\,;\,\sigma)$ where $T$ is a planar tree and $\sigma:\{1,\ldots, |T|\}\rightarrow in(T)$ is a bijection labelling the leaves of $T$. Such an element is denoted by $T$ if there is no ambiguity about the bijection $\sigma$. We denote by \textbf{tree}$_k$ the set of 
labelled planar trees with $k$ leaves.  The bijection $\sigma$ can be interpreted as an element in the symmetric group $\Sigma_k$. By convention, the $k$-corolla is the tree with one vertex and $k$ leaves, indexed by the identity permutation. By $\textbf{tree}_k^{\geq 1}$, respectively, $\textbf{tree}_k^{\geq 2}$,
we denote the subset of $\textbf{tree}_k$ all of whose vertices have arity $\geq 1$, respectively, $\geq 2$.

\begin{figure}[!h]
\begin{center}
\includegraphics[scale=0.3]{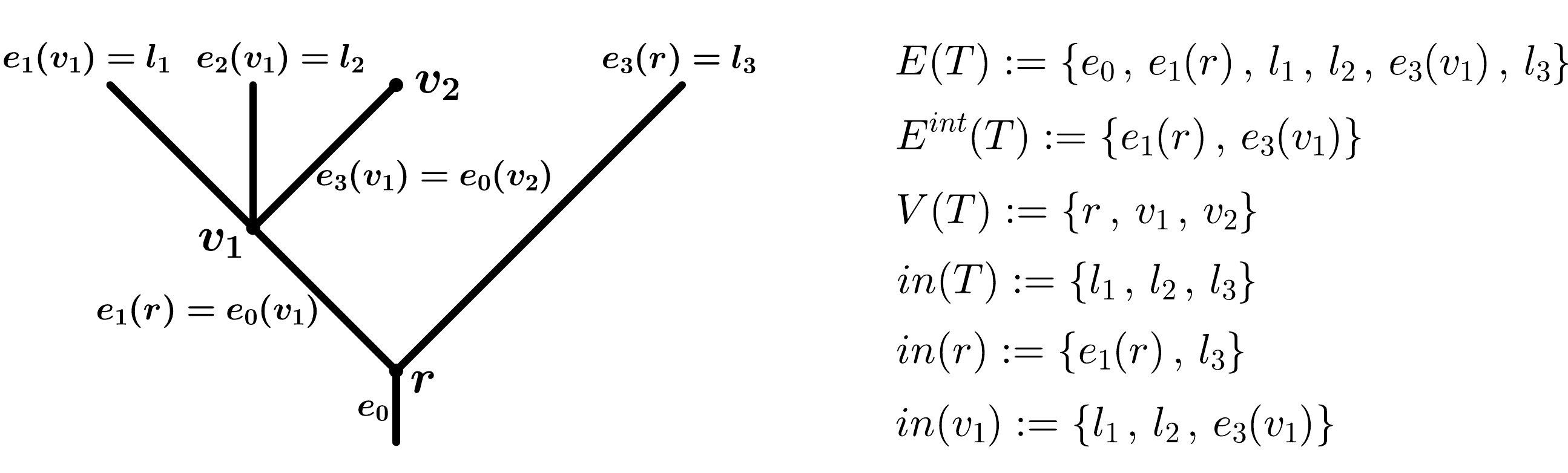}
\caption{Example of a planar tree.}\vspace{-5pt}
\end{center}
\end{figure}
  
\subsection{Boardman-Vogt resolution for bimodules}\label{ss:cof1}

Below we present a functorial cofibrant replacement for bimodules over an operad. The main difference with 
the similar resolution for operads is that vertices in the trees could now be of two types: pearls labeled by elements of the bimodule and usual internal vertices labeled by elements of the operad. Also instead of edges
being assigned numbers from $[0,1]$, such numbers are assigned to the usual internal vertices. We  need the following combinatorial set of trees.

\begin{defi} \textbf{The set of trees with section}\\
Let \textbf{stree}$_k$ be the set of pairs $(T\,;\, V^{p}(T))$ where $T\in\textbf{tree}_k$,  and $V^{p}(T)$ is a subset of $V(T)$, called the \textit{set of pearls}. Each path joining a leaf or a univalent vertex with the root passes through a unique pearl.  
 The set of pearls forms a horizontal section cutting the tree $T$ into two parts.  Elements in \textbf{stree}$_k$ are called (planar labelled) \textit{trees with section}. 
\end{defi}


\begin{center}
\begin{figure}[!h]
\begin{center}
\includegraphics[scale=0.45]{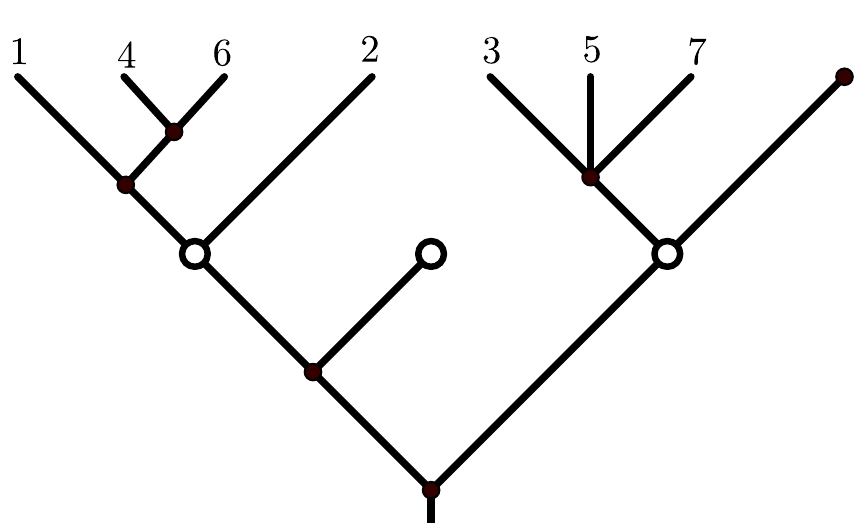}
\caption{A tree with section.}
\end{center}
\end{figure}
\end{center}

\begin{const}\label{B0}
From an $O$-bimodule $M$, we build an $O$-bimodule $\mathcal{B}^\Sigma_O(M)$.  
When $O$ is understood
we simply denote it by $\calB^\Sigma(M)$. The superscipt $\Sigma$ indicates that we
construct a cofibrant replacement of a bimodule in the projective model structure. The points of $\calB^\Sigma(M)(k)$, $k\geq 0$, are equivalence classes $[T\,;\, \{t_{v}\}\,;\, \{a_{v}\}]$, where $T\in \textbf{stree}_k$ and $\{a_{v}\}_{v\in V(T)}$ is a family of points labelling the vertices of~$T$. The pearls are labelled by points in $M$ with the corresponding arity, whereas the other vertices are labelled by points in the operad $O$ again assuming that the arity is respected.  Furthermore, $\{t_{v}\}_{v\in V(T)\setminus V^{p}(T)}$ is a family of real numbers in the interval $[0\,,\,1]$ indexing the vertices which are not pearls. If $e$ is an inner edge above the section, then $t_{s(e)}\geq t_{t(e)}$. Similarly, if $e$ is an inner edge below the section, then $t_{s(e)}\leq t_{t(e)}$. In other words, the closer  a vertex is to a pearl, the smaller is the corresponding number. The space $\mathcal{B}^\Sigma(M)(k)$ is the quotient of the sub-space of 
\begin{equation}\label{eq:union_stree}
\underset{T\in\textbf{\text{stree}}_k}{\coprod}\,\,\,\underset{v\in V^{p}(T)}{\prod}\,M(|v|)\,\,\times\underset{v\in V(T)\setminus V^{p}(T)}{\prod}\,\big[\,O(|v|)\times [0\,,\,1]\big]
\end{equation} 
determined by the restrictions on the families $\{t_{v}\}$. The equivalence relation is generated by the following conditions:
 
\begin{itemize}[topsep=0pt, leftmargin=*]
\item[$i)$] If a vertex is labelled by  $\ast_{1}\in O(1)$, then locally one has the identification
\begin{center}
\includegraphics[scale=0.25]{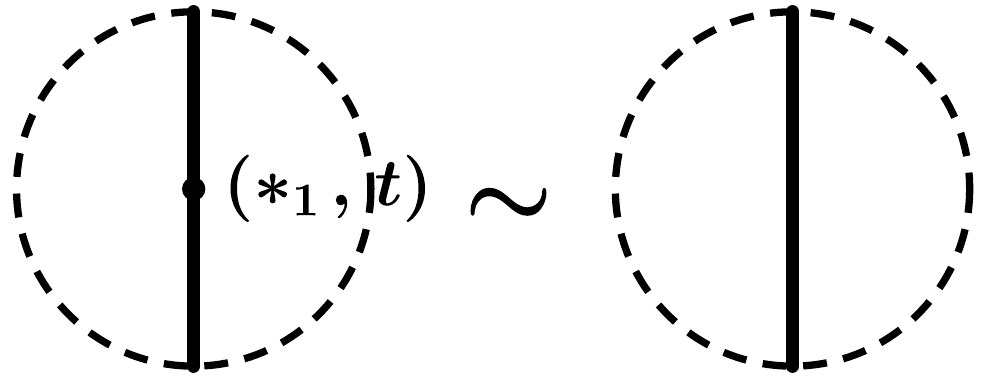}\vspace{-2pt}
\end{center}

\item[$ii)$] If a vertex is indexed by $a\cdot \sigma$, with $\sigma\in \Sigma$, then 
\begin{center}
\includegraphics[scale=0.42]{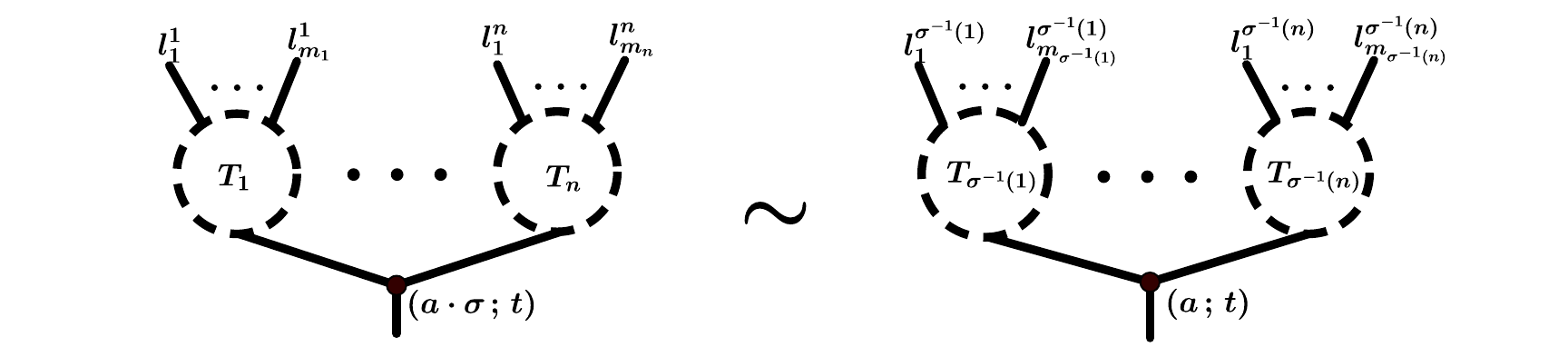}\vspace{3pt}
\end{center}


\item[$iii)$] If two consecutive vertices, connected by an edge $e$, are indexed by the same real number $t\in [0\,,\,1]$, then~$e$ is contracted using the operadic structure of~$O$. The vertex so obtained is indexed by the real number~$t$.
\begin{center}
\includegraphics[scale=0.52]{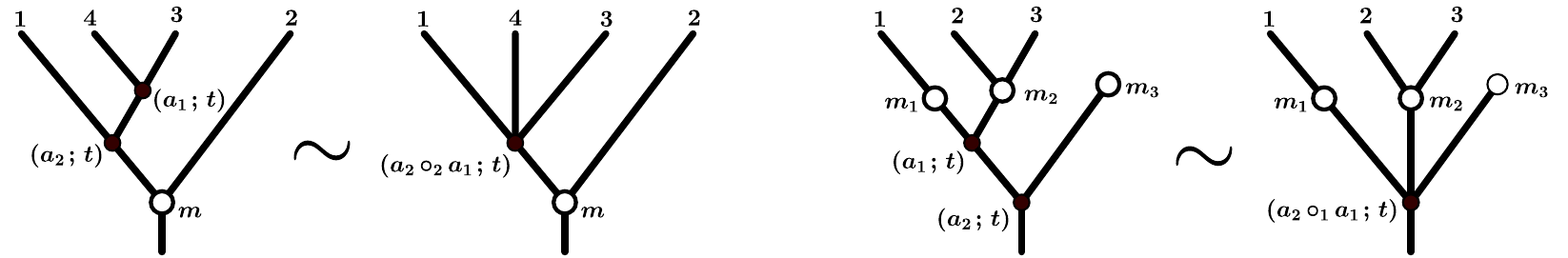}\vspace{3pt}
\end{center}

\item[$iv)$]  If a vertex above the section is indexed by $0$, then its output edge is contracted by using the right module structures. Similarly, if a vertex below the section is indexed by $0$, then all its incoming edges are contracted by using the left module structure. In both cases the new vertex becomes a pearl.
\begin{center}
\includegraphics[scale=0.52]{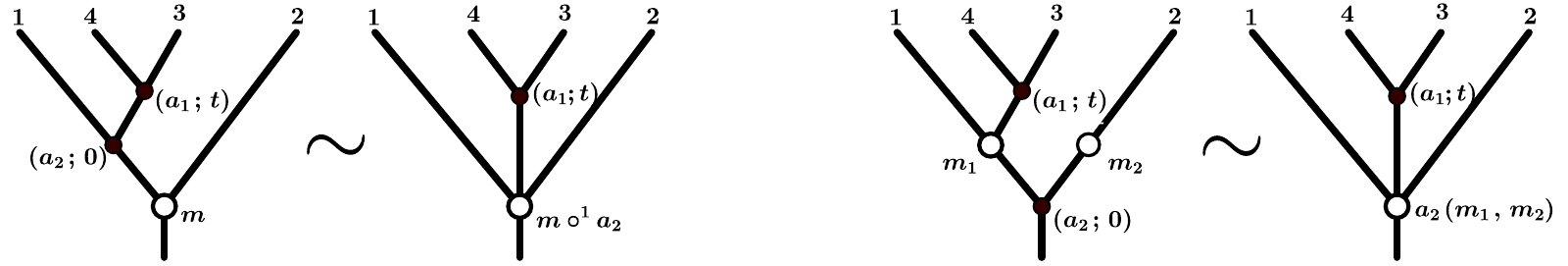}\vspace{3pt}
\end{center}

\item[$v)$] If a univalent pearl is indexed by a point of the form $\gamma_{0}(x)$, with $x\in O(0)$, then we contract its output edge by using the operadic structure of $P$. In particular, if all the pearls connected to a vertex $v$ are univalent and of the form $\gamma_{0}(x_{v})$, then the vertex is identified to the pearled corolla with no input.
\begin{center}
\includegraphics[scale=0.37]{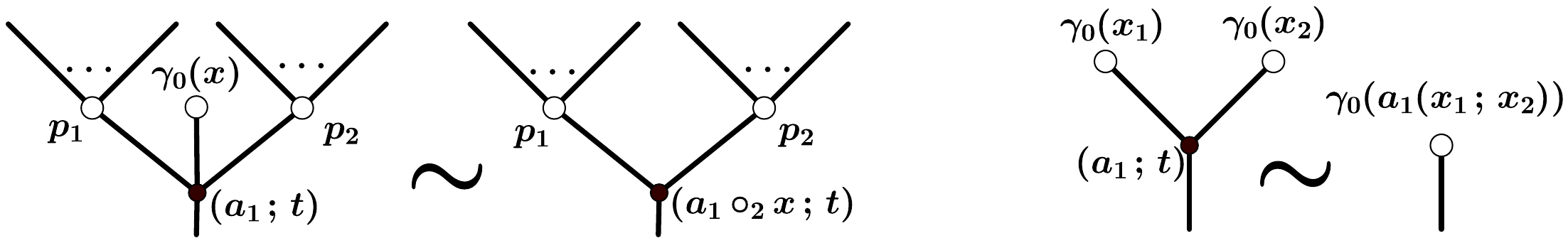}\vspace{3pt}
\end{center}

\end{itemize}

\newpage

Let us describe the $O$-bimodule structure. Let $a\in O(n)$ and $[T\,;\,\{a_{v}\}\,;\,\{t_{v}\}]$ be a point in $\mathcal{B}^\Sigma(M)(m)$. The composition $[T\,;\,\{a_{v}\}\,;\,\{t_{v}\}]\circ^{i}a$ consists in grafting the $n$-corolla labelled by $a$ to the $i$-th incoming edge of $T$ and indexing the new vertex by $1$. Similarly, let $[T^{i}\,;\,\{a^{i}_{v}\}\,;\,\{t^{i}_{v}\}]$ be a family of points in the spaces 
$\mathcal{B}^\Sigma(M)(m_{i})$. The left module structure over $O$ is defined as follows: each tree of the family is grafted to a leaf of the $n$-corolla labelled by $a$ from left to right. The new vertex, coming from the $n$-corolla, is indexed by $1$.

\begin{figure}[!h]
\begin{center}
\includegraphics[scale=0.45]{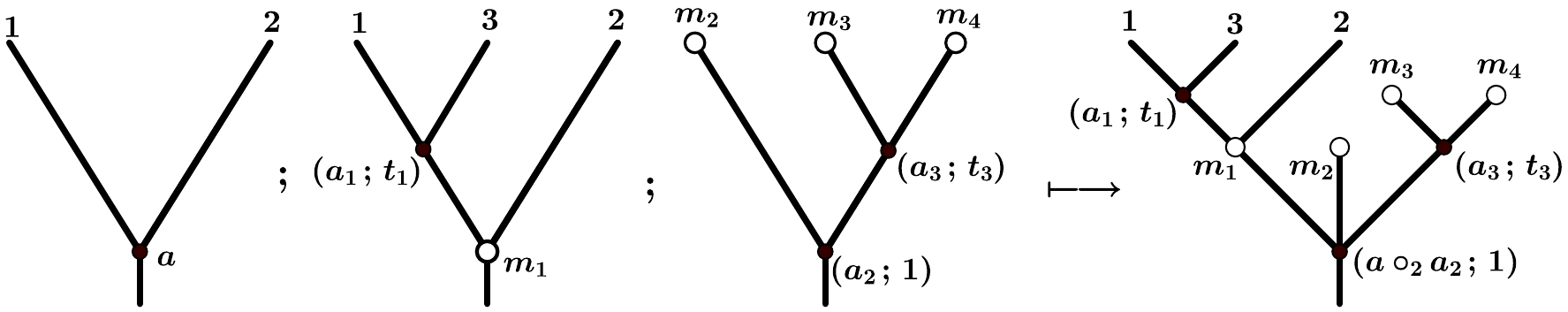}
\caption{Illustration of the left module structure.}
\end{center}
\end{figure}

One has an obvious inclusion of $\Sigma$-sequences $\iota \colon M\to \calB^\Sigma(M)$, where each element
$m\in M(k)$ is sent to a $k$-corolla labelled by $m$, whose only vertex is a pearl.  
%
 Furthermore, the following map:
\begin{equation}\label{D3}
\mu:\mathcal{B}^\Sigma(M)\rightarrow M\,\,;\,\, [T\,;\,\{t_{v}\}\,;\,\{a_{v}\}]\mapsto [T\,;\,\{0_{v}\}\,;\, \{a_{v}\}],
\end{equation}
is defined by sending the real numbers indexing the vertices other than the pearls to $0$. The element so obtained is identified to the pearled corolla labelled by a point in $M$.  It is easy to see that $\mu$ is an $O$-bimodule map.
\end{const}

\begin{rmk}\label{rm:BWbim_coend}
Each component $\mathcal{B}^\Sigma(M)(k)$ can also be defined as a coend over a category whose objects 
are trees slightly more general than those in $\textbf{stree}_k$. Namely, one should also allow
trees with univalent vertices below the section. The morphisms in this category are generated
by section preserving isomorphisms of trees, additions of a bivalent vertex in the middle of any edge,
edge contractions between two non-pearl vertices or between  a vertex above the section and a pearl, and finally
by contractions of all the edges in between a vertex (possibly of arity zero) below the section  and pearls above it.  In the latter case this vertex becomes a pearl. Compare with Remark~\ref{rm:BWbim_alt}.
\end{rmk}

In order to get resolutions for truncated bimodules, one considers a filtration in $\mathcal{B}^\Sigma(M)$ according to  the number of {\it geometrical inputs} which is the number of leaves plus the number of univalent vertices above the section. A point in $\mathcal{B}^\Sigma(M)$ is said to be \textit{prime} if the real numbers indexing the vertices are strictly smaller than $1$. Otherwise, a point is said to be \textit{composite}. Such a point can be decomposed into \textit{prime components} as shown in Figure~\ref{B1}. More precisely, the prime components 
 are obtained by removing the vertices indexed by the real number $1$. 
\begin{figure}[!h]
\begin{center}
\includegraphics[scale=0.4]{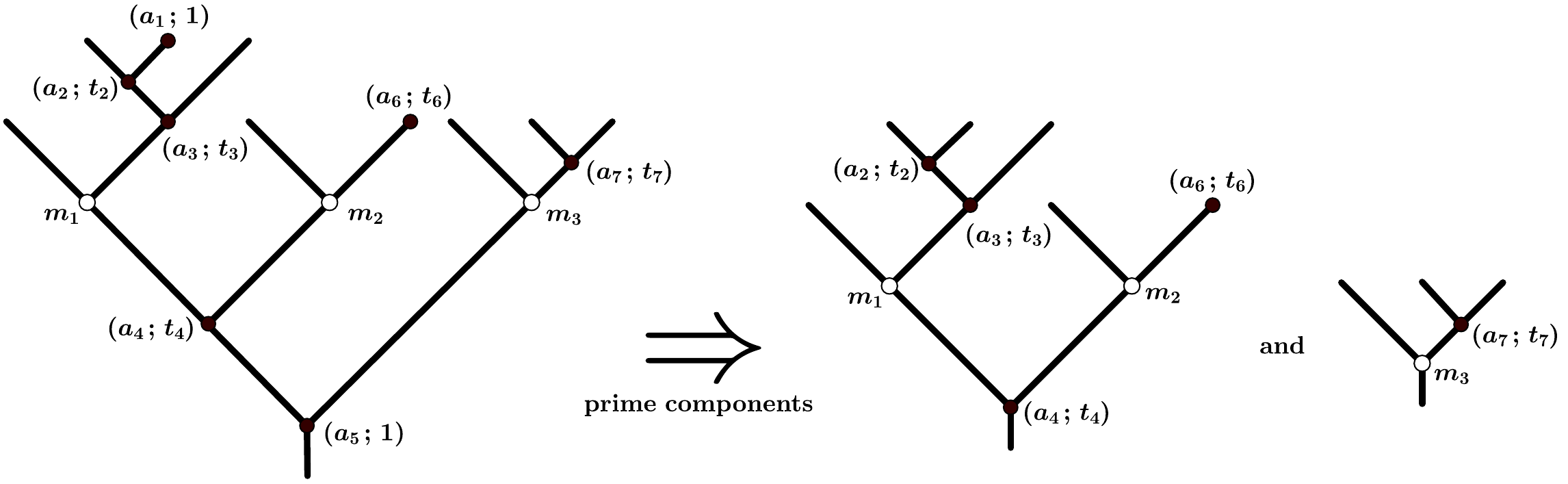}
\caption{A composite point and its prime components.}\label{B1}\vspace{-35pt}
\end{center}
\end{figure}

\newpage

 A prime point is in the $k$-th filtration term $\mathcal{B}_{k}^\Sigma(M)$ if the number of its geometrical inputs is at most~$k$. Similarly, a composite point is in the $k$-th filtration term if its all prime components are in $\mathcal{B}_{k}^\Sigma(M)$. For instance, the composite point in Figure~\ref{B1} is in the filtration term $\mathcal{B}_{6}^\Sigma(M)$. For each $k$, $\mathcal{B}_{k}^\Sigma(M)$ is a bimodule over $O$ and one has the following filtration of $\mathcal{B}^\Sigma(M)$:
\begin{equation}\label{B3}
\xymatrix{
 \mathcal{B}_{0}^\Sigma(M) \ar[r] & \mathcal{B}_{1}^\Sigma(M)\ar[r] & \cdots \ar[r] & \mathcal{B}_{k-1}^\Sigma(M) \ar[r] & \mathcal{B}_{k}^\Sigma(M) \ar[r] & \cdots \ar[r] &  \mathcal{B}^\Sigma(M).
}
\end{equation}

The term $\mathcal{B}_0^\Sigma(M)$ is non-empty only in arity zero, where it is the pushout in the
category of $O$-algebras of the free $O$-algebra generated by $M(0)$ with $O(0)$  along the free $O$-algebra generated by $O(0)$:
$$
\mathcal{B}_0^\Sigma(M)=O(0)\coprod_{\calF_B^\Sigma(O(0))} \calF_B^\Sigma(M(0)).
$$

\begin{thm}[Theorem 2.12 in \cite{Ducoulombier16}]\label{th:BV_proj_bimod}
Assume that  $O$  is a well-pointed $\Sigma$-cofibrant topological operad, and $M$ is a $\Sigma$-cofibrant 
$O$-bimodule for which the arity zero left action map $\gamma_0\colon O(0)\to M(0)$ is a cofibration. Then, the objects $\mathcal{B}^\Sigma(M)$ and $\TT_{k}\mathcal{B}_{k}^\Sigma(M)$ are cofibrant replacements of $M$ and $\TT_{k}M$ in the categories $\Sigma\Bimod_{O}$ and $\TT_{k}\Sigma\Bimod_{O}$, respectively. In particular the maps $\mu$ and $\TT_{k}\mu|_{\TT_{k}\mathcal{B}_{k}^\Sigma(M)}$ are weak equivalences.
\end{thm}

\begin{rmk}\label{rm:BWbim_alt}
There is a slightly different construction of the Boardman-Vogt replacement for bimodules in which we allow
univalent
vertices below the section. Its advantage is that it works even when $\gamma_0\colon O(0)\to M(0)$ 
is not a cofibration. The disadvantage is that this replacement combinatorially is much more complicated. For example, its filtration zero term $\mathcal{B}_{0}^\Sigma(M)$ consists
of trees with no vertices above the section. In the case $M=O$ and $O$ is doubly reduced to which we want to apply this construction, even to obtain  $\mathcal{B}_{0}^\Sigma(M)$, we need an infinite sequence of cell attachments, see construction below, which makes it
too clumsy   to apply for the proof of
our Main Theorem~\ref{th:main}. 
\end{rmk}
%

%
%

The idea of the proof of this theorem is that each map $\calB_{k-1}^\Sigma(M)\to \calB_k^\Sigma(M)$,
$k\geq 0$
is a cofibration being obtained as a possibly infinite sequence of so called cell attachments -- a procedure 
that we explain below. Here by $\calB_{-1}^\Sigma(M)$ we mean the free bimodule $\calF_B^\Sigma(\emptyset)$
generated by the empty sequence. It is $O(0)$ in arity zero and empty otherwise. \vspace{5pt}

The inclusion of $O$ bimodules $M\to M\rq{}$ is called a cell attachment if there is a cofibration
of $\Sigma$-sequences $\partial X\to X$ (in practice $\partial X$ is also cofibrant) and a morphism
of $\Sigma$-sequences $\partial X\to M$ inducing an $O$ bimodules map $\calF_B^\Sigma(\partial X)
\to M$, so that $M\rq{}$ is obtained as the following pushout in the category of $O$-bimodules:
\[
M\rq{}=M\coprod_{\calF_B^\Sigma(\partial X)} \calF_B^\Sigma(X).
\]

We
warn the reader that here and almost everywhere in the text $\partial$ does not mean boundary in the usual sense, but just a subset. However, because of the notation, we call it {\it boundary} sometimes and the 
complement $X\setminus\partial X$ is called {\it interior} of the cell~$X$.
The main reason for this notation is that because for the main example of the Fulton-MacPherson operad,
in all the constructions, the corresponding subsets are expected to be actual boundaries of manifolds.
It is not proved, but conjectured.\vspace{5pt}

The prime elements in $\calB^\Sigma(M)$ span free cells that are attached in certain order.
Note that the prime components for the $k$-th filtration term always have arity $\leq k$. This means that
$\calB_k^\Sigma(M)$ is obtained by a sequence of \lq\lq{}cell-attachments\rq\rq{} where we always attach
cells of arity $\leq k$.
Consequently, one has the following equivalences:
\begin{equation}\label{eq:der_bimod_vs_tr1}
\TT_{k}\Bimod_{O}^{h}(\TT_{k}M\,;\, \TT_{k}M')\simeq \TT_{k}\Bimod_{O}(\TT_{k}\mathcal{B}_{k}^\Sigma(M)\,;\, \TT_{k}M') \cong \Bimod_{O}(\mathcal{B}_{k}^\Sigma(M)\,;\, M'),
\end{equation}
where the last one is a homeomorphism. 

\begin{expl}\label{ex:BprojO}
Let us desrcibe a bit more explicitly the above cell attachments in the special case $O(0)=O(1)=*$ and
$M=O$. Since $O(1)=*$, thanks to relation ($i$), trees with bivalent non-pearl vertices do not appear.   First we get in this case
\[
\calB_{-1}^\Sigma(O)=\calB_{0}^\Sigma(O)=\calF_B^\Sigma(\emptyset),
\]
which means that there is no prime elements with zero geometrical inputs. To go from $\calB_{0}^\Sigma(O)$ to $\calB_1^\Sigma(O)$ we need to attach 2 cells, the first one corresponds to
the pearled 1-corolla, attached in arity~1, and the second one corresponds to the tree \includegraphics[scale=0.08]{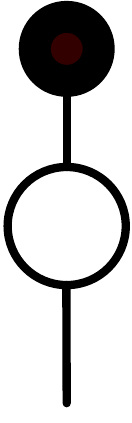}, attached in arity zero. In general the cofibration 
$\calB_{k-1}^\Sigma(O)\to \calB_k^\Sigma(O)$ is a composition of $k+1$ cell attachments, where at the first step we attach a cell whose interior consists of prime elements  labelled by trees of arity $k$ with no univalent vertices, then we attach
a cell whose interior consists of prime elements labelled by trees of arity $k-1$ with exactly one univalent vertex, and so on, at the
$(i+1)$-th step we attach a cell whose interior consists of  prime elements labelled by trees of arity $(k-i)$ with exactly $i$ univalent vertices, etc. 
Note that each time the generating sequence is concentrated in one arity. 
\end{expl}

\subsubsection{Boardman-Vogt resolution in the $\Lambda$ setting.}\label{sss:cof11}
Now we adjust the above construction to produce Reedy cofibrant replacements of bimodules.
We assume that our operad $O$ and a bimodule $M$ over it are reduced: $O(0)=M(0)=*$. As  a $\Sigma$-sequence, we set
\[
\calB_O^\Lambda(M):=\calB_{O_{>0}}^\Sigma(M_{>0})_+.
\]
When $O$ is understood, we also write $\calB^\Lambda(M)$. The superscript $\Lambda$
is to emphasize that we get a cofibrant replacement in the Reedy model structure.  The map $\mu\colon \calB^\Lambda(M)\to M$ is extended to  arity zero in the obvious way.\vspace{5pt}

 The arity zero left action $\gamma_0\colon O(0)\to \calB^\Lambda(M)(0)$ sends a point to a point
$*_0\mapsto *_0^\calB$. The
positive arity left action on positive arity components is defined in the same way as for
 $\calB_{O_{>0}}^\Sigma(M_{>0})$. For the positive arity left action on the components of $\calB^\Lambda(M)$, some of which are of arity zero, one uses the associativity of the left $O$-action to get
 \[
 x(*_0^\calB,y_1,\ldots,y_{k-1})=(x\circ_1 *_0)(y_1,\ldots,y_{k-1}),
 \]
 where $x\in O(k)$, $y_1,\ldots,y_{k-1}\in \calB^\Lambda(M)$.\vspace{5pt}

  The right action by the positive arity components is defined as it is on $\calB_{O_{>0}}^\Sigma(M_{>0})$. The right action by $*_0\in O(0)$ is defined in the obvious way as the right
 action by $*_0$ on $a$ in the vertex $(a,t)$ connected to the leaf labelled by $i$ as illustrated in the Figure \ref{A1}.
 \begin{figure}[!h]
 \begin{center}
 \includegraphics[scale=0.25]{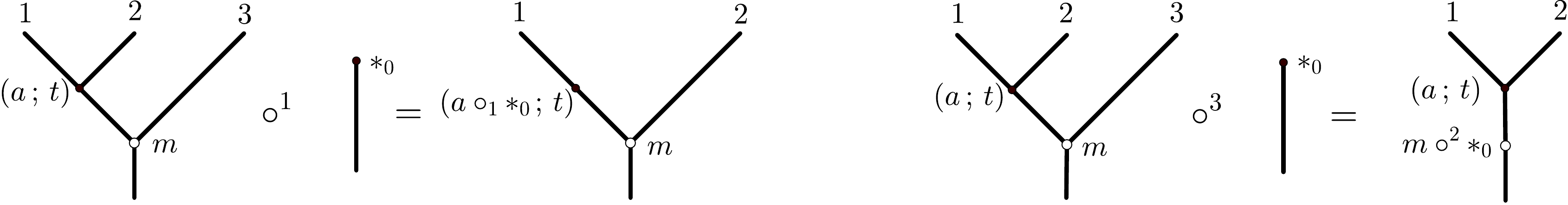}
 \caption{Illustration of the right action by $\ast_{0}$.}\label{A1}\vspace{-5pt}
 \end{center}
 \end{figure}

Note that since $O_{>0}$ and $M_{>0}$ have  empty arity zero components, in the union~\eqref{eq:union_stree} 
we can consider only trees whose all vertices have arity $\geq 1$. We denote this set by 
$\textbf{stree}^{\geq 1}_k$.
 Moreover, if $O$ is doubly reduced --
the case that we mostly use for the proof of  Main Theorem~\ref{th:main},
we can in addition assume that all non-pearl vertices are of arity~$\geq 2$ (thanks to relation~(i)
of Construction~\ref{B0}). The corresponding set is denoted by $\textbf{stree}^{\geq 2}_k$.\vspace{5pt}

The filtration~\eqref{B3} in $\calB_{O_{>0}}^\Sigma(M_{>0})$ induces filtration
\begin{equation}\label{eq:filtr_reedy_bimod}
\xymatrix{
 \mathcal{B}_{0}^\Lambda(M) \ar[r] & \mathcal{B}_{1}^\Lambda(M)\ar[r] & \cdots \ar[r] & \mathcal{B}_{k-1}^\Lambda(M) \ar[r] & \mathcal{B}_{k}^\Lambda(M) \ar[r] & \cdots \ar[r] &  \mathcal{B}^\Lambda(M),
}
\end{equation}
where $\mathcal{B}_{k}^\Lambda(M)$ is the $k$-th filtration term of $\calB_{O_{>0}}^\Sigma(M_{>0})$
plus $*_0^\calB$. In particular, the zeroth term $ \mathcal{B}_{0}^\Lambda(M)$ has only a point in arity zero and 
is empty in all the other arities.  Note also that when we go from $(k-1)$-th to $k$-th filtration term we only attach cells of arity exactly $k$
as the number of geometrical inputs in prime components is equal to their arity by the lack of arity zero vertices. (In case $O$ is doubly reduced there is exactly one cell attached at this step.)  Such attachments 
affect only components of arity $\geq k$. As a consequence,
 $\TT_k \mathcal{B}_{k}^\Lambda(M)=\TT_k\mathcal{B}^\Lambda(M)$.

\begin{pro}\label{p:BV_reedy_bimod}
Assume that   $O$ is a reduced well-pointed $\Sigma$-cofibrant topological operad, and $M$ is a reduced
$\Sigma$-cofibrant $O$-bimodule. Then, the objects $\mathcal{B}^\Lambda(M)$ and $\TT_{k}\mathcal{B}^\Lambda(M)$ are cofibrant
 replacements of $M$ and $\TT_{k}M$ in the categories $\Lambda\Bimod_{O}$ and $\TT_{k}\Lambda\Bimod_{O}$, respectively. In particular the maps $\mu$ and $\TT_{k}\mu$ are weak equivalences.
\end{pro}

\begin{proof}
Follows from Theorem~\ref{th:BV_proj_bimod} applied to $O_{>0}$ and $M_{>0}$ and 
Theorem~\ref{th:bim_reedy_model}~(ii).
\end{proof}

For $M'$ reduced and Reedy fibrant $O$-bimodule, one similarly has
\begin{equation}\label{eq:der_bimod_vs_tr2}
\TT_{k}\Bimod_{O}^{h}(\TT_{k}M\,;\, \TT_{k}M')\simeq \TT_{k}\Bimod_{O}(\TT_{k}\mathcal{B}^\Lambda(M)\,;\, \TT_{k}M') \cong \Bimod_{O}(\mathcal{B}_{k}^\Lambda(M)\,;\, M'),
\end{equation}
where again the last equivalence is a homeomorphism.

\begin{expl}\label{ex:1compl}
In case $O$ is doubly reduced,  the  only trees with section that  contribute to $\calB^\Lambda(M)$ are those that have pearls of arity $\geq 1$, and 
all the other vertices of arity $\geq 2$. In particular, for $k=1$ there is only one such tree 
\includegraphics[scale=0.07]{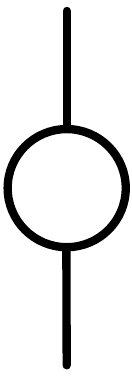}.
And one has $\calB^\Lambda(M)(1)=M(1)$. For $k=2$, here is the space of trees:
$$
\includegraphics[scale=0.25]{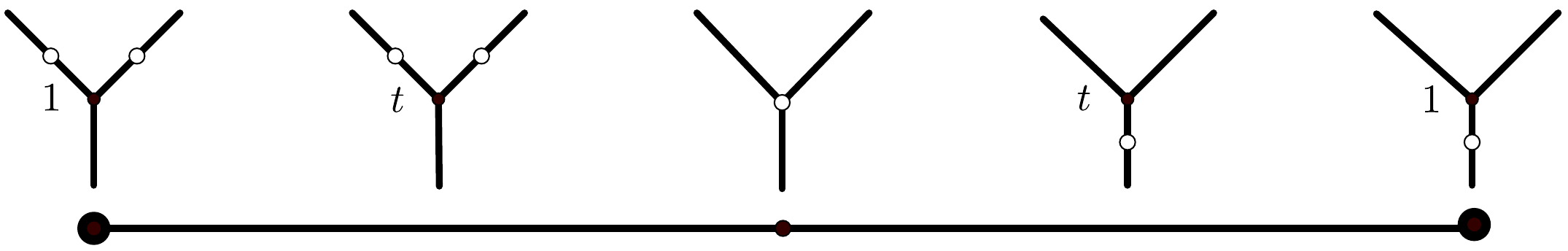}
$$
The corresponding space $\calB^\Lambda(M)(2)$ is the union of two mapping cylinders corresponding to the maps 
\[
\gamma_l\colon O(2)\times M(1)\times M(1)\to M(2), \quad
\gamma_r\colon M(1)\times O(2)\to M(2),
\]
which are glued together along the target space $M(2)$. For the  case $M=O$,  we get $\calB^\Lambda(O)(1)=*$, $\calB^\Lambda(O)(2)=O(2)\times [-1,1]$ (as both maps above are essentially identity maps).
\end{expl}

\subsection{Boardman-Vogt resolution for infinitesimal bimodules}\label{ss:cof2}

The constructions of the Boardman-Vogt resolutions for bimodules and infinitesimal bimodules are similar to each other. Consequently, we refer the reader to \cite[Section 2.3]{Ducoulombier16} for a proof of 
Theorem~\ref{C7} (see also~\cite{Berger06} for the operadic case and Appendix~\ref{s:A1} that
gives a proof of a similar statement). First, we introduce the set of trees used to define the Boardman-Vogt resolution.

\begin{figure}[!h]
\begin{center}
\includegraphics[scale=0.45]{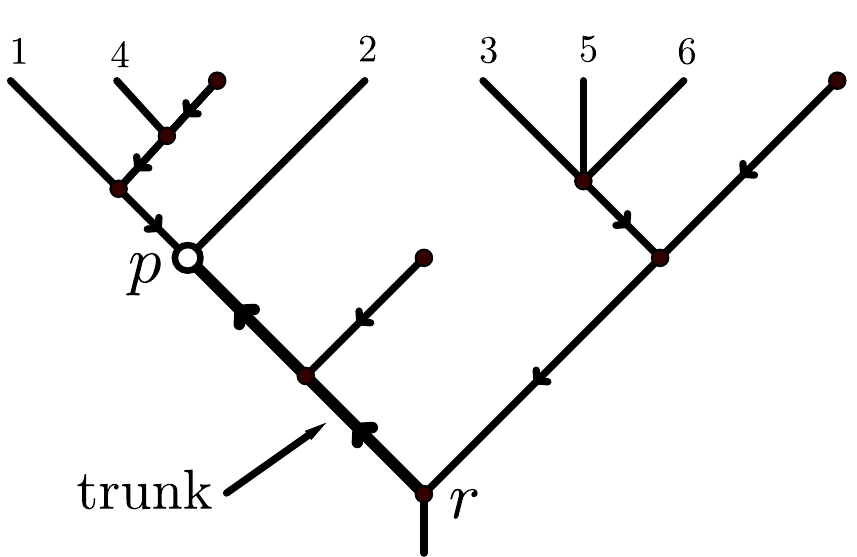}\vspace{-5pt}
\caption{Example of a pearled tree with its orientation toward the pearl.}\label{fig:ptree}\vspace{-15pt}
\end{center}
\end{figure}

\begin{defi}\textbf{The set of pearled trees}\\
Let \textbf{ptree}$_k$ be the set of pairs $(T\,;\,p)$ where $T\in\textbf{tree}_k$  and $p\in V(T)$ is
a distinguished vertex called the {\it pearl}. The path between the root $r$ and the pearl $p$ will be called the {\rm trunk} of $(T\,;\,p)$. Elements in \textbf{ptree}$_k$ are called \textit{pearled trees}.  
  Such a tree is endowed with another orientation toward the pearl. By abuse of notation, if $e$ is an inner edge of a pearled tree, then $s(e)$ and $t(e)$ denote respectively the source and the target vertices according to the orientation toward the pearl. If there is no ambiguity about the pearl, we denote by $T$ the pearled tree $(T\,;\,p)$.\vspace{-5pt}
\end{defi}

\begin{const}\label{C8}
From an $O$-Ibimodule $M$, we build an $O$-Ibimodule $\mathcal{I}b_O^\Sigma(M)$. When $O$
is understood we  also write $\mathcal{I}b^\Sigma(M)$. The  points of  $\mathcal{I}b^\Sigma(M)(k)$ are equivalence classes $[T\,;\, \{t_{v}\}\,;\, \{a_{v}\}]$ with $T\in \textbf{ptree}_k$ and $\{a_{v}\}_{v\in V(T)}$ is a family of points labelling the vertices of~$T$. The pearl is labelled by a point in $M$ whereas the other vertices are labelled by points in the operad $O$ always assuming that the arity is respected.  Furthermore, $\{t_{v}\}_{v\in V(T)\setminus \{p\}}$ is a family of real numbers in the interval $[0\,,\,1]$ indexing the vertices other than the pearl. If $e$ is an inner edge, then $t_{s(e)}\geq t_{t(e)}$ according to the orientation toward the pearl.  In other words, closer to the pearl is a vertex, smaller is the corresponding real number. 
The space $\mathcal{I}b^\Sigma (M)(k)$ is  the quotient of the sub-space of
$$
\underset{T\in\textbf{\text{ptree}}_k}{\coprod}\,M(|p|)\,\,\times\underset{v\in V(T)\setminus \{p\}}{\prod}\,\big[\,O(|v|)\times [0\,,\,1]\big] 
$$ 

\noindent determined by the restrictions on the families of real numbers $\{t_{v}\}$. The equivalence relation is generated by the axioms $(i)$ and $(ii)$ of Construction~\ref{B0} and also the following conditions:

\begin{itemize}[itemsep=-0pt, topsep=2pt, leftmargin=15pt]
\item[$iii)$] If two consecutive vertices, connected by an edge $e$, are indexed by the same real number $t\in [0\,,\,1]$, then $e$ is contracted by using the operadic structure. The vertex so obtained is indexed by the real number $t$.

\begin{figure}[!h]
\begin{center}
\includegraphics[scale=0.5]{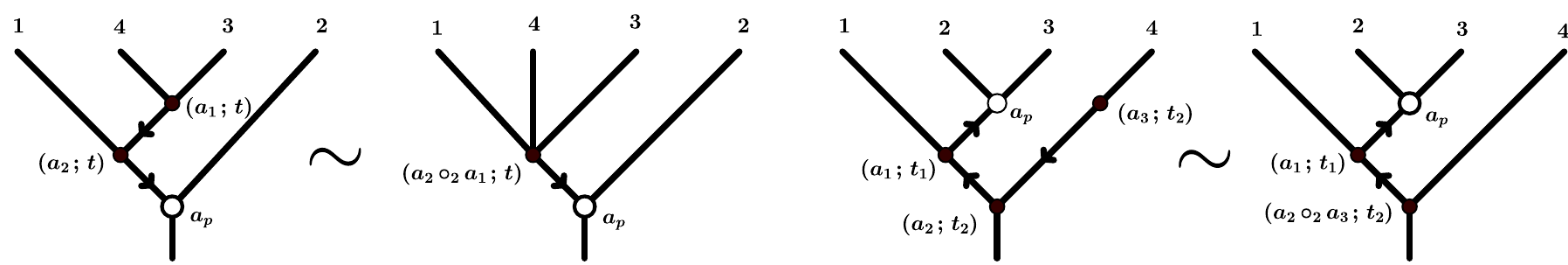}
\caption{Illustration of the relation $(iii)$.}
\end{center}
\end{figure}


\item[$iv)$]  If a vertex is indexed by $0$, then its output edge (according to the orientation toward the pearl) is contracted by using the infinitesimal bimodule structure. The obtained vertex becomes a pearl.
\begin{figure}[!h]
\begin{center}
\includegraphics[scale=0.5]{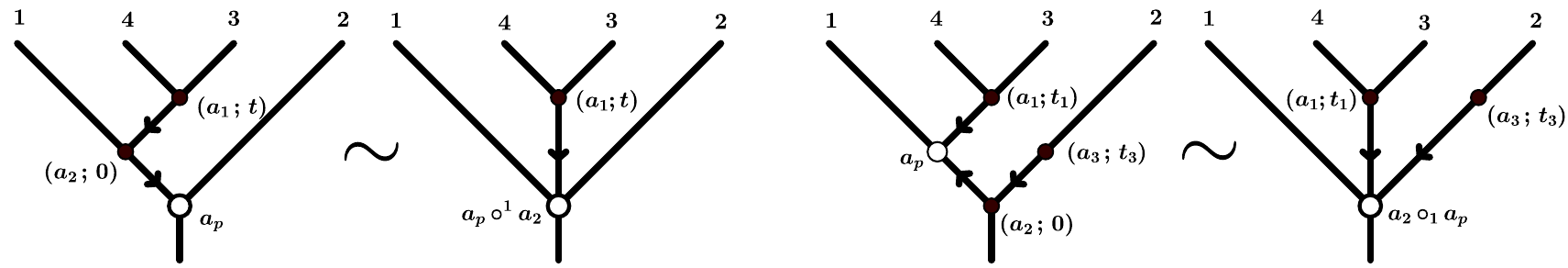}
\caption{Examples of the relation $(iv)$.}\vspace{-5pt}
\end{center}
\end{figure}
\end{itemize}

 Let us describe the $O$-Ibimodule structure. Let $a\in O(n)$ and $[T\,;\,\{a_{v}\}\,;\,\{t_{v}\}]$ be a point in $\mathcal{I}b^\Sigma(M)(m)$. The composition $[T\,;\,\{a_{v}\}\,;\,\{t_{v}\}]\circ^{i}a$ consists in grafting the $n$-corolla labelled by $a$ to the $i$-th incoming edge of $T$ and indexing the new vertex by $1$. Similarly, the composition $a\circ_{i}[T\,;\,\{a_{v}\}\,;\,\{t_{v}\}]$ consists in grafting the pearled tree $T$ to the $i$-th incoming edge of the $n$-corolla labelled by $a$ and indexing the new vertex by $1$.

One has an obvious inclusion of $\Sigma$-sequences $\iota'\colon M\to \mathcal{I}b^\Sigma(M)$ sending an element
$m\in M(k)$ to the pearled $k$-corolla labelled by $m$. 
  One also has a map
\begin{equation}\label{G8}
\mu':\mathcal{I}b^\Sigma(M)\rightarrow M\,\,;\,\, [T\,;\,\{t_{v}\}\,;\,\{a_{v}\}]\mapsto [T\,;\,\{0_{v}\}\,;\, \{a_{v}\}],
\end{equation}
 defined by sending the real numbers indexing the vertices other than the pearl to $0$. The element so obtained is identified to the pearled corolla labelled by a point in $M$. By construction, $\mu'$ is an $O$- Ibimodule map.
\end{const}

\begin{rmk}\label{rm:BWibim_coend}
Each component $\mathcal{I}b^\Sigma(M)(k)$ can also be defined as a coend over a category whose objects 
set is $\textbf{ptree}_k$.  The morphisms in this category are generated
by pearl preserving isomorphisms of trees, additions of a bivalent vertex in the middle of any edge,
and edge contractions.
\end{rmk}

In order to get resolutions for truncated infinitesimal bimodules, one considers a filtration in $\mathcal{I}b^\Sigma(M)$ according to the number of geometrical inputs, which is the number of leaves plus the number of univalent vertices other than the pearl. A point in $\mathcal{I}b^\Sigma(M)$ is said to be prime if the real numbers labelling its vertices are strictly smaller than $1$. Besides, a point is said to be composite if one of its vertex is labelled by $1$. Such a point can be associated to a prime component. More precisely, the prime component of a composite point is obtained by removing all the vertices indexed by $1$ as illustrated in Figure~\ref{C9}.  

\begin{figure}[!h]
\begin{center}
\includegraphics[scale=0.4]{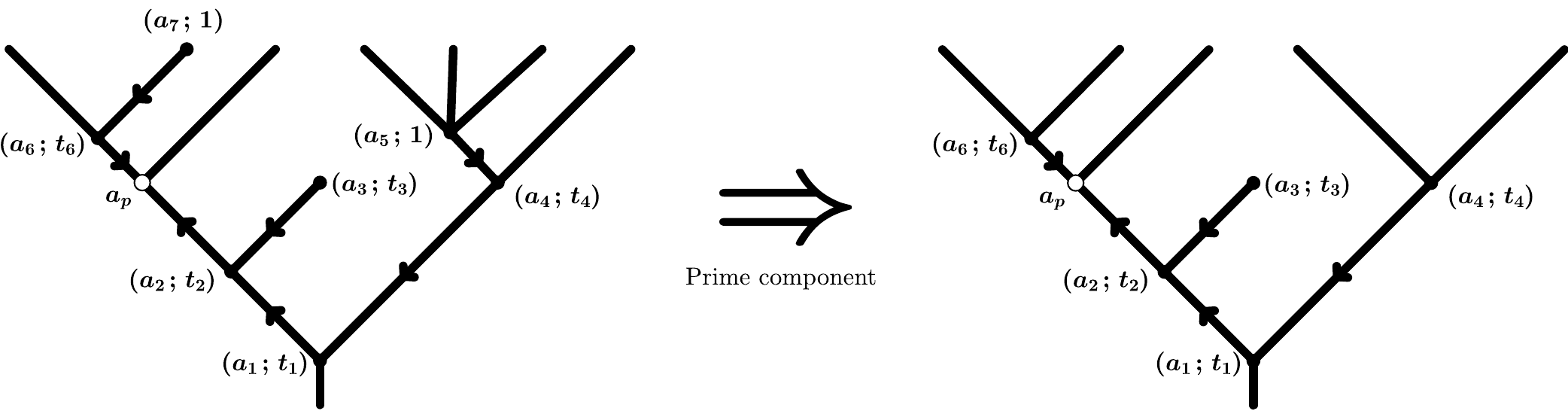}\vspace{-5pt}
\caption{Illustration of a composite point in $\mathcal{I}b^\Sigma(M)$ together with its prime component.}\label{C9}\vspace{-5pt}
\end{center}
\end{figure}

\noindent  A prime point is in the $k$-th filtration term 
$\mathcal{I}b_{k}^\Sigma(M)$ if the number of its geometrical inputs is smaller than~$k$. Similarly, a composite point is in the $k$-th filtration term if its prime component is in the $k$-th filtration term. For each $k\geq 0$, $\mathcal{I}b_{k}^\Sigma(M)$ is an $O$-Ibimodule and they define a filtration in  $\mathcal{I}b^\Sigma(M)$:
\begin{equation}\label{eq:filtr_proj_ibimod}
\xymatrix{
\mathcal{I}b_{0}^\Sigma(M) \ar[r] & \mathcal{I}b_{1}^\Sigma(M)\ar[r] & \cdots \ar[r] & \mathcal{I}b_{k-1}^\Sigma(M) \ar[r] & \mathcal{I}b_{k}^\Sigma(M) \ar[r] & \cdots \ar[r] & \mathcal{I}b^\Sigma(M).
}
\end{equation}

\begin{thm}\label{C7}
Let $O$ be a well-pointed operad and $M$ be an $O$-Ibimodule. If $O$ and $M$ are $\Sigma$-cofibrants, then  $\mathcal{I}b^\Sigma(M)$ and $\TT_{k}\mathcal{I}b_{k}^\Sigma(M)$ are cofibrant replacements of $M$ and $\TT_{k}M$ in the categories $\Sigma\Ibimod_{O}$ and $\TT_{k}\Sigma\Ibimod_{O}$, respectively. 
In particular the maps $\mu'$ and $\TT_{k}\mu'|_{\TT_{k}\mathcal{I}b_{k}^\Sigma(M)}$ are weak equivalences.  
\end{thm}

\begin{proof}
The proof is similar to that of \cite[Theorem 2.10]{Ducoulombier16} for bimodules (see also \cite{Berger06} for operads).  The main idea is that $\mathcal{I}b_{0}^\Sigma(M)$ is cofibrant and every
inclusion $\mathcal{I}b_{k-1}^\Sigma(M) \to \mathcal{I}b_{k}^\Sigma(M)$ is a cofibration obtained
as a sequence of cell attachments. To see that, one
notices first that  $\mathcal{I}b_{0}^\Sigma(M)=\calF_{Ib}^\Sigma(M(0))$, where abusing notation $M(0)$ denotes the sequence 
$
\begin{cases} M(0),&  k=0;\\ \emptyset, & k\geq 1. \end{cases} 
$  

Here we mean again that an $O$-Ibimodule
$N\rq{}$ is obtained from $N$ by a  {\it cell attachment} if $N\rq{}$ is obtained from $N$ by a pushout in 
$O$-Ibimodules:
\[
N\rq{}=N\coprod_{\calF_{Ib}^\Sigma (\partial X)} \calF_{Ib}^\Sigma (X),
\]
where $\partial X\to X$ is again a cofibration of $\Sigma$-sequences and the map $\calF_{Ib}^\Sigma (\partial X)\to N$ is induced by a map $\partial X\to N$ of $\Sigma$-sequences.  
\end{proof}

\begin{expl}\label{ex:Ibproj}
In case $O$ is doubly reduced $O(0)=O(1)=*$, the cofibration $\mathcal{I}b_{k-1}^\Sigma(M) \to \mathcal{I}b_{k}^\Sigma(M)$ is a sequence of $k+1$ cell attachments similarly to the case of bimodules,
see Example~\ref{ex:BprojO}. At the first step one attaches a cell whose interior consists of prime elements
of arity $k$ with no univalent non-pearl vertices; at the second step one attaches a cell whose interior
consists of
 prime elements of arity $k-1$ having exactly one univalent non-pearl; so on; at the $(i+1)-$th step one attaches a cell of prime elements of arity $k-j$ with $j$ univalent non-pearl vertices. 
\end{expl}

Similarly to the case of bimodules, one has the following equivalences of mapping spaces
%
%
\begin{equation}\label{eq:der_ibimod_vs_tr1}
\TT_{k}\Ibimod_{O}^{h}(\TT_{k}M\,;\, \TT_{k}M')\simeq \TT_{k}\Ibimod_{O}(\TT_{k}\mathcal{I}b_{k}^\Sigma(M)\,;\, \TT_{k}M') \cong \Ibimod_{O}(\mathcal{I}b_{k}^\Sigma(M)\,;\, M').
\end{equation}
The last equivalence, which is a homeomorphism, follows from the fact that $\mathcal{I}b_{k}^\Sigma(M)$
is obtained by attaching cells only of arity $\leq k$.

\subsubsection{Boardman-Vogt resolution in the $\Lambda$ setting.}\label{sss:cof21}

Now we assume that our operad $O$ is reduced and  we adapt the above construction to produce Reedy cofibrant replacements of $O$-Ibimodules.
 As  a $\Sigma$-sequence, we set
 \[
\calI b_O^\Lambda(M):=\calI b_{O_{>0}}^\Sigma(M).
\]
When $O$ is understood, we also write $\calI b^\Lambda(M)$.  \vspace{5pt}

The sequence $\calI b^\Lambda(M)$ is an $O_{>0}$-Ibimodule by construction. It is also an infinitesimal bimodule over $O$ in which the composition with the one point topological space $O(0)$  is defined as follows:
$$
\begin{array}{rcl}\vspace{4pt}
\circ^{i}: \mathcal{I}b_{>0}^\Sigma(M)(n)\times O(0) & \longrightarrow & \mathcal{I}b_{>0}^\Sigma(M)(n-1); \\ 
\left[ T\,;\,\{t_{v}\}\,;\,\{a_{v}\} \right]\,;\,\ast_{0} & \longmapsto & \left[ T\,;\,\{t_{v}\}\,;\,\{a'_{v}\}\right],
\end{array} 
$$  
where the family $\{a'_{v}\}$ is given by
$$
a'_{v}:=
\left\{
\begin{array}{cl}\vspace{3pt}
a_{v}\circ^{j}\ast_{0} & \text{if the } i\text{-th leaf of } T \text{ corresponds to the } j\text{-th incoming edge of } v= p, \\  \vspace{3pt}
a_{v}\circ_{j}\ast_{0} & \text{if the } i\text{-th leaf of } T \text{ corresponds to the } j\text{-th incoming edge of } v\neq p,  \\ 
a_{v} & \text{otherwise.}
\end{array} 
\right.
$$

One can easily see that the $O_{>0}$-Ibimodule map~\eqref{G8} induces a map
\begin{equation}\label{eq:mu_lambda_ibimod}
\mu'\colon \calI b^\Lambda(M)\to M,
\end{equation}
which respects the $\Lambda$ action and thus is an $O$-Ibimodules map.\vspace{5pt}

We define a filtration in $\calI b^\Lambda(M)$ exactly as the one  on $\calI b_{O_{>0}}^\Sigma (M)$,
see~\eqref{eq:filtr_proj_ibimod}:
\begin{equation}\label{eq:filtr_reedy_ibimod}
\xymatrix{
\mathcal{I}b_{0}^\Lambda(M) \ar[r] & \mathcal{I}b_{1}^\Lambda(M)\ar[r] & \cdots \ar[r] & \mathcal{I}b_{k-1}^\Lambda(M) \ar[r] & \mathcal{I}b_{k}^\Lambda(M) \ar[r] & \cdots \ar[r] & \mathcal{I}b^\Lambda(M).
}
\end{equation}
It is easy to see that the right action by $O(0)$ defined above preserves this filtration, because 
when acting on the prime components it only decreases their arity. As a consequence it is a filtration
of $O$-Ibimodules. Note that when we pass from 
$\mathcal{I}b_{k-1}^\Lambda(M)$ to 
$\mathcal{I}b_{k}^\Lambda(M)$ we only attach celles of arity exactly $k$, as the number of geometrical inputs
for the prime component is exactly the arity by the lack of the arity zero non-pearl vertices. 
(In case $O$ is doubly reduced, there is exactly one cell attached at this step.) As
 a consequences $\TT_k \mathcal{I}b_{k}^\Lambda(M) = \TT_k\mathcal{I}b^\Lambda(M)$.

\begin{pro}\label{p:BV_reedy_ibimod}
Assume that   $O$ is a reduced  $\Sigma$-cofibrant topological operad, and $M$ is a
$\Sigma$-cofibrant $O$-Ibimodule. Then, the objects $\mathcal{I}b^\Lambda(M)$ and $\TT_{k}\mathcal{I}b^\Lambda(M)$ are cofibrant replacements of $M$ and $\TT_{k}M$ in the categories $\Lambda\Ibimod_{O}$ and $\TT_{k}\Lambda\Ibimod_{O}$, respectively. In particular the maps $\mu'$ and $\TT_{k}\mu'$ are weak equivalences.
\end{pro}

\begin{proof}
Follows from Theorem~\ref{th:Ibim_reedy_model}~(ii) and Theorem~\ref{C7} applied to the operad 
$O_{> 0}$.
\end{proof}
 
 Assuming
 that $M'$ is a Reedy fibrant $O$-Ibimodule, we get:
 \begin{equation}\label{eq:der_ibimod_vs_tr2}
\TT_{k}\Ibimod_{O}^{h}(\TT_{k}M\,;\, \TT_{k}M')\simeq \TT_{k}\Ibimod_{O}
(\TT_{k}\mathcal{I}b^\Lambda(M)\,;\, \TT_{k}M') \cong \Ibimod_{O}(\mathcal{I}b_{k}^\Lambda(M)\,;\, M').
\end{equation}
The last equivalence, which is a homeomorphism, follows from the fact that $\mathcal{I}b_{k}^\Lambda(M)$
as an $O_{> 0}$-Ibimodule
is obtained by attaching cells only of arity $\leq k$.

\begin{expl}\label{D2}
Assuming that $O$ is doubly reduced $O(0)=O(1)=*$, one has $\calI b^\Lambda(O)(0)=*=
\includegraphics[scale=0.06]{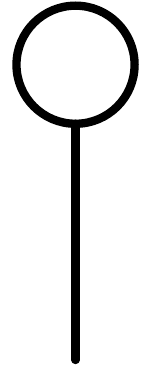}$. 
 As illustrated in Figure \ref{D0}, each point in the space $\calI b^\Lambda(O)(1)$ has a representative element of the form $[ T\,;\,\{t_{v}\}\,;\,\{a_{v}\}]$ in which $T$ is a pearled tree having two vertices such that the pearl is univalent and connected to the first incoming edge of the root. In other words, a point is determined by a pair $(\theta\,;\, t)$, with $\theta\in O(2)$ and $t\in [0\,;\,1]$, satisfying the relation $(\theta_{1}\,;\,0)\sim (\theta_{2}\,;\,0)$ due to the axiom~$(iv)$ of Construction~\ref{C8} and the condition $O(1)=*$, see Figure~\ref{D0}.  As a consequence, $\calI b^\Lambda(O)(1)$ is homeomorphic to the cone of the space $O(2)$, denoted $C(O(2))$.
\begin{figure}[!h]
\begin{center}
\includegraphics[scale=0.35]{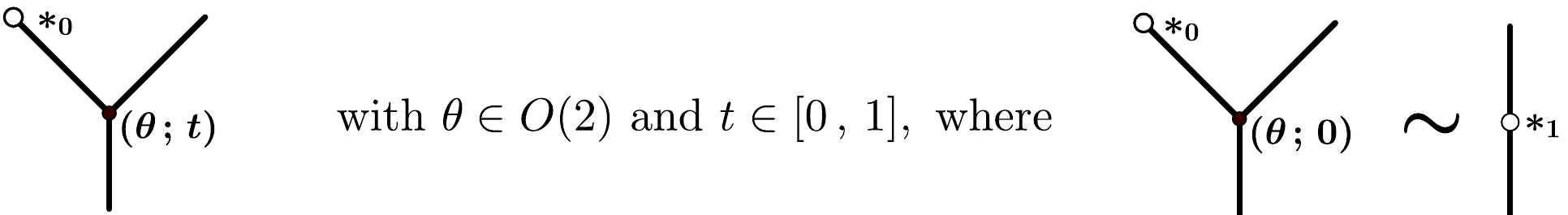}\vspace{-5pt}
\caption{Illustration of points in $\calI b^\Lambda(O)(1)$.}\label{D0}\vspace{-15pt}
\end{center}
\end{figure}
\end{expl}

\begin{expl}\label{ex:2compl}
Let again $O$ be doubly reduced. To visualize better the spaces $\calI b^\Lambda(M)(k)$,
it helps to look at the space of non-planar pearled trees with $k$ labelled leaves, whose non-pearls have 
arity $\geq 2$ (the pearl can be of any arity) and are labelled by numbers in $[0,1]$ respecting the inequalities $t_{s(e)}\geq t_{t(e)}$ from Construction~\ref{C8}. For example, for $k=0$, it is just
one point $*=\includegraphics[scale=0.06]{risunok1.pdf}$. For $k=1$, it is the closed interval:
\[
\includegraphics[scale=0.3]{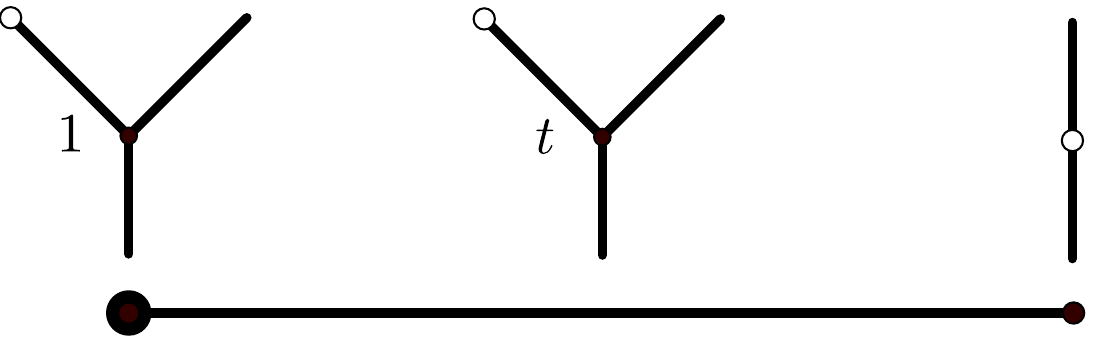}
\]


\noindent For $k=2$, it is a union of 3 triangles corresponding to the trees \quad
\includegraphics[scale=0.08]{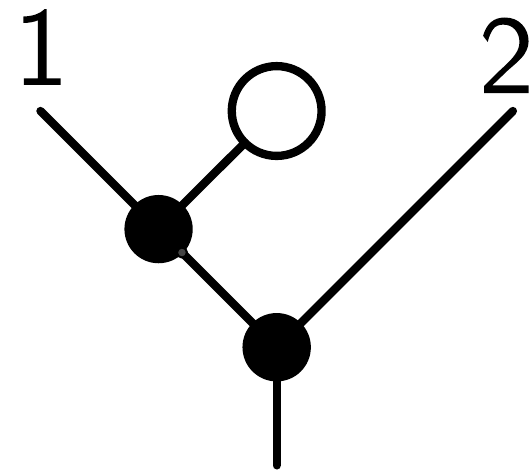}, \quad  \includegraphics[scale=0.08]{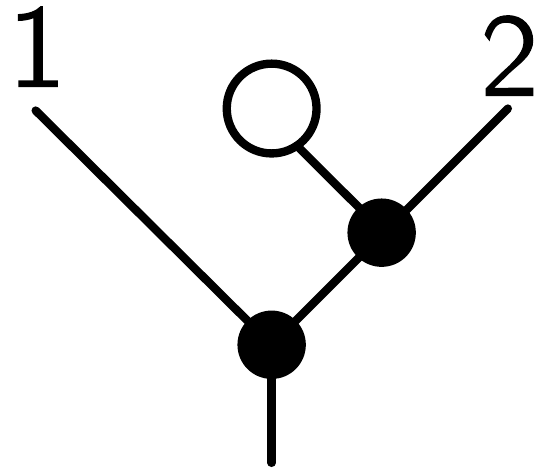},
\quad and \quad
\includegraphics[scale=0.1]{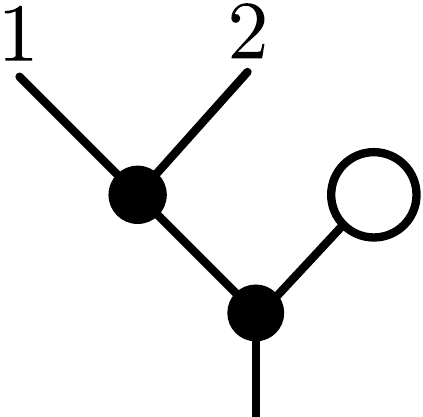}, as illustrated in  Figure~\ref{Fig:Ibbar2}.

\begin{figure}[!h]
\begin{center}
\includegraphics[scale=0.15]{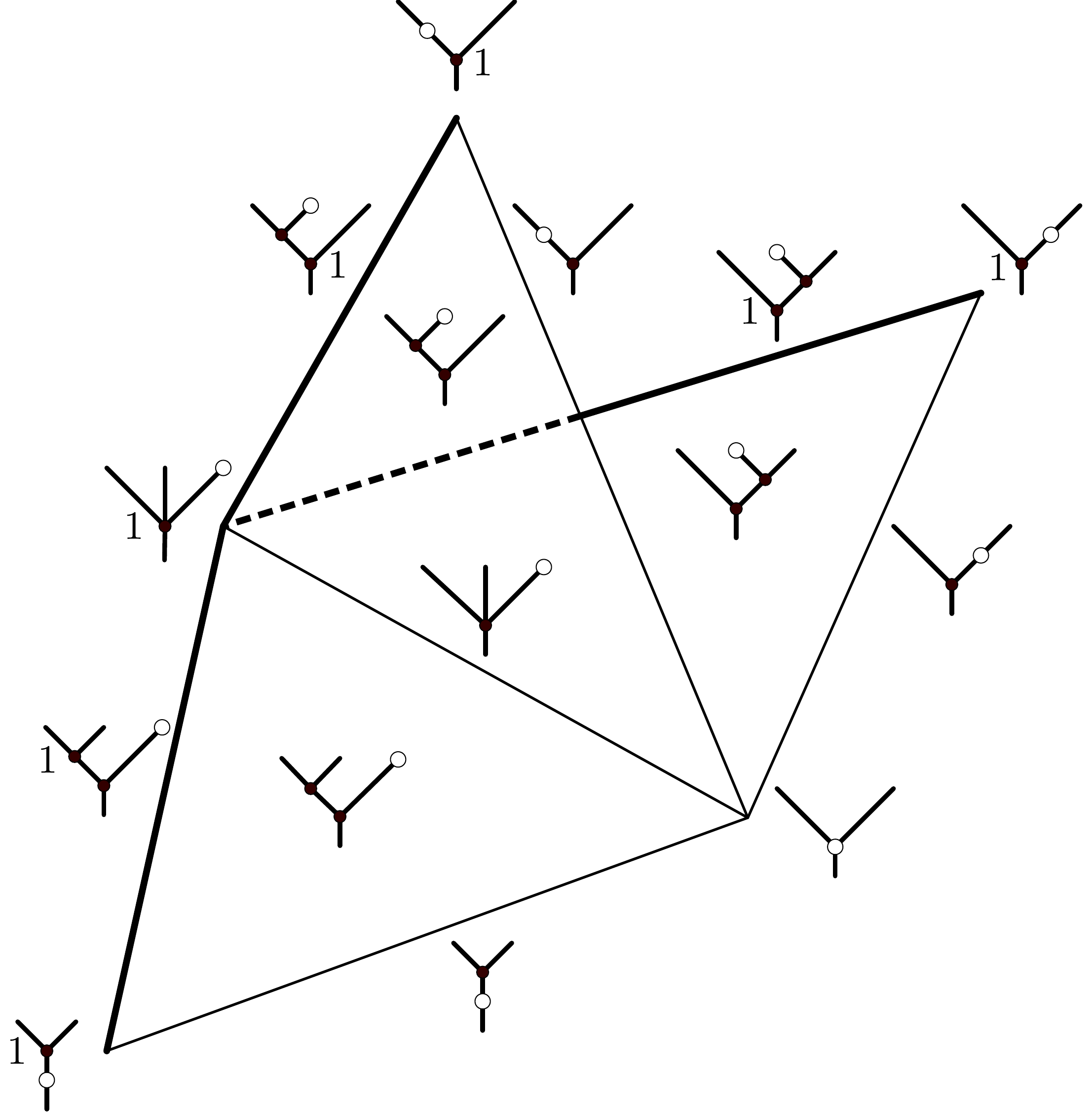}\vspace{-1pt}

\begin{minipage}{400pt}
\caption{Visualizing the space of trees associated to $\calI b^\Lambda(M)(2)$. The left leaf in each tree is labelled by~1, and the right one by~2.}\label{Fig:Ibbar2}
\end{minipage}\vspace{-35pt}
\end{center}
\end{figure}

\newpage

\noindent To each cell of this complex corresponding to a pearled tree $T$, one can assign the space of possible labelling of the vertices of~$T$ by $M$ and $O$ as in Construction~\ref{C8}. This assignment is a cellular
cosheaf, and the realization of this cosheaf is exactly the space $\calI b^\Lambda (M)(k)$, 
see Appendix~\ref{s:A2}.
Note that the part corresponding to the lower filtration term $\partial \calI b^\Lambda(M)(k):=
\calI b_{k-1}^\Lambda(M)(k)$ is the  union of the trees having at least one point labelled by~1.
For $k=1$, it is the part corresponding to the left point of the interval. For $k=2$, it is the union of the
left edges of the triangles that are drawn by fat lines.
\end{expl}

\begin{figure}[!h]
\begin{center}
\includegraphics[scale=0.25]{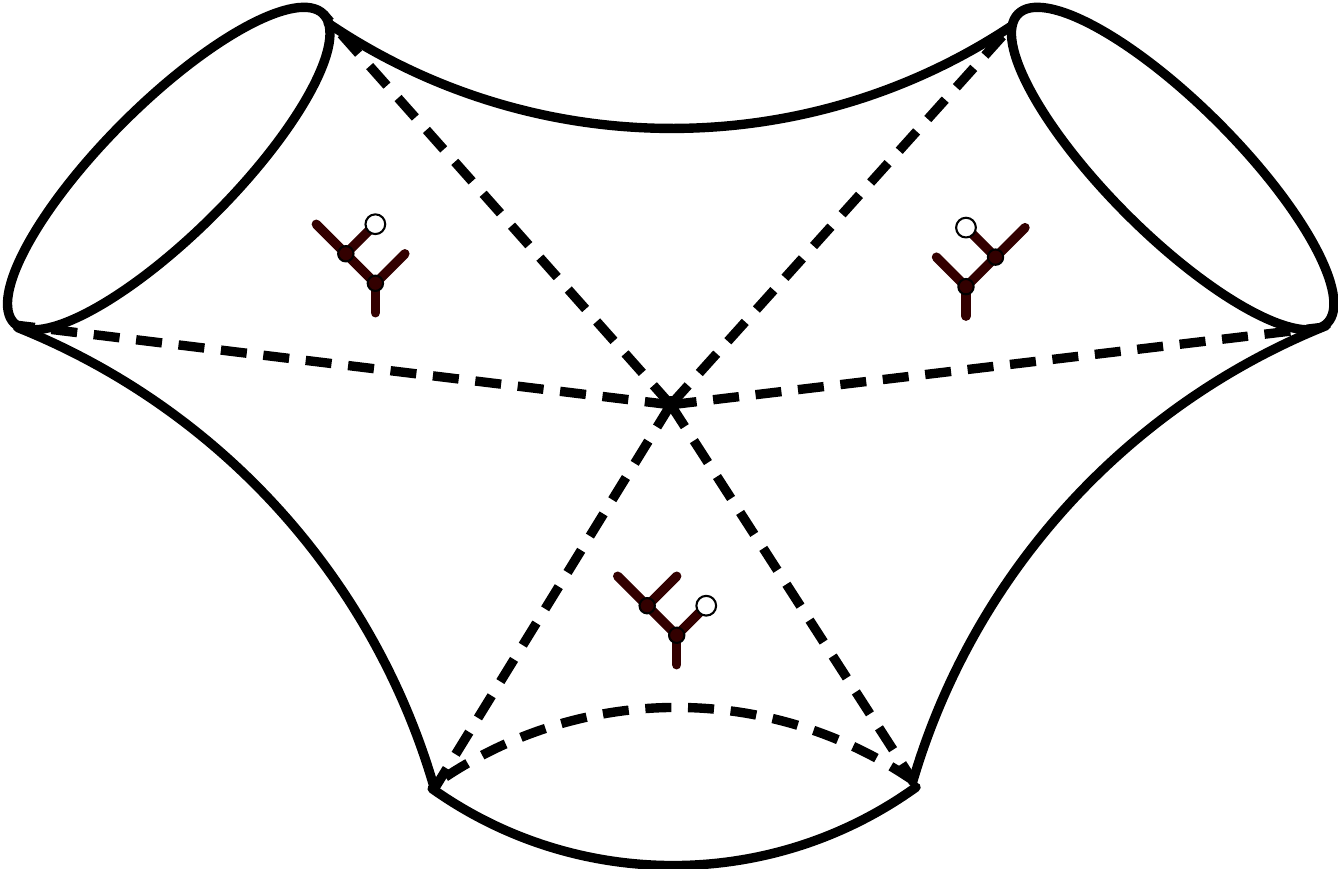}
\caption{The fiber of the map $\calI b^\Lambda(\calF_m)(2)\to \calF_m(2)$.}\label{Fig:3cones}\vspace{-10pt}
\end{center}
\end{figure}

\begin{expl}\label{ex:2FM}
As illustration to the two previous examples above, for the Fulton-MacPherson operad $O=M=\calF_m$, one has 
$\calI b^\Lambda(\calF_m)(0)=*$, $\calI b^\Lambda(\calF_m)(1)=C(\calF_m(2))=C(S^{m-1})=D^m$.
To describe $\calI b^\Lambda(\calF_m)(2)$, consider the projection $\mu'\colon 
\calI b^\Lambda(\calF_m)(2)\to \calF_m(2)=S^{m-1}$. One can easily see that it is a fibration, whose fiber
is an $(m+1)$-disc split into the union of three cones of $D^m$ and a cone of $S^m$ with three holes, see Figure~\ref{Fig:3cones}.
The three cones of $D^m$ correspond to the three triangles in Figure~\ref{Fig:Ibbar2}. Left sides of the triangles
correspond to the bases of the cones. The shared side for each triangle corresponds to the side face.
And the right sides correspond to the axes of the cones. The complement to these 3 cones is the part
of $\calI b^\Lambda(\calF_m)(2)$ corresponding to the segment
$$
\includegraphics[scale=0.3]{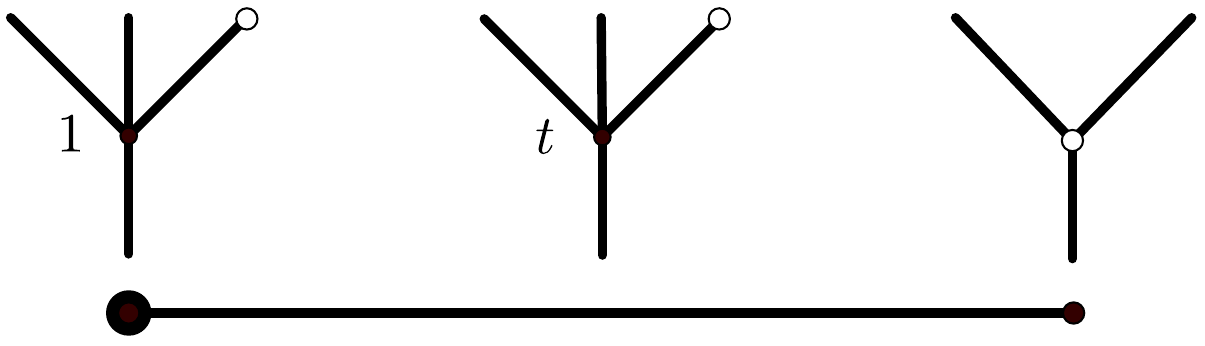}
$$
Note that any preimage of the map $\calF_m(3)\to\calF_m(2)$, $x\mapsto x\circ_3 \ast_0$, that forgets the last point is an $m$-disc
with two disc-shaped holes. The boundary of the first hole corresponds to the limit configurations when point 3 collides with~1; the boundary
of the second hole corresponds to the case when~3 collides with~2; and the outer boundary corresponds to the case when~3 escapes to infinity, or,
equivalently, 1 and~2 collide.   It is easy to see that the part of this $(m+1)$-disc lying in $\partial\calI b^\Lambda(\calF_m)(2)$ is 
exactly its boundary. As a fibration over $S^{m-1}$, the lower cone, that corresponds to the lower triangle in Figure~\ref{Fig:Ibbar2}, 
does not move, while the two other cones rotate one around the other.

\end{expl}

\section{Coherence}\label{S:COH}

In the first subsection we give definition of coherence in terms of homotopy colimit and in the second subsection we explain what this 
definition implies on the level of the Boardman-Vogt resolution.

\subsection{Definition of  coherence and strong coherence}\label{s:coherent}
In this subsection we explain the property of being {\it coherent} for an operad used in Main Theorem~\ref{th:main}. We also define a slightly
stronger, but easier to check property of {\it strong coherence}.

\begin{notat}\label{not:boudnary}
For a category $C$, its full subcategory of non-terminal objects is denoted by $\partial C$.
\end{notat}

\begin{defi}\label{d:psi}
(i) For any $k\geq 0$, the category $\Psi_k$ is defined to have as objects non-planar pearled  trees with $k$ leaves labelled bijectively
by the set $\underline{k}$, whose pearl can have any arity $\geq 0$, and the other vertices are of arity $\geq 2$. The morphisms in
$\Psi_k$ are inner edge contractions.

(ii) For each $\Psi_k$, we denote by $c_k$ its terminal object -- the pearled $k$-corolla, and by $c'_k$, $k\geq 2$, the tree with 2 vertices:
a pearled root of arity one, whose only outgoing edge is attached to the other vertex of arity~$k$.

(iii) We denote by $\Psi'_k$, $k\geq 2$, the subcategory of  $\Psi_k$ of all morphisms except the morphism $c'_k\to c_k$.

\end{defi}

\begin{figure}[!h]
\begin{center}
\includegraphics[scale=0.13]{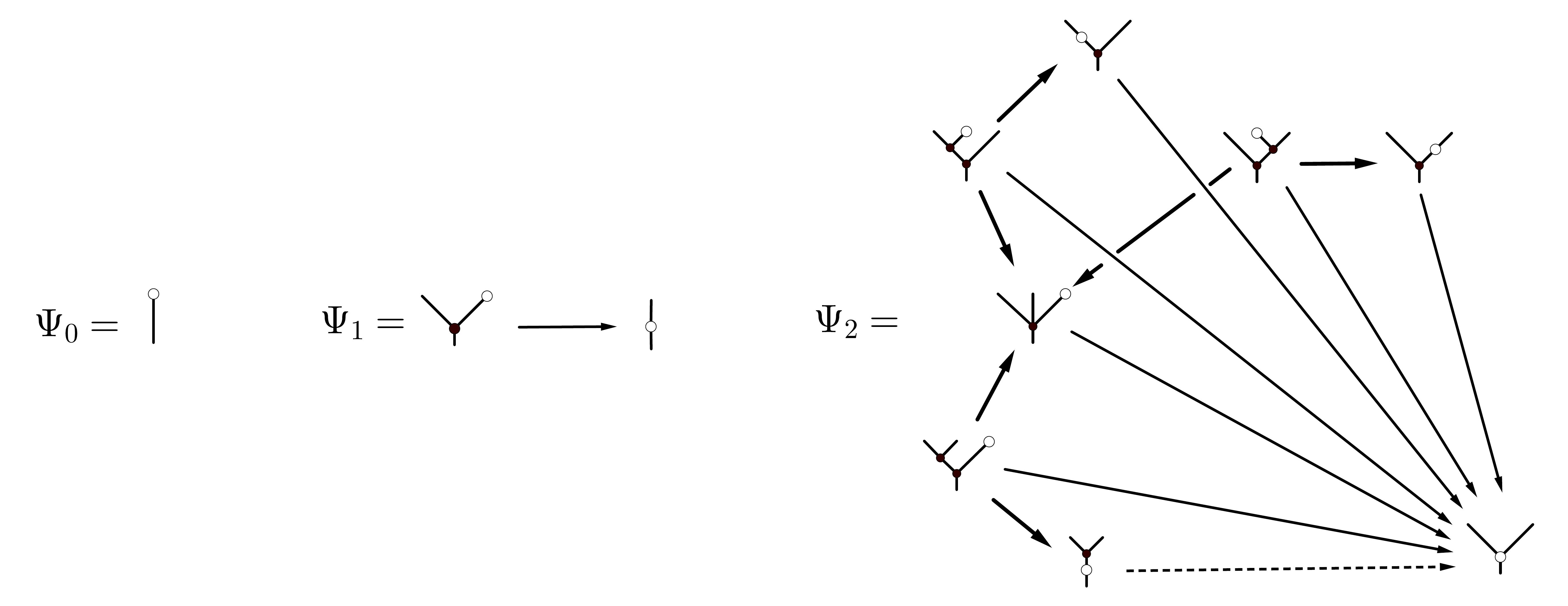}\vspace{-2pt}

\begin{minipage}{420pt}
\caption{Categories $\Psi_0$, $\Psi_1$, $\Psi_2$.  Morphism $c'_2\to c_2$ is shown by a dotted arrow.  We do  not put labels  on the leaves: for $\Psi_0$, its only graph has no leaves; for $\Psi_1$, trees have only one leaf that must be labelled by~1; for $\Psi_2$, the left leaf in each tree is labelled by~1
and the right one by~2. }\label{fig:psi}
\end{minipage}\vspace{-15pt}
\end{center}
\end{figure}

\begin{defi}\label{d:rho}
For a topological operad $O$, an $O$-Ibimodule $M$, and $k\geq 0$, define a $\Psi_k$ shaped diagram
$$
\begin{array}{rcl}
\rho_k^M\colon \Psi_k&\longrightarrow&Top;\\
T&\mapsto&M\left(|p|\right)\times\displaystyle\prod_{v\in V(T)\setminus\{p\}}O\bigl(|v|\bigr).
\end{array} 
$$
On morphisms it is defined by choosing a planar representative of each pearled tree, and then using the operadic composition and $O$-action on $M$ for each edge contraction. The choice of planar representatives does not matter in the sense that the obtained diagrams are objectwise homeomorphic.
\end{defi}

\begin{defi}\label{d:coherent}
(i) A diagram $\rho\colon\Psi_k\to\Topo$, $k\geq 2$, is called {\it coherent}, if the natural map
\begin{equation}\label{eq:coherent}
\underset {\partial\Psi_k}{\hocolim}\,\rho\, \longrightarrow\, \underset{\Psi'_k}{\hocolim}\, {\rho}
\end{equation}
is a weak equivalence.

(ii) A topological operad $O$ is called {\it coherent}, respectively, {\it $k$-coherent}, if the diagrams $\rho_i^O$ are coherent
for all $i\geq 2$, respectively, for all $i$ in the range $2\leq i\leq k$.
\end{defi}

Note that by definition any weakly doubly reduced operad is 1-coherent. In the next section we show that for a strictly doubly reduced operad $O$, and its any infinitesimal bimodule $M$, one has two homeomorphisms 
$$
\calI b^\Lambda(M)(k)\cong \underset {\Psi_k}{\hocolim}\, \rho_k^M \hspace{15pt}\text{and}\hspace{15pt} \partial\calI b^\Lambda(M)(k)=
\calI b^\Lambda_{k-1}(M)(k) \cong \underset {\partial\Psi_k}{\hocolim}\,\rho_k^M.
$$ 
Thus, informally speaking, an operad is coherent if making some small
hole in the interior of $\calI b^\Lambda(O)(k)$, for any $k\geq 2$, makes it weakly collapsible to the boundary $\partial\calI b^\Lambda(O)(k)$. 

\begin{defi}\label{d:trunk_arity}
For a pearled tree $T$, the {\it arity of its trunk} is the arity of the pearl $p'$ in the tree $T'$ obtained from $T$
by contracting all edges of its trunk to one vertex $p'$.
\end{defi}

For example the arity of the trunk for the pearled tree from Figure~\ref{fig:ptree} is four.

\begin{defi}\label{d:ud}
For $k\geq 2$, define:

(i) $\Psi_k^U$ as a full subcategory of $\Psi_k$ of objects that have no morphism to $c'_k$. 
It consists of trees whose trunk has arity $\geq 2$.

(ii) $\Psi_k^L$ as a full subcategory of $\Psi_k$ whose elements can be reached (by a morphism)  from  objects that have a morphism to $c'_k$. 
It consists of trees either with a pearled root, or with a univalent pearl and a trunk that has only two vertices.

(iii) $\Psi_k^{UL}:=\Psi_k^U\cap\Psi_k^L$. Explicitly, its objects have either a pearled root of arity $\geq 2$, or a trunk with two vertices:
a univalent perl and a root of arity $\geq 3$. 
\end{defi}

The letter $U$ stays for {\it upper}, and the letter $L$ stays for {\it lower}. We think about $\Psi_k^U$ and $\Psi_k^L$ as respectively, upper and lower parts of the category $\Psi_k$. For instance, in Figure~\ref{fig:psi}, $\Psi_2^U$ is the union of two upper squares, $\Psi_2^L$ is the lower square while $\Psi_2^{UL}$ is the central edge: it has
   only two objects and a morphism between them.

\begin{defi}\label{d:scoherent}
(i) A diagram $\rho\colon\Psi_k\to\Topo$ is called {\it strongly coherent} if the natural map
\begin{equation}\label{eq:scoherent}
\underset {\partial\Psi_k^U}{\hocolim}\, \rho \, \longrightarrow \, \rho(c_k)\vspace{-5pt}
\end{equation}
is a weak equivalence.

(ii) A topological operad $O$ is called {\it strongly coherent}, respectively, {\it strongly $k$-coherent}, if the diagrams
$\rho_i^O$ are strongly coherent for all $i\geq 2$, respectively, for all $i$ in the range $2\leq i\leq k$. 
\end{defi}

\begin{lmm}\label{l:coherent}
The properties of an operad to be coherent, $k$-coherent, strongly coherent, strongly $k$-coherent are preserved by weak equivalences.
\end{lmm}

\begin{proof} A weak equivalence $O\to O\rq{}$ induces a weak equivalence of diagrams $\rho_k^O\to\rho_k^{O\rq{}}$, which induces
equivalences of homotopy colimits.
\end{proof}

\begin{thm}\label{th:coherent}
(i) If a diagram $\rho\colon \Psi_k\to\Topo$, $k\geq 2$, is strongly coherent, then it is coherent.

(ii) If an operad is strongly coherent, respectively, strongly $k$-coherent, then it is coherent, respectively, $k$-coherent.
\end{thm}

Obviously, $(ii)$ follows from $(i)$, so we need to show only $(i)$. The part $(i)$ is a consequence of the following proposition.

To recall, for a diagram $F\colon C\to\Topo$, its homotopy colimit is $\underset{C}{\hocolim}\, F:= |(-\downarrow C)|\otimes_C F(-)$.

\begin{pro}\label{p:coherent}
For any $\rho\colon\Psi_k\to\Topo$, $k\geq 2$, one has

(i) the natural inclusion
\begin{equation}\label{eq:coh1}
\underset{\partial\Psi_k^U}{\hocolim}\,\rho
\underset{ \underset{\partial\Psi_k^{UL}}{\hocolim}\,\rho }{\coprod} 
\underset{\partial\Psi_k^L}{\hocolim}\,\rho\,
\xrightarrow{\cong} \,
\underset{\partial\Psi_k}{\hocolim}\,\rho\vspace{-5pt}
\end{equation}
is a homeomorphism;

(ii) the natural inclusion
\begin{equation}\label{eq:coh2}
\underset{\Psi_k^U}{\hocolim}\,\rho
\underset{ \underset{\partial\Psi_k^{UL}}{\hocolim}\,\rho }{\coprod} 
\underset{\partial\Psi_k^L}{\hocolim}\,\rho\,
\xrightarrow{\simeq} \,
\underset{\Psi'_k}{\hocolim}\,\rho\vspace{-5pt}
\end{equation}
is a homotopy equivalence.
\end{pro}

Indeed, this shows that both sides of~\eqref{eq:coherent} are described as a union of an upper and  a lower parts. By proposition above, their
lower parts are always equivalent, while the equivalence of upper parts is exactly the property of being strongly coherent (as the terminal object
of $\Psi_k^U$ is $c_k$, one gets $\underset{\Psi_k^U}{\hocolim}\,\rho\simeq \rho(c_k)$).

\begin{proof}[Proof of Proposition~\ref{p:coherent}]
(i) follows from the fact that $\partial \Psi_k$ is a union of $\partial\Psi_k^U$ and $\partial\Psi_k^L$, and there is no morphism from the objects
of their intersection $\partial\Psi_k^{UL}$ outside this subcategory.\vspace{5pt}

To prove (ii) we construct a deformation retraction of the target in~\eqref{eq:coh2} onto its source. Figure~\ref{fig:coh_retr} is an example how the deformation retraction looks like in case $k=2$. For this purpose, we introduce the following notation:
\begin{itemize}
\item[$\blacktriangleright$] A {\it cubical category} $C_i$, $i\geq 0$, is the $i$-th power of the category $\{0\to 1\}$ with 2 objects and 1 morphism between them.
\item[$\blacktriangleright$] A {\it subcubical category} is $\partial C_i$, $i\geq 0$, see Notation~\ref{not:boudnary}.
\item[$\blacktriangleright$] An {\it almost cubical category} is the category $C'_i$, $i\geq 1$, obtained from $C_i$ by removing only one morphism $\tau'\to\tau$
to its terminal object. The objects  $\tau$ and $\tau'$  are called {\it almost terminal}  and {\it subterminal} objects of  $C'_i$, respectively.
\end{itemize}

Note that $\tau'\to\tau$ must be indecomposable, meaning that all  coordinates of $\tau'$ except exactly one are~1s. \vspace{5pt}

Cubical and subcubical categories are well-known objects. The realization $|C_i|$ is the $i$-dimensional cube, and $|\partial C_i|$ is the barycentric subdivision of an $(i-1)$-simplex. \vspace{5pt}

Let us describe $|C'_i|$. We call {\it front face} the $(i-1)$-subcube of $|C_i|$, whose terminal object is $\tau'$. The parallel to it $(i-1)$-subcube
containing $\tau$ gets called {\it back face}. One can easily see that $|C'_i|$ is the cube $|C_i|$ minus the interior of the pyramid $P$ 
with apex $\tau$ and base the front face, and minus the interiors of all faces of $P$ that contain the edge $(\tau',\tau)$. 

$$
\includegraphics[scale=0.22]{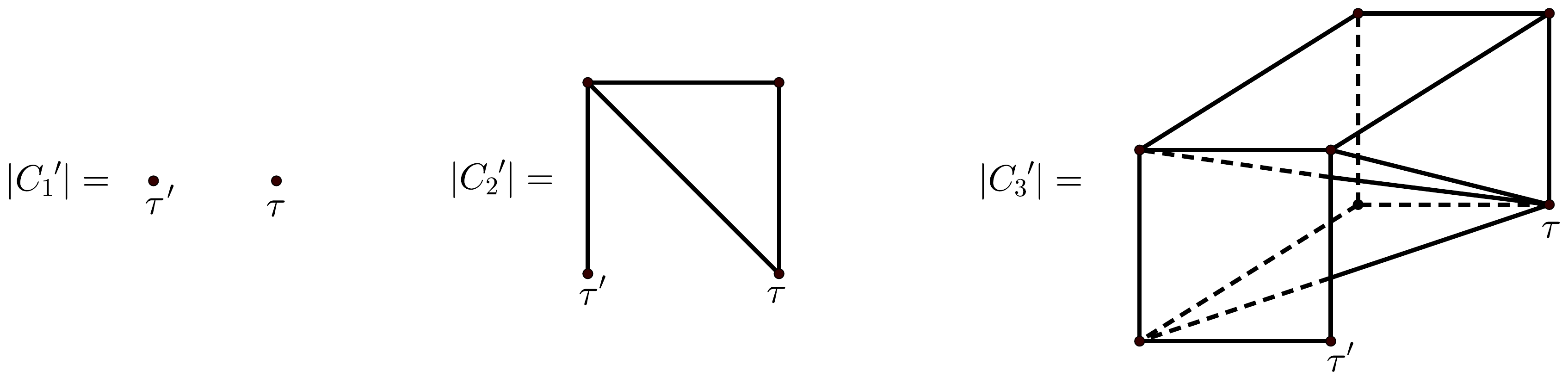}
$$

Define projection $p\colon |C'_i|\to |C'_i|$ as the identity map on the front face and as the stereographic projection from $\tau'$ to the union of faces in the
boundary of 
the cube $|C_i|$ that do not contain $\tau'$. One can easily see that the homotopy
\begin{equation}\label{eq:cubic_retr}
H(t)=(1-t)\cdot id +t\cdot p
\end{equation}
defines a deformation retraction from $|C'_i|$ onto the union of $|\partial C_i|$ and the back face.

\begin{figure}[!h]
\begin{center}
\includegraphics[scale=0.09]{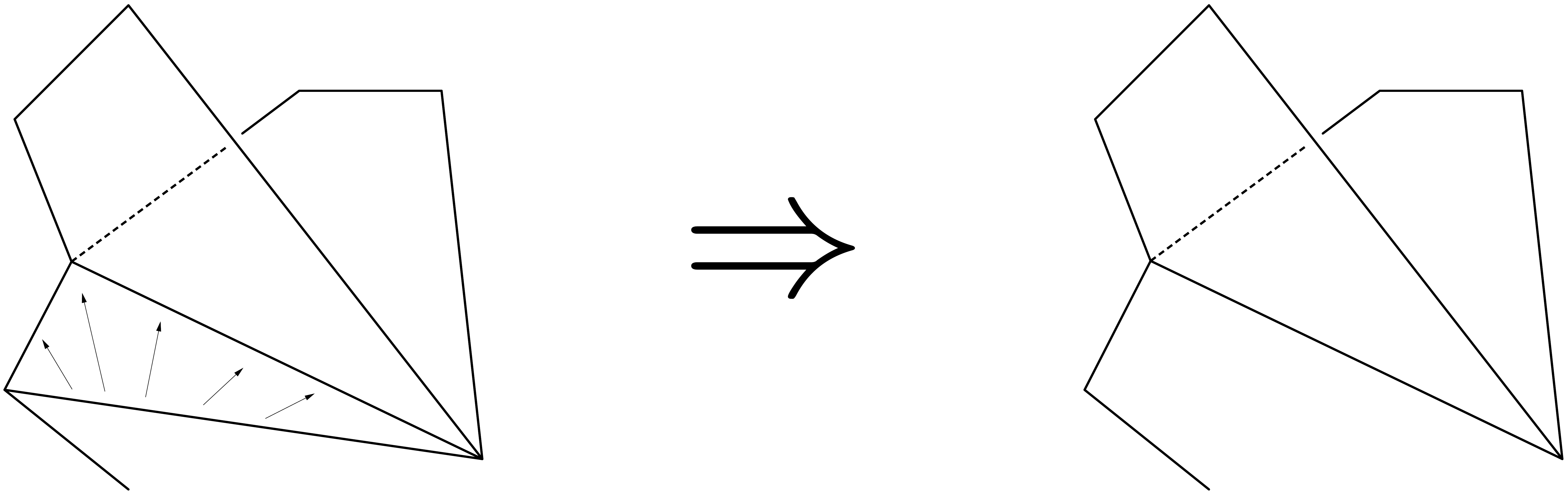}\vspace{-5pt}
\caption{Deformation retraction of $\underset{\partial\Psi'_2}{\hocolim}\,\rho$ onto $\underset{\Psi_2^U}{\hocolim}\,\rho
\underset{ \underset{\partial\Psi_2^{UL}}{\hocolim}\,\rho }{\coprod} 
\underset{\partial\Psi_2^L}{\hocolim}\,\rho$.}\label{fig:coh_retr}
\end{center}
\end{figure}

Now we apply this construction to prove our proposition. We will define a deformation retraction $H$ from $|\Psi'_k|$ onto
$|\Psi_k^U|\underset{|\partial\Psi_k^{UL}|}{\coprod} |\partial\Psi_k^{L}|$. The reader can easily see that our retraction can be  
extended to homotopy colimits. For any tree $T\in \Psi_k$, the category $(T\downarrow\Psi_k)\subset\Psi_k$ is cubical with 
initial object $T$, terminal object $c_k$, 
and the dimension of
the cube being the number of internal edges of $T$. If there is no map $T\to c'_k$, i.e. $T\in\Psi_k^U$, then $T\downarrow\Psi'_k$ is
the same cube lying entirely in $\Psi_k^U$, and we define $H$ as the constant identity on $|T\downarrow\Psi'_k|\subset|\Psi'_k|$. 
If there is a map $T\to c'_k$, then $T\downarrow \Psi'_k$ is almost cubical, whose almost terminal object is $c_k$ and subterminal object is $c'_k$. 
Note that $|T\downarrow\Psi'_k|\cap|\Psi_k^U|$ is its back face and $|T\downarrow\Psi'_k|\cap|\partial\Psi_k^L|$ is its subcubical part.
Define $H$ on $|T\downarrow\Psi'_k|\subset|\Psi'_k|$ by formula~\eqref{eq:cubic_retr}. One has $|\Psi'_k|=\bigcup_{T\in\Psi_k}|T\downarrow
\Psi'_k|$.
It is easy to see that $H$ is compatibly defined on all such pieces of $|\Psi'_k|$ and produces the required deformation retraction.
\end{proof}

\begin{rmk}\label{r:coherent}
(i) Being coherent does not imply being strongly coherent. Two non-equivalent spaces might become weakly equivalent after attaching the same space. For example, let $T$ be an open tubular neighborhood of a non-trivial knot $K$  in $S^3$, and let $S^1\subset\partial T$ be a circle
bounding a disc $D^2$ that transversely intersects~$K$. Then  attaching this disc $D^2$  to $S^1$ and to $S^3\setminus T$ produces 
in both cases
a contractible space even though the initial inclusion $S^1\subset S^3\setminus T$ is not an equivalence.

(ii) Note, however, that if the map~\eqref{eq:coherent} is a weak equivalence, then by the Exactness Axiom (and
Proposition~\ref{p:coherent}), the map~\eqref{eq:scoherent}
must induce isomorphism for any homology and cohomology theory.
\end{rmk}

\begin{expl}
Let us check that the Fulton-MacPherson operad $\calF_2$ is strongly 2-coherent. The category $\partial\Psi_2^U$
is as follows:
$$
\includegraphics[scale=0.2]{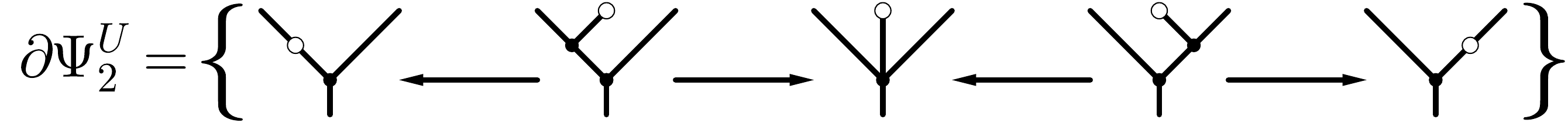}
$$
Thus, one has to check that the map
\begin{equation}\label{eq:Psi_2U}
\hocolim\bigl( \calF_2(2)\xleftarrow{pr_1}\calF_2(2)\times\calF_2(2)\xrightarrow{\circ_1}\calF_2(3)
\xleftarrow{\circ_2}\calF_2(2)\times\calF_2(2)\xrightarrow{pr_1}\calF_2(2)\bigr)
\longrightarrow \calF_2(2)
\end{equation}
is an equivalence. Here, $pr_1$ is the projection to the first factor. The natural maps from each entry to $\calF_2(2)$ are: identity for the leftmost and rightmost ones;
$pr_1$ from both $\calF_2(2)\times\calF_2(2)$; the map forgetting the second point 
$x\mapsto x\circ_2 \ast_0$, for $\calF_2(3)$.
 As shown  in Figure~\ref{G0}, the space $\mathcal{F}_{2}(3)$ is homeomorphic to a solid torus minus two open solid tori that turn one around the other.\vspace{5pt}

\begin{figure}[!h]
\begin{center}
\includegraphics[scale=0.24]{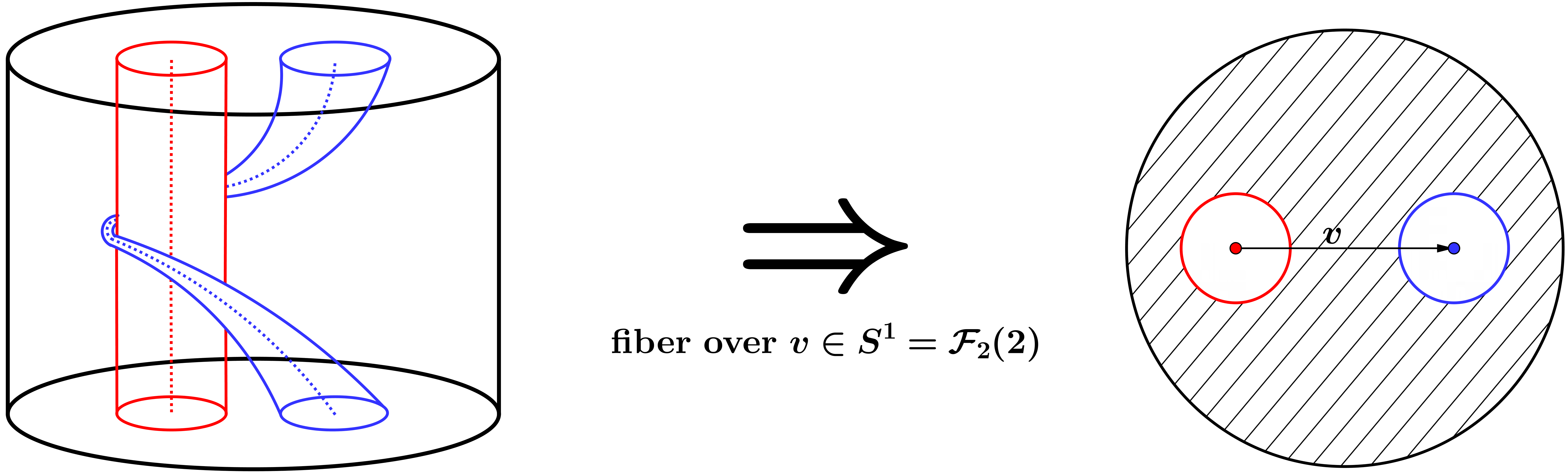}
\caption{Illustration of $\mathcal{F}_{2}(3)$ together with the fiber over $v\in \mathcal{F}_{2}(2)$.}\label{G0}\vspace{-5pt}
\end{center}
\end{figure}

\noindent Indeed,  the map $\mathcal{F}_{2}(3)\rightarrow \mathcal{F}_{2}(2)$ is a fiber bundle. If $v\in \mathcal{F}_{2}(2)=S^{1}$, then the fiber over $v$ consists of (possibly infinitesimal) configurations of $3$ points in which the direction between the points $1$ and $3$ is fixed to be $v$. The boundary of $\calF_2(3)$ has three connected components, that all are tori. One inner boundary torus corresponds to configurations, where $2$ collided with $1$,  the other one 
 to configurations, where $2$ collided with $3$. The outside boundary torus represents limit configurations in which the point $2$ escapes to infinity or, equivalently, the points $1$ and $3$ are infinitsimally close one to another. 
On the other hand, the inclusions $\circ_1$ and $\circ_2$ in the diagram~\eqref{eq:Psi_2U} produce the two inner boundary tori. Taking homotopy colimit corks in these two holes in $\calF_2(3)$. As a result the homotopy colimit above is homeomorphic to $S^1\times D^2\simeq\calF_2(2)$.

To compare, if we were to check that $\calF_2$ is 2-coherent, one would need to verify that the map 
\begin{equation}\label{eq:Psi_2U2}
\underset{\partial\Psi_2}{\hocolim}\,\rho_2^{\calF_2}\,\longrightarrow
\underset{\Psi_2^U}{\hocolim}\,\rho_2^{\calF_2}
\underset{ \underset{\partial\Psi_2^{UL}}{\hocolim}\,\rho_2^{\calF_2} }{\coprod} 
\underset{\partial\Psi_2^L}{\hocolim}\,\rho_2^{\calF_2}
\end{equation}
is a weak equivalence. The first space $\underset{\partial\Psi_2}{\hocolim}\,\rho_2^{\calF_2}$ is homeomorphic to $S^1\times S^2$, obtained 
from $\underset{\partial\Psi_2^U}{\hocolim}\,\rho_2^{\calF_2}\cong S^1\times D^2$, described above,   by \lq\lq{}corking\rq\rq{} with a solid torus the outside part of the boundary of $\calF_2(3)$. The right-hand side of~\eqref{eq:Psi_2U2} is homeomorphic to $S^1$ times the space $D^3$ from which one removes 
the interior of the lower cone,
see Figure~\ref{Fig:3cones}.
\end{expl}

\begin{expl}
The operad $\calB_0$, denoted by $\Lambda$ in this paper, is obviously strongly coherent. In Section~\ref{s:FM} we   prove that the Fulton-MacPherson operad $\mathcal{F}_{m}$ and, therefore, the little discs operad $\calB_m$ are strongly coherent, $m\geq 1$. Since the homotopy colimits commute, one has that $\calB_\infty$ is also strongly coherent. 
This statement is equivalent to the fact that the categories $\partial\Psi_k^U$, $k\geq 2$, are contractible, which
is not obvious, but can be checked directly.
Unfortunately, these are the only  examples of  coherent operads that we know. 
\end{expl}

\subsection{Boardman-Vogt resolution as a homotopy colimit}\label{s:hocolim}

Recall Subsection~\ref{ss:cof2}, where the Boardman-Vogt resolution $\calI b^\Lambda(M)$ of an $O$-Ibimodule is defined. Elements
of $\calI b^\Lambda(M)(k)$ are triples $[\, T\, ; \, \{a_v\},\, \{t_v\}\,]$, where $T\in \textbf{ptree}_k^{\geq 1}$ is a planar pearled tree, whose non-pearl vertices have arity $\geq 1$; its leaves are labelled
by a permutation $\sigma\in\Sigma_k$; $\{a_v\}$ is a collection of elements in $M$ and $O$ labelling the vertices of~$T$; and $\{t_v\}$
is an admissible collection of numbers in $[0,1]$ labelling $V(T)\setminus\{p\}$. It will be convenient sometimes to assume that the pearl is also labelled
by~0. \lq\lq{}Towards the pearl orientation\rq\rq{} of edges enables the set $V(T)$ of vertices of~$T$ with a   \lq\lq{}toward the pearl\rq\rq{} poset structure, for which the pearl $p$
is the minimal element. The space $H(T)\subset [0,1]^{V(T)}$ of admissible labels is the space of pointed order-preserving maps $V(T)\to [0,1]$. {\it Pointed} means
$p\mapsto 0$; {\it order preserving} translates into a set of inequalities
\[
t_{s(e)}\geq t_{t(e)},
\]
for any internal edge $e\in E^{int}(T)$. For example, for the tree
$$
\includegraphics[scale=0.3]{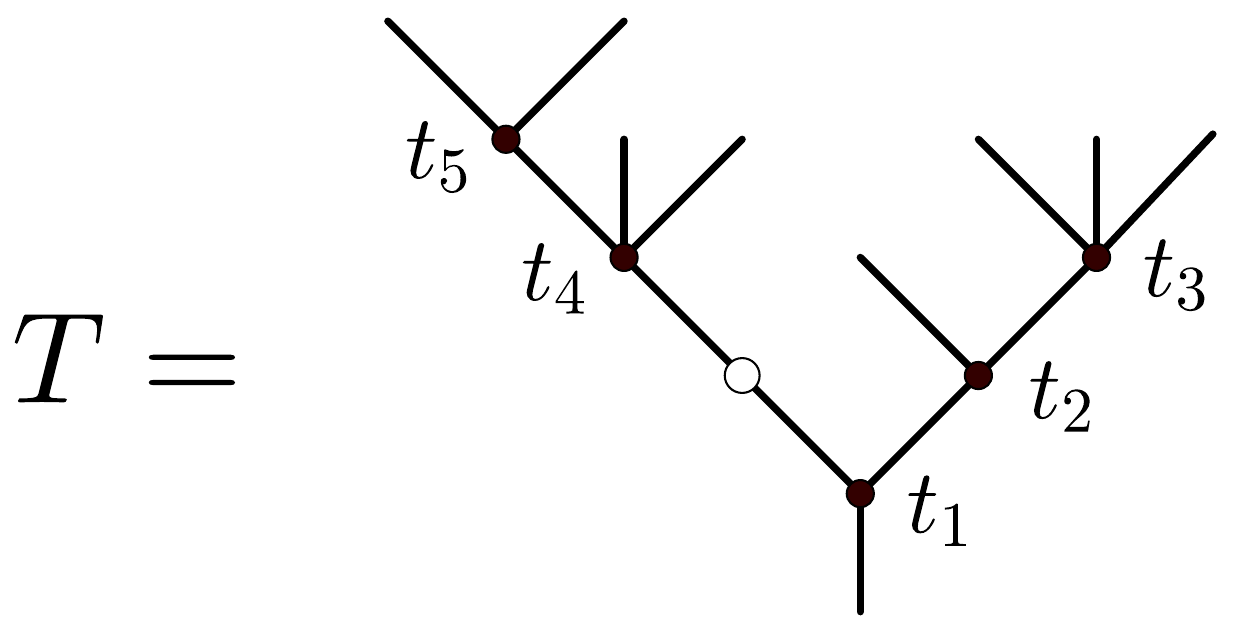}
$$
the polytope $H(T)$ is given by the inequalities $0\leq t_1,t_2,t_3,t_4,t_5\leq 1$, $t_1\leq t_2\leq t_3$ and $t_4\leq t_5$. \vspace{5pt}

In case $O$ is doubly reduced, the set $\textbf{ptree}_k^{\geq 1}$  can be replaced by its subset $\textbf{ptree}_k^{\geq 2}$ of  trees, whose 
 non-pearl vertices have
arity $\geq 2$. One has an obvious projection $\pi_k\colon \textbf{ptree}_k^{\geq 2}\to \Psi_k$, that forgets the planar structure. We say that an element
$x\in \calI b^\Lambda(M)(k)$ is labelled by $T_0\in\Psi_k$ if $x$ has a representative $[\, T\, ; \, \{a_v\},\, \{t_v\}\,]$ with $\pi_k(T)=T_0$. 
Recall also that we denote by $\partial \calI b^\Lambda(M)(k)$ the $(k-1)$-th filtration term $\calI b_{k-1}^\Lambda(M)(k)$.\vspace{5pt}

Since $\Psi_k$ is a poset, $T\downarrow \Psi_k$ is viewed below as a subcategory of $\Psi_k$.

\begin{thm}\label{th:hocolim}
For a doubly reduced topological operad $O$, and an $O$-Ibimodule $M$, one has a natural homeomorphism:
\begin{equation}\label{eq:delta}
\delta_k\colon\underset{\Psi_k}{\hocolim}\, \rho_k^M\, \longrightarrow\, \calI b^\Lambda(M)(k),\quad k\geq 0.
\end{equation}
Moreover, 
\begin{itemize}
\item[(i)] it sends homeomorphically $\underset{\partial\Psi_k}{\hocolim}\, \rho_k^M$ onto $\partial \calI b^\Lambda(M)(k)$;

\item[(ii)] for any $T\in\Psi_k$, it sends homeomorphically $\underset{T\downarrow \Psi_k}{\hocolim}\, \rho_k^M$ onto the subspace of elements 
labelled by the trees from~$T\downarrow\Psi_k$.
\end{itemize}
\end{thm}

\begin{proof}

By definition, our homotopy colimit 
  is a quotient space 
\begin{equation}\label{eq:hocolim_union}
\underset{\Psi_k}{\hocolim}\, \rho_k^M\, =
\left. \left( \coprod_{T\in \Psi_k} |T\downarrow \Psi_k| \times \rho_k^M(T)\right)\right/\sim.
\end{equation}
We will define $\delta_k$ on each summand of~\eqref{eq:hocolim_union}:
\[
\delta_{T}\colon |T\downarrow\Psi_k|\times\rho_k^M(T) \longrightarrow \calI b^\Lambda (M)(k).
\]

Recall that
\[
\rho_k^M(T)=M(|p|)\times\prod_{v\in V(T)\setminus\{p\}} O(|v|).
\]
Recall also that for  defining  $\rho_k^M$ we need to choose a planar representative for each $T\in\Psi_k$. Abusing notation we denote by $T$
this planar tree. (Because of the relation~$(ii)$ of Construction~\ref{C8} this choice does not matter.) Thus an element $a\in \rho_k^M(T)$ is a collection of labels $\{a_v\}$ of vertices of~$T$. For $(x,a)\in |T\downarrow\Psi_k|\times\rho_k^M(T)$, we set
\begin{equation}\label{eq:delta2}
\delta_{T}(x,a)= [\, T\, ;\, a, \, f_{T}(x)\, ],
\end{equation}
where $f_{T}$ is a continuous map
\[
f_{T}\colon |T\downarrow\Psi_k|\longrightarrow H(T)
\]
from the cube $|T\downarrow\Psi_k|$ to the polytope $H(T)$ of admissible labels of $T$. The maps $f_{T}$, $T\in\Psi_k$, are defined below. (Note that both polytopes have the same dimension $d=|E^{int}(T)|=|V(T)\setminus\{p\}|$.) In other words, the idea of our homeomorphism is to replace
the labelling by elements from  the polytopes $H(T)$ with the labelling from  the cubes $|T\downarrow\Psi_k|$, $T\in\Psi_k$. Figure~\ref{F5} gives an example of such cube and Figure~\ref{F6} describes the image of the vertices of this cube in $H(T)$ by the map $f_T$ that we are going to construct. 

\begin{figure}[!h]
\begin{center}
\includegraphics[scale=0.15]{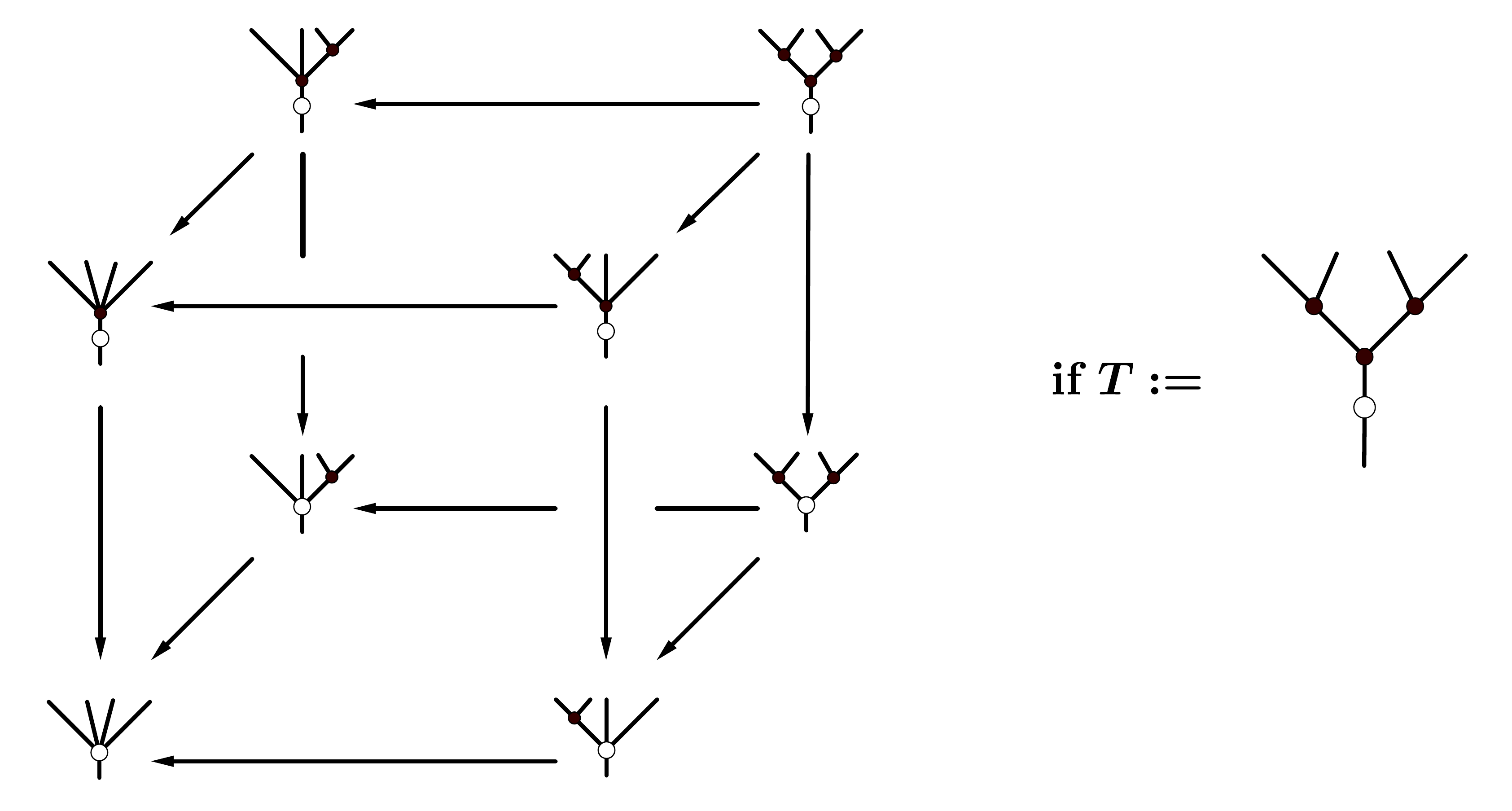}\vspace{-5pt}
\caption{The cube $T\downarrow\Psi_4$ associated to the planar pearled tree $T$.}\label{F5}\vspace{-10pt}
\end{center}
\end{figure}

To make the maps $\delta_{T}$, $T\in\Psi_k$, compatible with the relations of homotopy colimit~\eqref{eq:hocolim_union}, for each morphism
$T\xrightarrow{\alpha} T'$ in $\Psi_k$, the following square must commute:
\begin{equation}\label{eq:relab_compat}
\xymatrix{
|T'\downarrow\Psi_k|\ar@{^{(}->}[d]\ar[r]^{f_{T'}}&H(T')\ar@{^{(}->}[d]^{\alpha^*}\\
|T\downarrow\Psi_k|\ar[r]^{f_{T}}& H(T)
}
\end{equation}
Indeed, this would send the relations of the homotopy colimit to  relations $(iii)$ and $(iv)$ of Construction~\ref{C8}. The left vertical arrow here
is a subcubical inclusion (induced by an inclusion of cubical posets); the right vertical arrow $\alpha^*$ is the precomposition with the map 
of vertices
$\alpha_*\colon V(T)\to V(T')$ induced by the edge contraction~$\alpha$. The cube $|T\downarrow\Psi_k|$ is a union of $d!$ simplices
of dimension~$d$ 
corresponding to the maximal chains in $T\downarrow\Psi_k$. We will define $f_{T}$ on each vertex of the cube, and then extend inside each 
such simplex linearly. Note also that if all the functions $f_{T'}$, $T'\in (T\downarrow\Psi_k)$, $T'\neq T$, are already (compatibly) defined,
then the only freedom to extend this map to a map on the whole cube $|T\downarrow\Psi_k|$ would be the choice of
 $z_T:=f_T(T)\in H(T)$ the image  of the initial element $T$ of this cube.\vspace{5pt}

Choose any collection $z=\{z_{T'}\}$ of points $z_{T'}\in H(T')$, $T'\in \Psi_k$. For any $T\in\Psi_k$ and any $T\xrightarrow{\alpha}T'$, define 
$f^z_T(T'):= \alpha^*(z_{T'})$, and then extend $f^z_T$ on the maximal simplices of $|T\downarrow\Psi_k|$ by linearity. 
As a result, the compatibility of the face inclusions~\eqref{eq:relab_compat} is satisfied, and thus by means of~\eqref{eq:delta2} these $f^z_T$ 
determine a continuous map $\delta_k^z\colon \underset{\Psi_k}{\hocolim}\, \rho_k^M\, \longrightarrow\, \calI b^\Lambda(M)(k)$. 
For $\delta_k^z$ to be a homeomorphism, all the maps $f^z_T\colon |T\downarrow\Psi_k|\to H(T)$, $T\in\Psi_k$, must be homeomorphisms. \vspace{5pt}

For any $T\in\Psi_k$, define $MAX(T)\subset V(T)\setminus\{p\}$ as the set of maximal elements of 
$V(T)\setminus\{p\}$ with respect to the
\lq\lq{}toward the pearl\rq\rq{} partial order. We then define
\begin{equation}\label{eq:z}
z_{T'}\colon V(T')\longrightarrow [0\,,\,1]\,;\quad\, v\longmapsto 
\left\{
\begin{array}{cl}\vspace{4pt}
1-\frac{1}{2^{d(v\,;\,p)}}, & \text{if } v\notin MAX(T'); \\ 
1, & \text{otherwise},
\end{array} 
\right.
\end{equation}
where $d(v\,;\,p)$ is the distance between $v$ and $p$, i.e.  the number of edges in the path between them. 
We claim that for this collection of $z_{T'}\in H(T')$, $T'\in\Psi_k$, each map $f_T:=f^z_T$ is a homeomorphism. The proof is by induction over
the dimension $d$ of the cubes. If $d=0$, i.e. $T$ is a pearled corolla, both $|T\downarrow\Psi_k|$ and $H(T)$ are points and the statement is obvious.
Now, consider any $d\geq 1$. There are two cases.

\begin{figure}[!h]
\begin{center}
\includegraphics[scale=0.17]{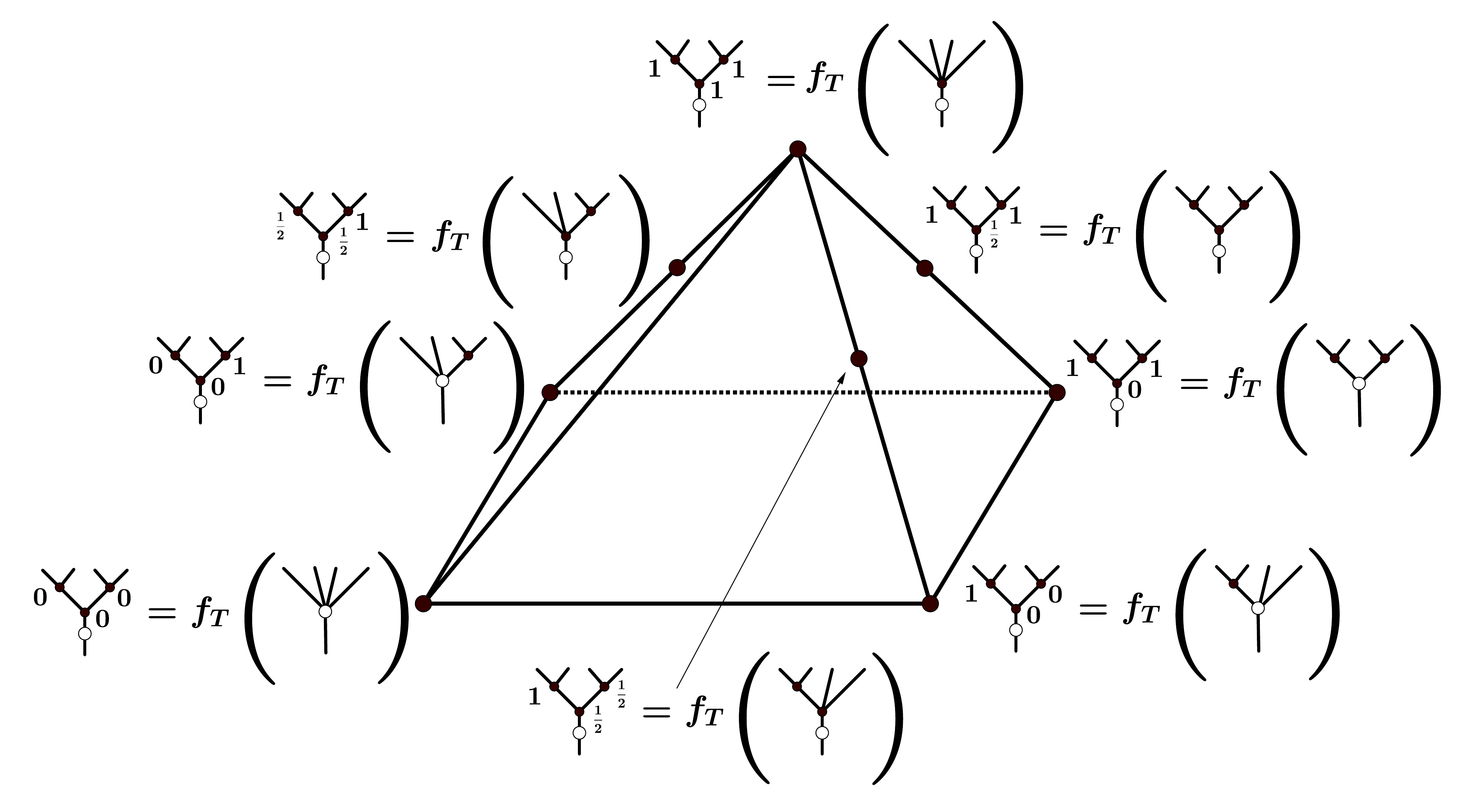}\vspace{-10pt}
\caption{The image of the vertices in the polytope $H(T)$ associated to the cube in Figure \ref{F5}.}\label{F6}\vspace{-10pt}
\end{center}
\end{figure}

{\bf Case 1:} The pearl $p$ of $T$ is adjacent to more than one inner edge.  In this case the polytope $H(T)$ is a product of polytopes
\[
H(T)=\prod_i H(T_i),
\]
where $T_i$ is a pearled tree obtained from $T$ by cutting off all the internal edges adjacent to $p$ (together with all the subtrees not containing $p$)
except one. On the other hand, the cube $|T\downarrow\Psi_k|$ also factors into a similar product (of subcubes this time). Since our formula for
$z_{T'}\colon V(T')\to [0,1]$  uses only the distance to the pearl, the map $f_{T}$ respects this factorization. By induction, for smaller cubes
the map is a homeomorphism, therefore it is so for their product.

 {\bf Case 2:} The pearl $p$ of $T$ is connected to only one inner vertex $v_0$. We call $e:=(p,v_0)$ the {\it distinguished edge} of $T$. 
 In this case $H(T)$ is a pyramide $P$ with base $H(T/e)$, that corresponds to $t_{v_0}=0$, and with the apex $A$ corresponding to $t_{v_0}=1$,
 which forces the constant labelling $(V(T)\setminus\{p\})\to 1$, see Figure~\ref{F6}. By the induction hypothesis, the $(d-1)$-subcube 
 $|T/e\downarrow\Psi_k|\subset |T\downarrow\Psi_k|$ is sent homeomorphically to the base of~$P$. Let $T_1\in T\downarrow\Psi_k$ be the tree
 obtained from $T$ by contracting its all non-distinguished edges. By definition of $f_T$, one has $f_{T}(T_1)=A$ is the apex
 of the pyramide. Each of the remaining $(2^{d-1}-1)$
 trees $T'\in (T\downarrow\Psi_k)$, $ T'\neq T_1$, $T'\notin (T/e\downarrow\Psi_k)$, is sent in the middle of the side edge of $P$ between 
 the apex $A$ and the vertex $f_T(T'/e)$ of the base cube:
 \[
 f_T(T')=\frac 12 (f_T(T'/e)+f_T(T_1)).
 \]
 This formula is a consequence of our choice of $z_{T'}$, see \eqref{eq:z}, and also the fact that the function $f_T(T_1)\colon (V(T)\setminus\{p\})\to [0,1]$ is constantly one.\vspace{5pt}

 To convince ourselves that $f_T$ is a homeomorphism from a cube to a pyramide, one can pull the apex $A$ down along the
 side  edge of constant labels
 (the only one that does not have a vertex from $f_T(T\downarrow\Psi_k)$ in the middle):
 
 $$
 \includegraphics[scale=0.2]{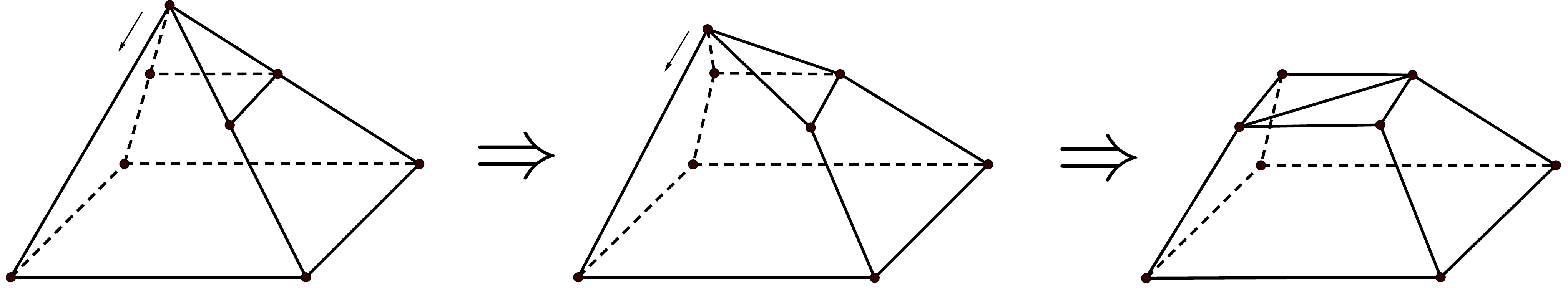}
 $$

For the statement $(i)$ of the theorem, recall that $\partial \calI b^\Lambda(M)(k)$ 
consists of points $[\, T\, ;\, \{a_v\}\, ;\{t_v\}\, ]$ for which at least one $t_v=1$. 
We get that $(i)$  follows from the fact that any tree $T\in\Psi_k$ except the terminal one  $c_k$ has a non-empty set $MAX(T)$ and also from
our choice of~$z$, see~\eqref{eq:z}. To verify~$(ii)$ we note that all the elements in the image of $\delta_{T}$ are labelled by $T$, and one also has
\[
\delta_k\left(\underset{T\downarrow \Psi_k}{\hocolim}\, \rho_k^M\,\right) =\bigcup_{T'\in(T\downarrow\Psi_k)}\operatorname{Im} \delta_{T'}.
\]
\end{proof}

\begin{rmk}\label{r:hocolim}
The constructed homeomorphism $\delta_k$ sends the subspace $$\underset{\Psi_k^U}{\hocolim}\,\rho_k^M
\underset{ \underset{\partial\Psi_k^{UL}}{\hocolim}\,\rho_k^M }{\coprod} 
\underset{\partial\Psi_k^L}{\hocolim}\,\rho_k^M\,,\quad k\geq 2,$$  to the subspace  denoted $\partial'\calI b^\Lambda(M)(k)
\subset \calI b^\Lambda(M)(k)$,
which has all the points in the boundary  $\partial\calI b^\Lambda(M)(k)$ 
plus all the elements labelled by pearled trees
from $\Psi_k^U$. Thus by Proposition~\ref{p:coherent}~$(ii)$, a doubly reduced  operad $O$ is $k$-coherent if and only if for any $i$ in the range 
$2\leq i\leq k$, the inclusion $\partial\calI b^\Lambda(O)(i)\subset \partial'\calI b^\Lambda(O)(i)$ is a weak equivalence.
\end{rmk}

\section{Fulton-MacPherson operad $\calF_m$}\label{s:FM}

The Fulton-MacPherson operad $\mathcal{F}_m$ is a well known model  of the little discs operad $\mathcal{B}_m$. 
In case $m=1$, $\calF_1$ is  the Stasheff $A_\infty$ operad: $\calF_1(k)$ is a disjoint union of $k!$ associahedra $A_{k-2}$~\cite{Stasheff63}, \cite[Theorem~4.19]{Sinha04}.
It appeared first  in a work by Getzler-Jones~\cite{Getzler94} and also in papers   by Kontsevich~\cite{Kontsevich99,Kontsevich03} and by Markl~\cite{Markl99}. Since then it had a countless number of applications
in mathematical physics and topology. Here are some of them~\cite{CFT02,Kontsevich01,Salvatore99,Sinha04,Sinha09,Turchin13,Volic14}.
In this section we provide one more application of this operad, which together with Main Theorem~\ref{th:main} completes the proof of Main Theorem~\ref{th:delooping1}.

\begin{thm}\label{th:FM}
For any $m\geq 1$, the Fulton-MacPherson operad $\calF_m$ is strongly coherent.
\end{thm}

In the first three subsections we prove this theorem. 
In the last subsection we outline an alternative construction of a Reedy cofibrant replacement
of $\calF_m$ as a bimodule over itself. We denote it by $\calBF_m$. It is used in Subsection~\ref{ss:FM5},
where we explain the initial geometrical idea for the delooping result. In the following, we recall the main properties of $\calF_m$. We refer the reader to~\cite[Section~5]{Volic14}, where this operad and the conventions below about infinitesimal configurations are explained in full detail. 
\begin{itemize}
\item[$\blacktriangleright$] $\calF_m(0)=\calF_m(1)=\ast$ and $\calF_m(k)$, with $k\geq 2$, is a manifold with corners whose interior is the configuration space $C(k\,;\,\mathbb{R}^{d})$, of $k$ distinct points in $\mathbb{R}^{d}$, quotiented out by translations and rescalings. 
\item[$\blacktriangleright$] The strata of $\calF_m(k)$, with $k\geq 2$, are encoded by non-planar rooted trees with $k$ leaves (labelled by $1,\ldots, k$) whose all internal vertices are of valence $\geq 3$. The codimension of the stratum encoded by a tree $T$ is $|V(T)|-1$. The closure of the stratum encoded by $T$ is homeomorphic to $\prod_{v\in V(T)}\calF_m(|v|)$. We denote this closure by $\calF_m(T)$. It is a manifold with corners whose interior is the stratum we started with.
\item[$\blacktriangleright$] By construction, $\calF_m$ is an operad in compact semi-algebraic sets, see~\cite{Volic14}.
\end{itemize}

\begin{figure}[!h]
\begin{center}
\includegraphics[scale=0.18]{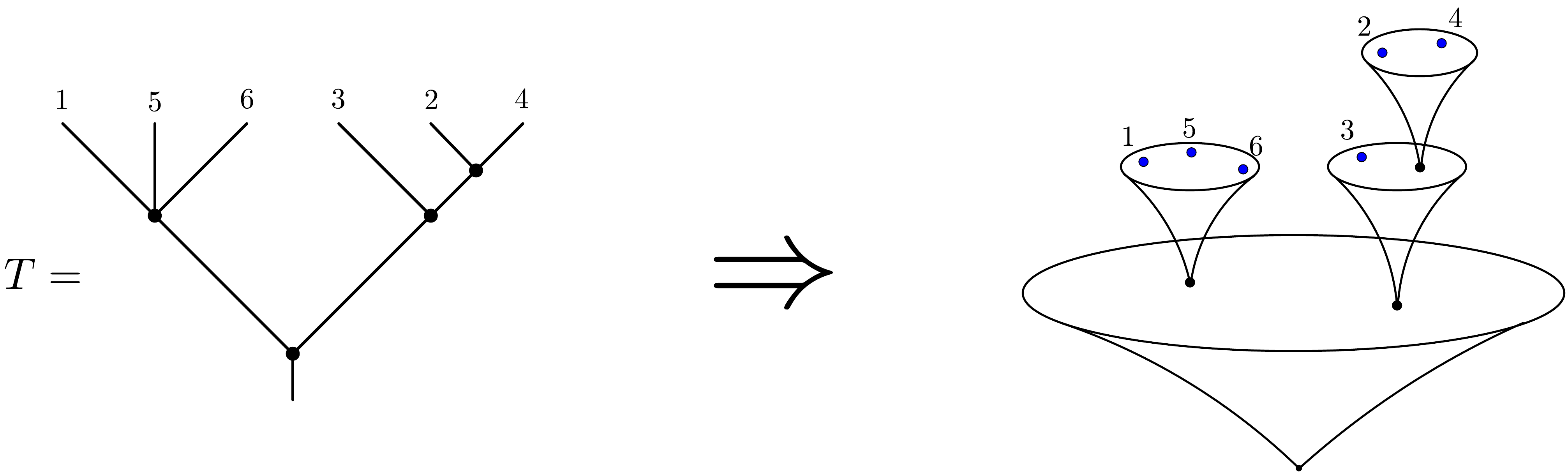}\vspace{-5pt}
\begin{minipage}{420pt}
\caption{A tree $T$ together with an element in the interior of $\calF_m(T)$. It represents an infinitesimal configuration in a stratum of $\calF_m(6)$:  points $1$, $5$, and $6$ collided together, and so did $2$, $3$, and $4$. But the distance between the points $2$ and $4$ is infinitesimally small compared to the distance between $2$ and $3$.}\label{Fig:FM}
\end{minipage}\vspace{-10pt}
\end{center}
\end{figure}


%
%

In order to prove that the Fulton-MacPherson operad is strongly coherent, we consider an alternative  cofibrant replacement of the Fulton-MacPherson operad in the category $\Lambda\Ibimod_{\calF_m}$ that uses compactification of configuration spaces. Similarly to the Fulton-MacPherson operad, the components of the infinitesimal bimodule so obtained are manifolds with corners. By studying the geometry of its stratification, we  prove Theorem~\ref{th:FM}. 

\subsection{Infinitesimal bimodule $\calIF_m$}\label{ss:IF}

In this subsection, we describe a different model of a  Reedy cofibrant replacement of $\calF_m$ as an infinitesimal bimodule over $\calF_m$. We denote this model by $\calIF_m$. In case $m=1$, the component $\calIF_1(k)$ is  a disjoint union of $k!$ cyclohedra $C_k$~\cite{Bott94,Lambrechts102}.  Similarly to the Fulton-MacPherson operad $\calF_m$, the components of $\calIF_m$ are also obtained as Axelrod-Singer-Fulton-MacPherson compactifications of configuration spaces (see~\cite{Sinha04}), but this time we don't quotient out by any kind of translations or rescalings. More precisely, we view those spaces as configuration spaces of $(k+1)$ points in $S^m=\mathbb{R}^{m}\cup \{\infty\}$, where the $(k+1)$st point is fixed to be $\infty$, and then we take the compactification in question. \vspace{5pt}

Explicitly, $\calIF_m$ is described in the work of the second author~\cite{Turchin13}. In order to see that the obtained spaces have the structure of an infinitesimal bimodule over $\calF_m$, one should parametrize the strata from the "flat" point of view. That is, instead of looking at how points approach $\infty$, one should look at how the points are located with repect to each other when they escape to infinity (for details see~\cite{Turchin13}). Here are the main properties of $\calIF_m$.
\begin{itemize}
\item[$\blacktriangleright$] Each $\calIF_m(k)$, with $k\geq 0$, is a smooth manifold with corners, whose interior is $C(k\,;\,\mathbb{R}^{m})$. 
\item[$\blacktriangleright$]  The strata of $\calIF_m(k)$ are encoded by non-planar pearled trees with $k$ labelled leaves. Their pearl could be of any valence $\geq 1$, while all the other vertices must be of valence $\geq 3$. The codimension of a stratum encoded by a pearled tree $T$ is $|V(T)|-1$.
\item[$\blacktriangleright$] The closure of the stratum encoded by $(T\,,\,p)$, denoted by $\calIF_m(T)$, is homeomorphic to 
$$
\calIF_m(|p|)\times \prod_{v\in V(T)\setminus \{p\}}\calF_m(|v|).$$
\end{itemize}
Another important property for us is that $\calIF_m$ is an infinitesimal bimodule over $\calF_m$ in the category of semi-algebraic compact spaces. 
Figure~\ref{Fig:FM2} gives an example of an infinitesimal  configuration in which points $2$, $6$, $7$, $1$, $8$ escape to infinity; the points $3$, $4$ and $5$ stay inside $\mathbb{R}^{m}$ (the shaded disc in the figure corresponds to the pearl of the tree), but $4$ and $5$ collide. Also while escaping to infinity $2$, $6$, $7$ stay close to each other, and the same happen with the pair $1$ and $8$.
\begin{figure}[!h]
\begin{center}
\includegraphics[scale=0.17]{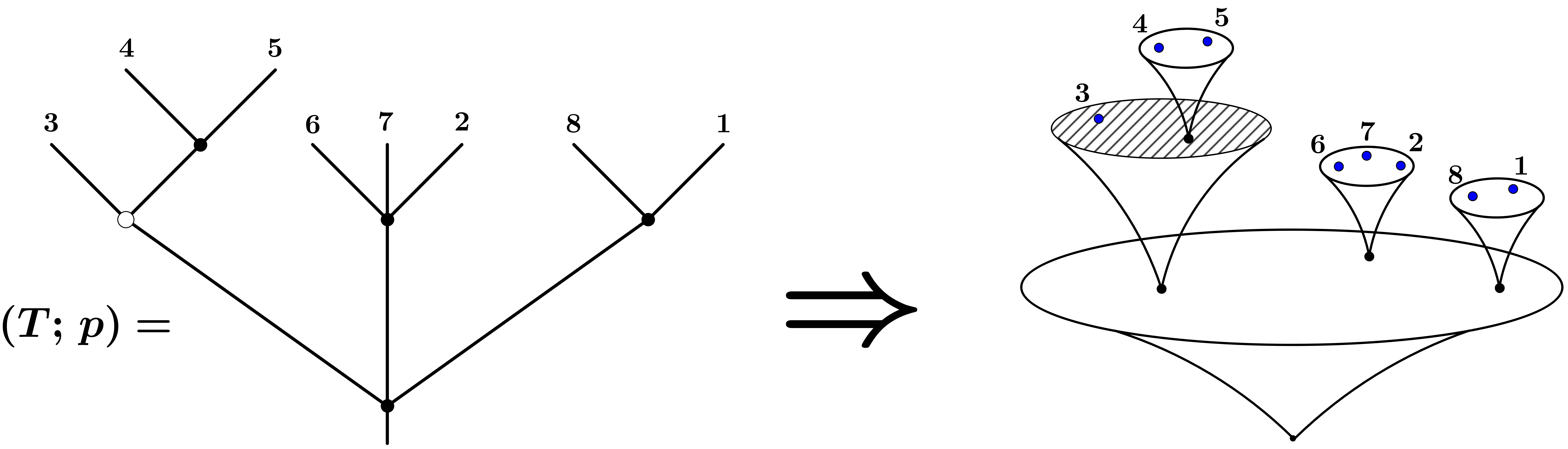}
\caption{A pearled tree $(T\,;\,p)$ together with an element in the interior of $\calIF_m(T)$.}\label{Fig:FM2}\vspace{-19pt}
\end{center}
\end{figure}

\begin{rmk}\label{r:FM1}
We believe that $\calIF_m$ and $\calI b^\Lambda(\calF_m)$ are homeomorphic as  $(\calF_m)_{>0}$-Ibimodules. Example~\ref{ex:2FM}
and the projection  that we study in the next subsection 
show that this is the case for the arities $k\leq 2$. If true, this would be similar to the fact that the Boardman-Vogt resolution of $(\calF_m)_{>0}$
as operad is homeomorphic to $(\calF_m)_{>0}$, see~\cite[Proposition~3.7]{Salvatore99}.
\end{rmk}

%
%
%

\subsection{Projection map $\calIF_m\rightarrow \calF_m$}

In this subsection we want to understand a little bit better the geometry of $\calIF_m$. The interior of $\calIF_m$ is $C(k\,;\,\mathbb{R}^{m})$. 
It turns out that the natural quotient map of interiors can be extended to a smooth map 
$$
\mu_\calI:\calIF_m(k)\longrightarrow \calF_m(k), \hspace{15pt}\text{with }k\geq 0.
$$
These maps turn out to respect the infinitesimal bimodule structure over  $\calF_m$ making $\calIF_m$ into 
a Reedy cofibrant replacement of $\calF_m$ as $\calF_m$-Ibimodule. Moreover, by construction, these maps are semi-algebraic. In the following subsection, we will be using not only the statement of Proposition~\ref{G5}, but also the description of the stratified structure of the preimages given in the proof. 

\begin{pro}\label{G5}
Let $\mu_\calI:\calIF_m(k)\longrightarrow \calF_m(k)$ be the map described above. For any point $x\in \calF_m(k)$, with $k\geq 2$, the preimage $\mu_\calI^{-1}(x)$ is an $(m+1)$-dimensional disc.
\end{pro}

\begin{proof}
Let $T$ be a tree encoding a stratum of $\calF_m(k)$, with $k\geq 2$. Then, the preimage under $\mu_\calI$ of this stratum consists of the strata of $\calIF_m(k)$ indexed by  pearled trees $(T'\,;\,p)$ of one of the four types:
\begin{itemize}
\item[$\blacktriangleright$] Type I. Tree $(T'\,;\,p)$ is obtained by putting a pearl in one of the vertices of~$T$.
\begin{center}
\includegraphics[scale=0.3]{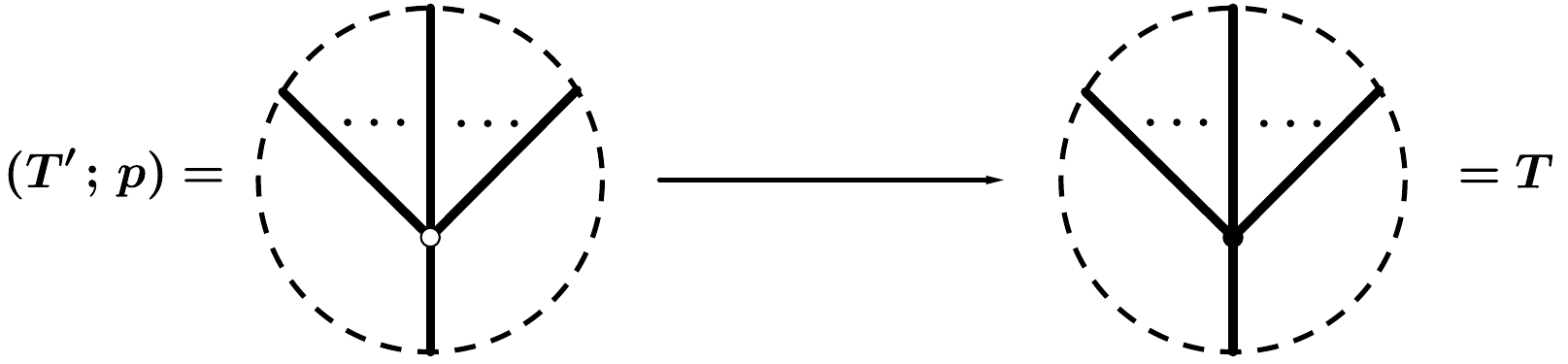}
\end{center}
\item[$\blacktriangleright$] Type II. Tree $(T'\,;\,p)$ is obtained by putting a pearl in the middle of one of the edge of $T$.
\begin{center}
\includegraphics[scale=0.3]{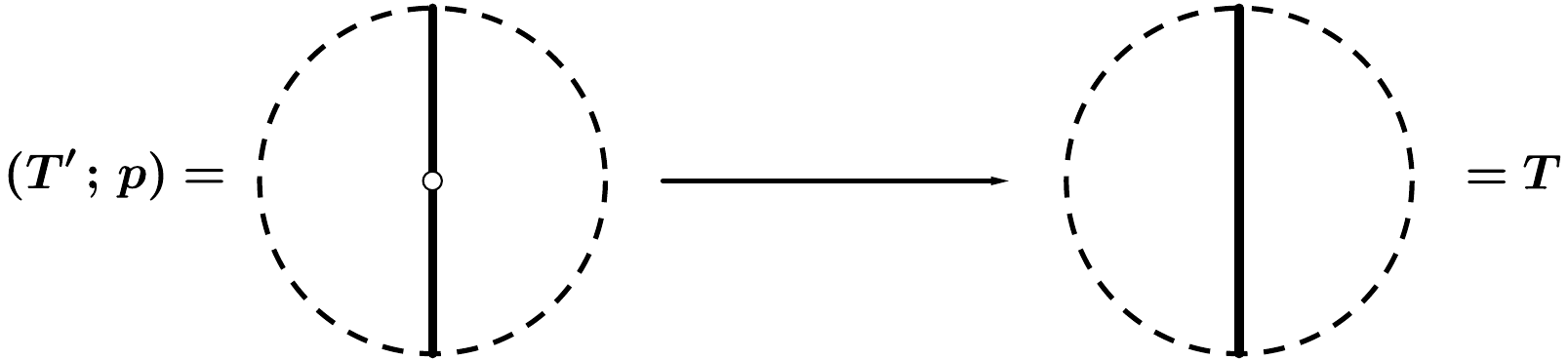}
\end{center}
\item[$\blacktriangleright$] Type III. Tree $(T'\,;\,p)$ is obtained  by attaching a univalent pearl to a vertex of $T$.
\begin{center}
\includegraphics[scale=0.3]{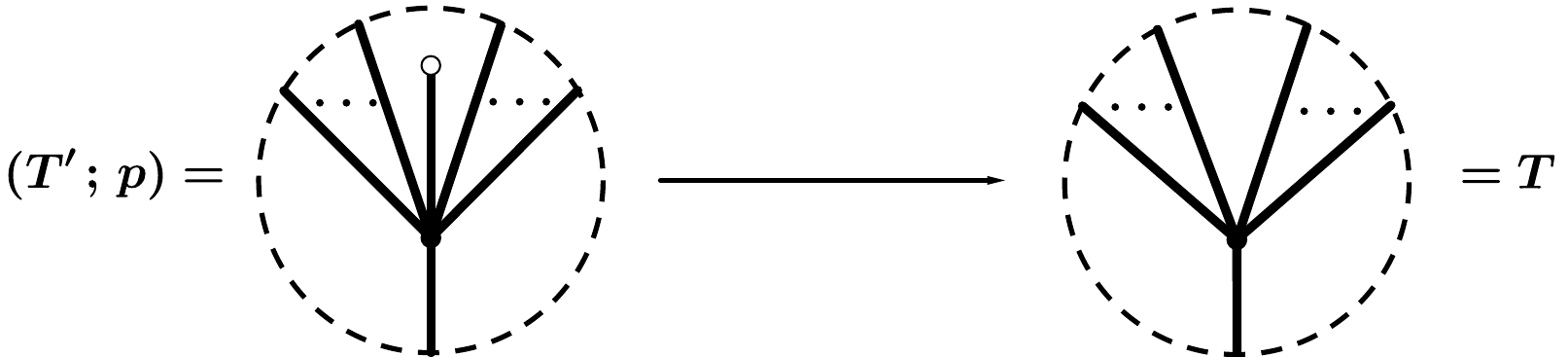}
\end{center}
\item[$\blacktriangleright$] Type IV. Tree $(T'\,;\,p)$ is obtained by attaching a univalent pearl to the middle of an edge of $T$.
\begin{center}
\includegraphics[scale=0.3]{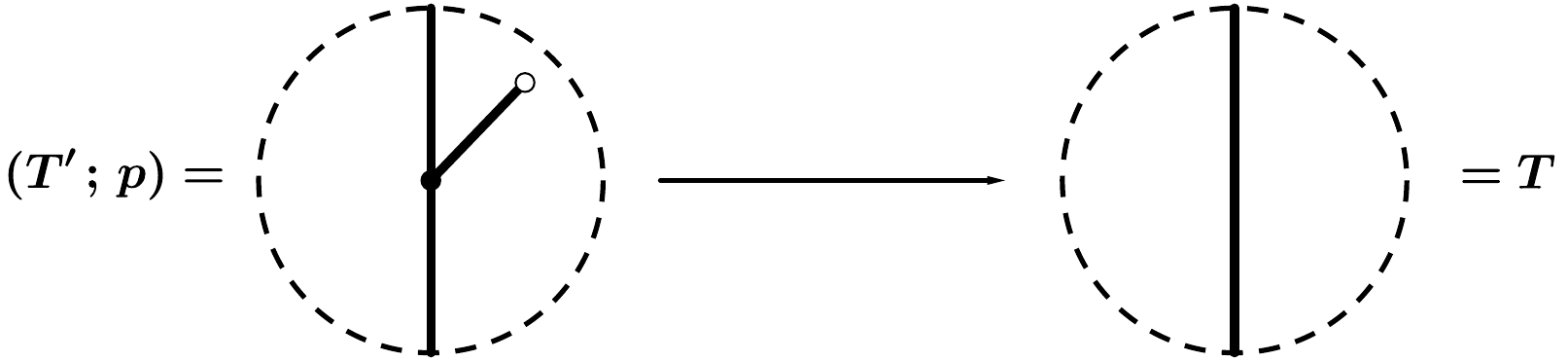}
\end{center}
\end{itemize}
The figures above show how locally $(T'\,;\,p)$ differs from $T$. Outside the dotted circles $(T'\,;\,p)$ and $T$ are identical. \vspace{5pt}

Consider a point $x\in \calF_m(k)$, with $k\geq 2$. We want to analyze the contribution to the preimage $\mu_\calI^{-1}(x)$ of the strata above and show that these contributions glue together into a closed $(m+1)$-disc. To start with, let us consider  $x$  in the interior of $\calF_m(k)$, or in other words let it lie in the stratum encoded by the $k$-corolla:

\begin{center}
\includegraphics[scale=0.4]{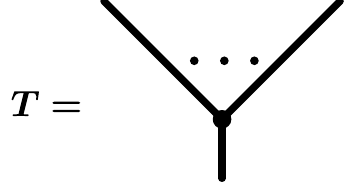}.
\end{center}

\noindent The part of the preimage of $x$ of Type I is the part in the interior of $\calIF_m(k)$, labelled by the pearled $k$-corolla. This part is an open $(m+1)$-ball (it is the set $\mathbb{R}^{m}\times \mathbb{R}_{+}$ of translations and rescalings). This part is exactly the interior of the preimage disc. The Type II preimage lies in the strata labelled by the pearled trees  
\begin{center}
\includegraphics[scale=0.4]{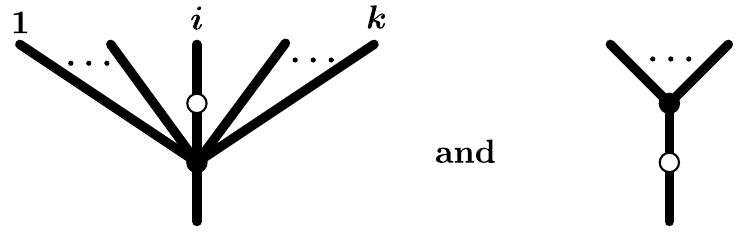}.
\end{center}
There are $k+1$ such trees. The preimage corresponding to each of these trees is homeomorphic to the interior of $\calIF_m(1)$, which is an open $m$-disc. Their closures are $k+1$ disjoint closed $m$-discs, whose boundary spheres are the Type IV preimages corresponding respectively to the pearled trees
\begin{center}
\includegraphics[scale=0.4]{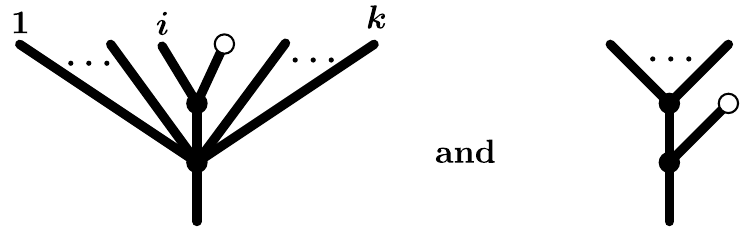}.
\end{center}
We are left to analyze the Type III preimage, corresponding to the pearled tree
\begin{center}
\includegraphics[scale=0.4]{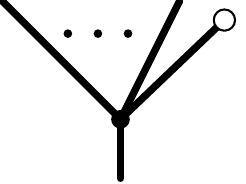}.
\end{center}
Note that this preimage is the same as the interior part of the preimage under the map $\calF_m(k+1)\rightarrow \calF_m(k)$ that forgets the last point, and is $\mathbb{R}^{m}$ minus $k$ points. The compactification of this part is a closed $m$-disc minus $k$ disjoint open sub-discs. There are $k+1$ spheres in the boundary, which are exactly the Type IV preimages. Note that the Type II, III, and IV preimages all together form an $m$-sphere. On the other hand, by using the methods developed by D.Sinha in~\cite{Sinha04}, one can show that $\mu_\calI^{-1}(x)$ is
a manifold with corners, whose boundary is a sphere stratified as above. This manifold with corners must be contractible since any topological manifold with boundary is homotopy equivalent to its interior.  But then, by the topological $h$-cobordism theorem, which is true for all dimensions, $\mu_\calI^{-1}(x)$ must be homeomorphic to a closed disc. (The topological $h$-cobordism theorem is a combination of many deep results in the theory of manifolds due to Smale, Freedman and Perelman. For our purposes we only need that these preimages are contractible, which as we have seen follows from the fact that their interiors are contractible being open $(m+1)$-balls.)
\begin{figure}[!h]
\begin{center}
\includegraphics[scale=0.15]{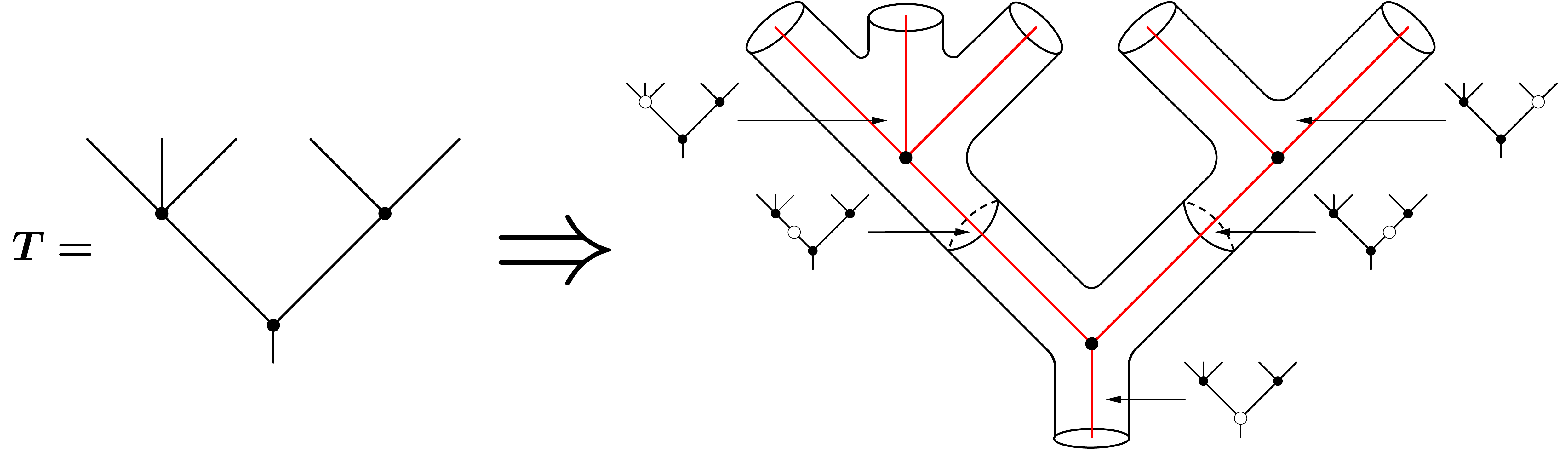}
\caption{The preimage $\mu_\calI^{-1}(x)$ with $x$ in the stratum indexed by $T$.}\label{Fig:preimage}\vspace{-15pt}
\end{center}
\end{figure}

Now, consider the general case $x\in \calF_m(k)$, where $x$ lies in the stratum indexed by any tree $T$. We claim that $\mu_\calI^{-1}(x)$ is a union of 
$|V(T)|$ closed $(m+1)$-discs glued with each other along $m$-discs lying in their boundary spheres, so that $\mu_\calI^{-1}(x)$ looks like a tubular neighborhood of the tree $T$ embedded in $\mathbb{R}^{m+1}$. The interior of those $(m+1)$-discs are Type I parts of the preimage (one disc for each pearled Type I tree $(T'\,;\,p)$). The $m$-discs are Type II preimages. Their boundary $(m-1)$-spheres are Type IV preimages. The parts of the boundaries of $(m+1)$-discs, that look like $m$-discs with holes, are Type III preimages. Figure~\ref{Fig:preimage} gives an example of a tree $T$ and how the corresponding preimage $\mu_\calI^{-1}(x)$ looks like. 
\end{proof}


\begin{rmk}
It was shown in~\cite{Volic14} that the map $\pi\colon\calF_m(k+1)\rightarrow \calF_m(k)$, that forgets one of the points, is a semi-algebraic fiber bundle, whose fibers are closed $m$-discs with $k$ holes.
We believe that our map $\mu_\calI:\calIF_m(k)\longrightarrow \calF_m(k)$ is also a semi-algebraic fiber bundle (that contains $\pi$ as a subbundle). Moreover, we expect that this bundle is  topologically trivial. 
\end{rmk}

\subsection{Fulton-MacPherson operad is strongly coherent}

This subsection is devoted to the proof of Theorem~\ref{th:FM}. Recall definitions and notation from Subsection~\ref{s:coherent}. 
 Because the projection $\mu_\calI:\calIF_m\longrightarrow \calF_m$ is a weak equivalence of infinitesimal bimodules, it induces an equivalence of diagrams $\rho_k^{\calIF_m}\to\rho_k^{\calF_m}$, $k\geq 0$. Thus it is enough for us to prove strong coherence of the diagrams $\rho_k^{\calIF_m}$,
 $k\geq 2$. We note also that for any pearled tree $T\in\Psi_k$, the functor $\rho_k^{\calIF_m}$ assigns a space, which is nothing but the
 stratum of $\calIF_m$ labelled by~$T$:
 \[
 \rho_k^{\calIF_m}(T)=\calIF_m(T).
 \]
%
%
%

\begin{pro}\label{p:FM}
The following natural map
$$
\underset{\partial\Psi_k^U}{\hocolim}\, \rho_k^{\calIF_m}\, \longrightarrow 
\underset{\partial\Psi_k^U}{\operatorname{colim}}\, \rho_k^{\calIF_m}
$$
is a weak equivalence.
\end{pro}

\begin{proof}
Let $C$ be a small category. We consider the category $\Topo^{C}$ of diagrams of shape $C$ with its {\it projective model structure}, in which weak equivalences and fibrations are, respectively, objectwise weak equivalences and objectwise fibrations~\cite[Section~11.6]{Hirschhorn03}. In case $F:C\rightarrow\Topo$ is cofibrant, the natural map 
$$
\underset C\hocolim \, F\longrightarrow \underset C{\operatorname{colim}}\, F
$$
is a weak equivalence \cite[Theorem~11.6.8~(2)]{Hirschhorn03}. To prove our proposition we need to show that $\left. \rho_k^{\calIF_m}\right|_{\partial \Psi_k^U}$
 is a cofibrant diagram. 

An example of a cofibrant diagram is a functor of the form $X\times \hom_{C}(c\,;\,-)$, where $c\in C$  and $X$ is a cofibrant space. The cofibrancy (left lifting property) follows from the fact (Yoneda lemma) that
$$
\Nat_{C}(X\times \hom_{C}(c\,;\,-) \,,\,F)=\Map(X\,;\,F(c)).
$$ 
In case $X\hookrightarrow Y$ is a cofibration of spaces, one similarly gets that the natural transformation 
$$
X\times \hom_{C}(c\,;\,-)\longrightarrow Y\times \hom_{C}(c\,;\,-) 
$$
is a cofibration in $\Topo^{C}$. More generally, since a pushout of a cofibration is a cofibration, we get that the right vertical map in any pushout diagram of the form 
\begin{equation}\label{eq:pushout}
\xymatrix{
X\times \hom_{C}(c\,;\,-) \ar[r] \ar@{^{(}->}[d] & F \ar@{^{(}->}[d] \\
Y\times \hom_{C}(c\,;\,-) \ar[r] & F'
}
\end{equation}
is a cofibration. Here $F\in\Topo ^{C}$ is any diagram of shape $C$, and the upper horizontal map is obtained from some $f:X\rightarrow F(c)$ applying the Yoneda lemma. This type of a map $F\hookrightarrow F'$ will be called \textit{cellular inclusion}. In what follows, a diagram of shape $C$ is said to be \textit{cellular cofibrant} if it is obtained from the constantly empty diagram by a finite sequence of cellular inclusions. 

Obviously, being cellular cofibrant implies being cofibrant. Let us now show that  $\left. \rho_k^{\calIF_m}\right|_{\partial \Psi_k^U}$
 is cellular cofibrant. To do it, we show first that $\rho_k^{\calIF_m}\in\Topo^{\Psi_k}$ is cellular cofibrant. For this purpose, we take the "skeleton filtration"
$$
\emptyset \subset F_{(kd-k)}\rho_k^{\calIF_m} \subset F_{(kd-k+1)}\rho_k^{\calIF_m}\subset \cdots \subset F_{(kd)}\rho_k^{\calIF_m} = 
\rho_k^{\calIF_m},
$$
where $F_{(j)}\rho_k^{\calIF_m}$ assigns to a planar pearled tree $T$ the union of strata in $\calIF_m(T)$ of dimension $\leq j$. Then each inclusion
 $F_{(kd-i-1)}\rho_k^{\calIF_m} \hookrightarrow F_{(kd-i)}\rho_k^{\calIF_m}$ can be seen as a sequence of cell inclusions: we consecutively attach strata
  labelled by  trees $T$ with $i$ internal vertices different from the pearl. For such cell inclusion (see~\eqref{eq:pushout}), the corresponding cofibration $X\hookrightarrow Y$ in this case is $\partial \calIF_m(T)\hookrightarrow\calIF_m(T)$ (recall that $\calIF_m(T)$ is a manifold with corners, and $\partial\calIF_m(T)$ is its boundary). The element $c$ is the tree $T$ itself, and the corresponding attachment map $X\to F(c)$ is the identity map.

\begin{lmm}\label{l:FM}
In case $T=c_k$ or $c'_k$, the poset $T\downarrow\partial\Psi_k^U$ is empty; for any other $T\in\Psi_k$, the poset $T\downarrow\partial\Psi_k^U$ 
has an initial element.
\end{lmm}

\begin{proof}[Proof of Lemma~\ref{l:FM}]
Indeed, $c_k$ is the terminal element and does not have  any morphism to $\partial\Psi_k^U$. Similarly with $c'_k$: besides the identity map, it only has a morphism to $c_k$.   If $T\in \partial \Psi_k^U$,   the initial element is $T$ itself. Finally, for $T\notin\Psi_k^U$  and $T\neq c'_k$, the 
poset $T\downarrow\partial\Psi_k$ is non-empty, but it has an initial element as shown in Figure~\ref{G3}.

\begin{figure}[!h]
\begin{center}
\vspace{5pt}
\includegraphics[scale=0.4]{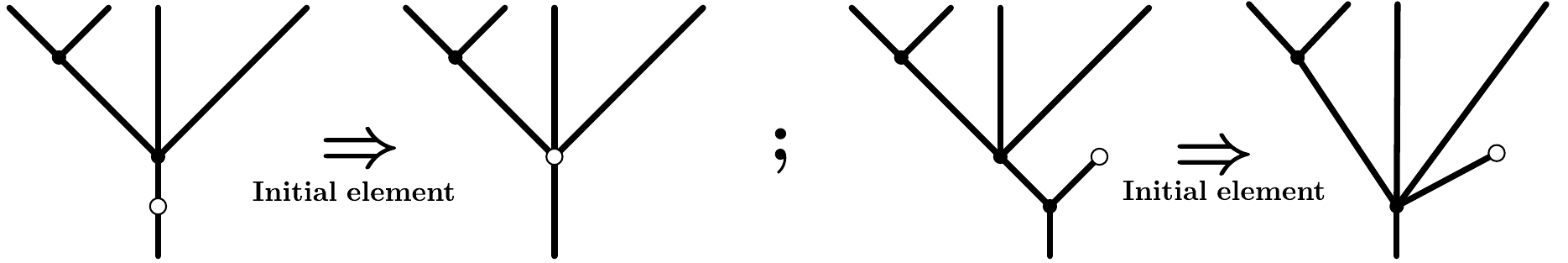}
\caption{Examples of planar pearled trees together with their initial elements.}\label{G3} \vspace{-25pt}
\end{center}
\end{figure}

\end{proof}

 This property implies that any functor of the form $X\times \hom_{\Psi_{k}}(T\,; - )$ when restricted on $\partial\Psi_k^U$ is either constantly empty (in case $T$ is $c_k$ or $c'_k$) or is $X\times \hom_{\partial\Psi_k^U}(T_{min}\,;-)$, where $T_{min}$ is the initial element of the set $T\downarrow
 \partial\Psi_k^U$. As a consequence, any cellular cofibrant $G\in\Topo^{\Psi_{k}}$ remains cellular cofibrant when restricted on $\partial\Psi_k^U$. This proves our proposition.
\end{proof}

From now on, it is easy to see that $\underset{\partial\Psi_k^U}{\operatorname{colim}}\, \rho_k^{\calIF_m}$, denoted by $A(k)$ for shortness, is a subspace of $\calIF_m(k)$, which, as a consequence of Lemma~\ref{l:FM}, is the union of all its (open) strata except  those labelled by $c_k$ and $c'_k$. The first one is the interior of $\calIF_m(k)$, and the second one is part of its boundary. To finish the proof of the strong coherence we are left to show that the inclusion $A(k)\subset \calIF_m(k)$
is an equivalence, or, equivalently, the projection $\mu_\calI|_{A(k)}\colon A(k)\to \calF_{m}(k)$ is one.
 The latter map is a semi-algebraic map between compact spaces and, as we will see,  has contractible preimages,
which implies that the map is an equivalence~\cite[Corollary~1.3]{Lacher69} and~\cite[Corollary~2]{Dydak88}. Indeed, $\mu_\calI$ sends the open strata encoded by $c_k$ and $c'_k$ in the interior of $\calF_m(k)$. Thus for $x\in\partial \calF_m(k)$, $\mu_\calI^{-1}(x)\subset A(k)$ and is a closed $(m+1)$-disc by
Proposition~\ref{G5}. On the other hand, if $x$ is in the interior of $\calF_m(k)$, then  $\mu_\calI^{-1}(x)$ is a manifold with corners which looks like a tubular neighborhood of a $k$-corolla in $\mathbb{R}^{m+1}$:

\begin{center}
\includegraphics[scale=0.15]{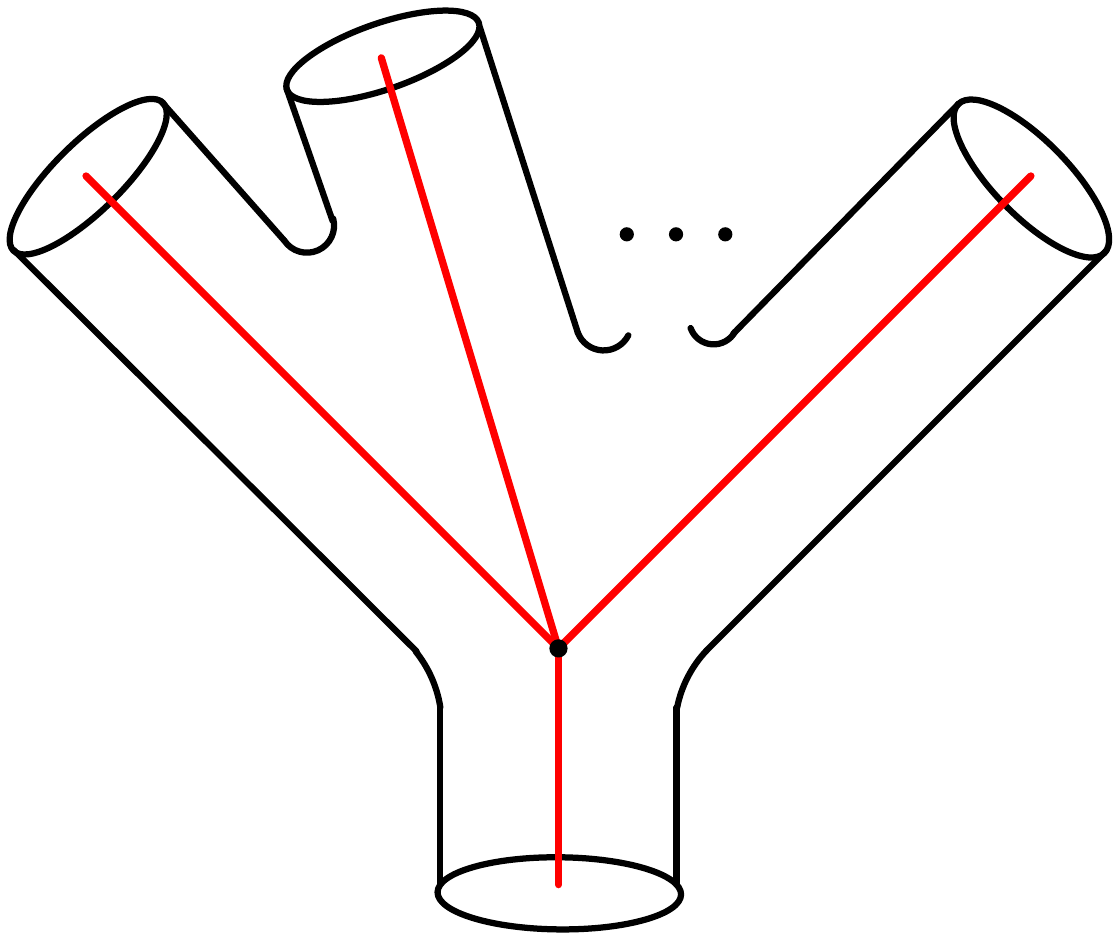}
\end{center}

\noindent See the proof of Proposition \ref{G5}. The part lying in the stratum encoded by $c_k$ is the interior, and the part encoded by $c'_k$ is the bottom (open) $m$-disk. Thus, $\mu_\calI^{-1}(x)\cap A(k)$ is homeomorphic to a closed $m$-disc, and therefore is contractible. We conclude that $\mu_\calI|_{A(k)}$ is a weak equivalence. This finishes the proof of Theorem~\ref{th:FM}.

\subsection{Bimodule $\mathcal{BF}_m$}\label{ss:BF}

As we mentioned in the introduction, our initial idea for the proof of Main Theorem~\ref{th:delooping1} was to use geometrical cofibrant replacements
of $\calF_m$ as an infinitesimal bimodule and as a bimodule over itself. This geometrical replacement as Ibimodule was described in Subsection~\ref{ss:IF} and 
it was  used for the proof of the strong coherence of $\calF_m$. For the completeness of the picture, in this subsection we outline the idea of the construction of the
geometrical
bimodule replacement $\calBF_m$ of $\calF_m$. For $m=1$, the component $\calBF_1(k)$ is expected to be a 
disjoint union of $k!$ multiplihedra $M_{k-1}$~\cite{Stasheff63}.  Similarly to $\calIF_m$ which is conjectured to be
 homeomorphic to $\calI b^\Lambda(\calF_m)$ as an $(\calF_m)_{\geq 0}$-Ibimodule, and also similarly to the operad
$(\calF_m)_{>0}$ itself, which is known to be homeomorphic to its operadic Boardman-Vogt replacement~\cite{Salvatore99}, we conjecture that $\calBF_m$ is homeomorphic to $\calB^\Lambda(\calF_m)$ as an $(\calF_m)_{>0}$-bimodule, implying that all the componens
$\calB^\Lambda(\calF_m)(k)$, $k\geq 0$, are naturally topological manifolds with boundary.\vspace{5pt}

The  components $\mathcal{BF}_m(k)$, with $k\geq 1$,  of this hypothetical bimodule should be obtained as compactification of the spaces $C(k\,;\,\mathbb{R}^{m})$ quotiented out by translations only. For example, $\calBF_m(1)=*$ and $\calBF_m(2)=S^{m-1}\times [0,{+}\infty]\cong
\calF_m(2)\times [-1,1]$, compare with Example~\ref{ex:1compl}. Here are the main properties that one expects:
\begin{itemize}
\item[$\blacktriangleright$] $\calBF_m(0)=*$ and  $\mathcal{BF}_m(k)$, with $k\geq 1$, is a smooth manifold with corners, whose interior is $C(k\,;\,\mathbb{R}^{m})$ quotiented out by translations. 
\item[$\blacktriangleright$]  The strata of $\mathcal{BF}_m(k)$, with $k\geq 1$, are encoded by non-planar trees with section having $k$ labelled leaves, whose pearls have arity $\geq 1$, and all the other vertices are of arity $\geq 2$. The codimension of the stratum encoded by a tree with section $(T\,;\,V^{p}(T))$ is $|V(T)\setminus V^{p}(T)|$.
\item[$\blacktriangleright$] The closure of the stratum encoded by $(T\,,\,V^{p}(T))$, denoted by $\mathcal{BF}_m(T)$, is homeomorphic to 
$$
\prod_{p\in V^{p}(T)}\mathcal{BF}_{m}(|p|)\times \prod_{v\in V(T)\setminus V^{p}(T)}\mathcal{F}_{m}(|v|).
$$
\end{itemize}
Similarly to the infinitesimal bimodule case, one should expect that $\mathcal{BF}_{m}$ is a bimodule over $\mathcal{F}_{m}$ in the category of semi-algebraic compact spaces. As example, Figure~\ref{fig:bf}  shows a limit configuration in which points $1$, $2$, $3$ have finite non-zero distance to each other. Furthermore, points $4$ and $5$ collided whereas the distance between $1$ and $4$ is infinity. 
\begin{figure}[!h]
\begin{center}
\includegraphics[scale=0.2]{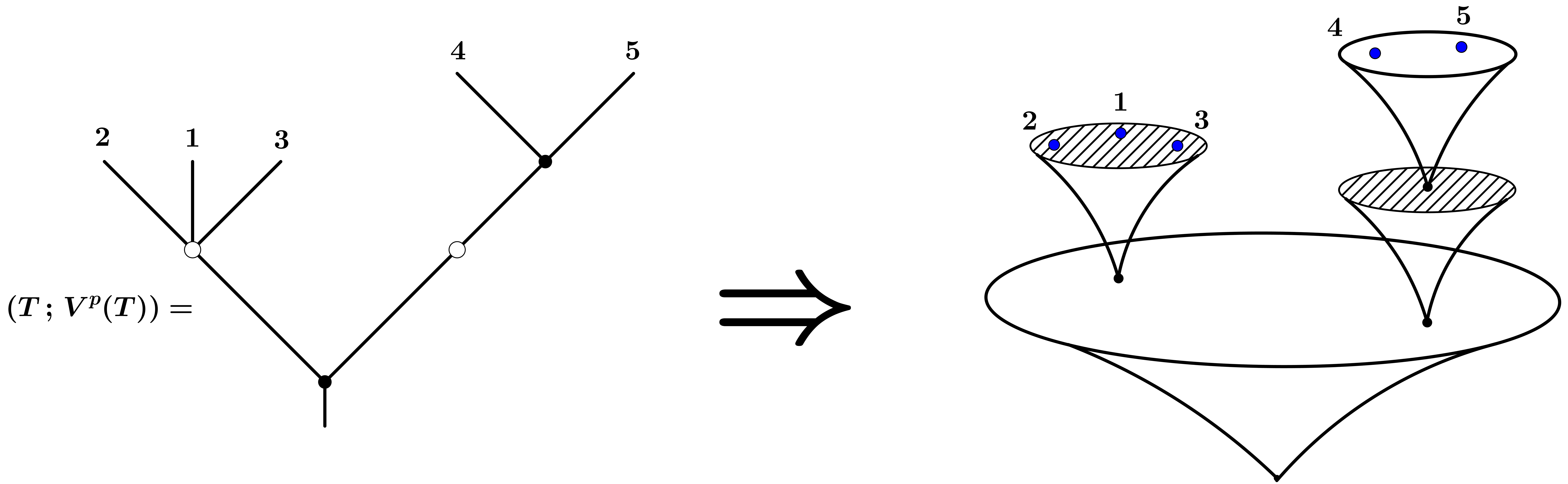}
\caption{A tree with section $(T\,;\,V^{p}(T))$ together with an element in the interior of $\mathcal{BF}_m(T)$.}\label{fig:bf}\vspace{-10pt}
\end{center}
\end{figure}

One expects that the obvious projections
\[
C(k,\R^m)\big/\text{(translations)} \longrightarrow C(k,\R^m)\big/\text{(translations $+$ rescalings)}.
\]
extend to maps
\[
\mu_B\colon\calBF_m(k)\to \calF_m(k)
\]
that respect the $\calF_m$-bimodule structure and thus turn $\calBF_m$ into a Reedy cofibrant replacement of $\calF_m$ as a bimodule over itself.

\section{Map between the towers}\label{S:MAP}

The goal of this section is to construct a map between the towers from Main Theorem~\ref{th:main}. 
The first two subsections are heavily combinatorial. If the reader at any moment feels bored, lost,
or annoyed, they are strongly encouraged to look at the last and short Subsection~\ref{ss:FM5},
which gives a geometrical insight for the constructions of Subsections~\ref{s:alt}-\ref{s:map}.

\subsection{An alternative cofibrant replacement of $O$ as $O$-Ibimodule}\label{s:alt}

In Section~\ref{s:cof}, we give a functorial way to get cofibrant replacements  in the categories  of bimodules and infinitesimal bimodules over an operad $O$.  This allows us to understand the mapping 
spaces of both towers in~\eqref{eq:equiv_general1}. However, in order to get an explicit map between the 
two towers, we need a refined version of this replacement for $O$ as $O$-Ibimodule. 
The replacements $\overline{\calI b}{}^\Sigma(O)$ and $\overline{\calI b}{}^\Lambda(O)$ that we produce in this section topologically  are not 
very much different from the replacements ${\calI b}^\Sigma(O)$ and ${\calI b}^\Lambda(O)$, see for example Theorem~\ref{H5}, but combinatorially they are 
more involved  being split in much more smaller pieces. \vspace{5pt}

First, we give a general construction.

\begin{const}\label{D7}
Let $N$ be an $O$-Ibimodule, and $M$ an $O$-bimodule endowed with an $O$-bimodules
 map $\mu\colon M\rightarrow O$ and a basepoint  $\ast_1^{M}\in M(1)$, such that 
 $\mu(\ast_1^M)=*_1\in O(1)$. From these data, we will construct an $O$-Ibimodule $N{\ltimes}M$.
 The sequence $N{\ltimes}M$ is defined as a quotient of $N{\circ_O} M$, see Notation~\ref{not:circO}. To recall
 $N{\circ_O} M$ itself is a quotient of $N\circ M$ by
 \begin{itemize} 
 \item[$i)$]  the relation equalizing the right $O$-action on $N$ and the left $O$-action on $M$:\vspace{-1pt}
$$
\includegraphics[scale=0.4]{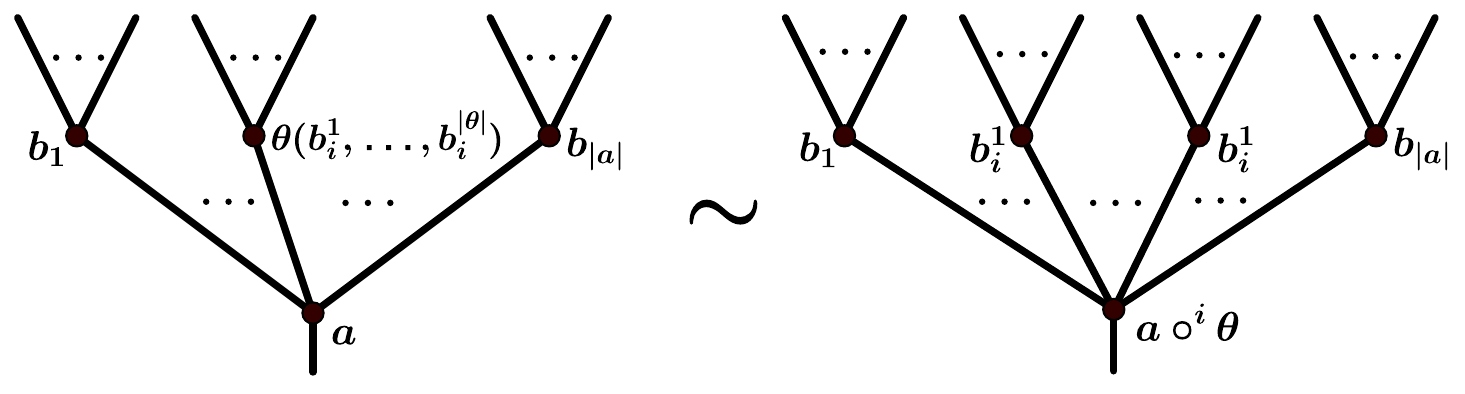}\vspace{-1pt}
$$
In particular, if $\theta\in O(0)$ is an arity zero element and $\gamma_0\colon O(0)\to M(0)$ is the arity zero left action, the above relation looks as follows:\vspace{3pt}
$$
\hspace{-15pt}\includegraphics[scale=0.4]{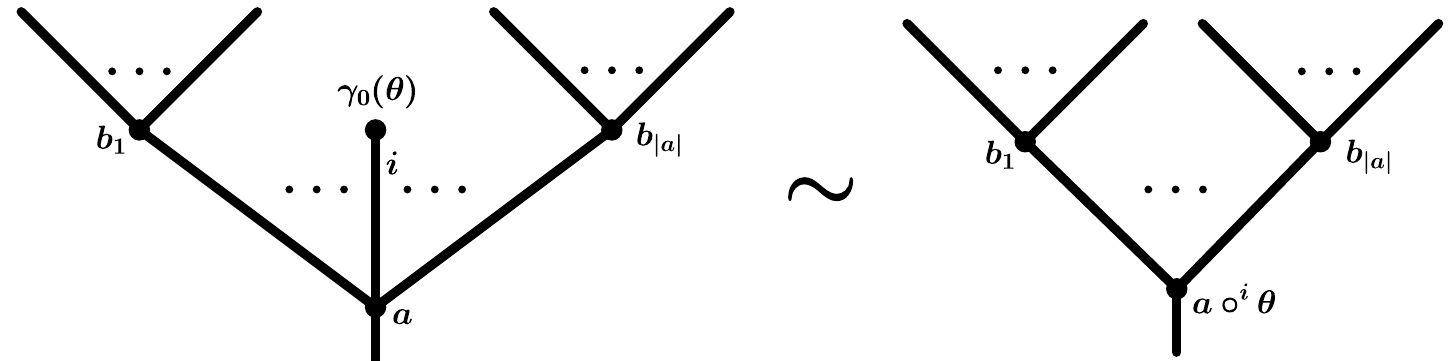}\vspace{-1pt}
$$
\end{itemize}
We impose one additional relation on $N{\circ_O}M$.

  \begin{itemize}
\item[$ii)$] For $\theta\in O$,  $a\in N$,  and $b_1,\ldots,b_{|\theta|+|a|-1}\in M$, 
and setting $i'=|b_{1}|+\cdots + |b_{i-1}|+1$:
$$
\includegraphics[scale=0.4]{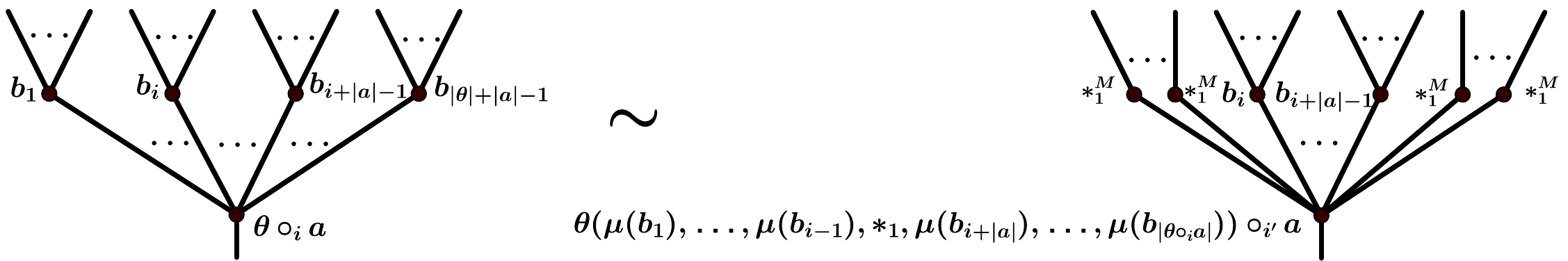}\vspace{3pt}
$$
\end{itemize}
Elements of $N{\ltimes}M$ will be denoted as $[a\{b_1,\ldots,b_{|a|}\}\, ;\, \sigma]$, where $a\in N$,
$b_1,\ldots,b_{|a|}\in M$, $\sigma\in \Sigma_{|b_1|+\ldots +|b_{|a|}|}$, where $\sigma$ is the labelling
of the leaves.
 The infinitesimal left $O$-action on $N{\ltimes}M$ is defined as follows:
$$
\begin{array}{lrll}\vspace{4pt}
\circ_{i}: & O(n)\times N{\ltimes}M(m) & \longrightarrow & N{\ltimes}M(n+m-1);\\
&\theta\,;\,[a(\{b_{1},\ldots,b_{|a|}\}\, ;\, \sigma] &\longmapsto & [(\theta\circ_{i}a)\{\underset{i-1}{\underbrace{\ast_1^{M},\ldots,\ast_1^{M}}},b_{1},\ldots,b_{|a|},\underset{n-i}{\underbrace{\ast_1^{M},\ldots,\ast_1^{M}}}\}\, ;\, id_{\Sigma_{n}}\circ_{i}\sigma].
\end{array} 
$$
 It is easy to see that the formula above enables $N{\circ_O} M$ with a well-defined infinitesimal left $O$-action. 
 By construction, $N{\circ_O} M$ has a right $O$-action.
 It is also straightforward that the  relation $(ii)$ respects both these actions and 
 forces the compatibility between them, thus turning $N{\ltimes}M$ into an $O$-Ibimodule.
\end{const}

Assuming that $O$ is reduced, we will be interested in two particular cases:
\begin{equation}\label{eq:barIb}
\overline{\calI b}{}^\Sigma(O):= \calI b^\Lambda(O){\ltimes}\calB^\Sigma(O)\hspace{15pt} \text{ and } \hspace{15pt}
\overline{\calI b}{}^\Lambda(O):= \calI b^\Lambda(O){\ltimes}\calB^\Lambda(O).
\end{equation}
In both cases we choose  the basepoint $\ast_1^\calB\in \calB^\Sigma(O)(1)$ and $\ast_1^\calB\in \calB^\Lambda(O)(1)$ to be the pearled 1-corolla labelled by $\ast_1\in O(1)$. \vspace{5pt}

Similarly to the previous section, we introduce a filtration in $\overline{\mathcal{I}b}{}^\Sigma(O)$
and $\overline{\mathcal{I}b}{}^\Lambda(O)$ according to the number of geometrical inputs. We start with 
$\overline{\mathcal{I}b}{}^\Sigma(O)$. We say that an element is prime if it is not obtained as a result of the $O$-action (other than by the action
of $\ast_1\in O(1)$).  Let $x=[a\{b_1,\ldots,b_{|a|}\}\, ;\, \sigma]$ be an element in $\overline{\mathcal{I}b}{}^\Sigma(O)$. 
Recall that $a\in\calI b^\Lambda(O)$ can be seen as a pearled tree whose all vertices (except the pearl) have arity $\geq 1$ and are labelled by
elements of $O$ and non-pearl vertices in addition are labelled by numbers in $[0,1]$. Similarly, 
$b_i\in \calB^\Sigma(O)$ 
are trees with a section without restriction on the arity of their vertices (except that vertices below the section must have arity $\geq 1$), and whose all vertices are labelled by elements
in $O$ and in addition non-pearls are labelled by numbers in $[0,1]$. Such an element is
prime if the two conditions hold:
\begin{itemize}
\item  the root of $a$ is a pearl or if it is not a pearl, it is labelled by a number strictly smaller than~1.
\item the vertices of $b_i$, $i=1,\ldots, |a|$, above the section are all labelled by numbers strictly smaller than~1.
\end{itemize}
The first condition means that $x$ is not an image of a non-trivial left $O$-action, and the second one 
means that $x$ is not an image of a non-trivial right action. If one of the conditions above does not hold,
the element is composite. In the latter case, one can assign to $x$ its prime component by removing the root 
 of $a$ (in case it is labelled by~1) together with  all $b_i$, such that the $i$-th leaf of $a$ is directly  connected to the root; and, 
 in addition, by removing all 
 vertices from the (non-removed) elements $b_j$, $j=1,\ldots, |a|$, that are above the section and are labelled by~1. 
  Figure~\ref{D8} gives an example of an element and its prime component.
\begin{figure}[!h]
\hspace{13pt}
\includegraphics[scale=0.12]{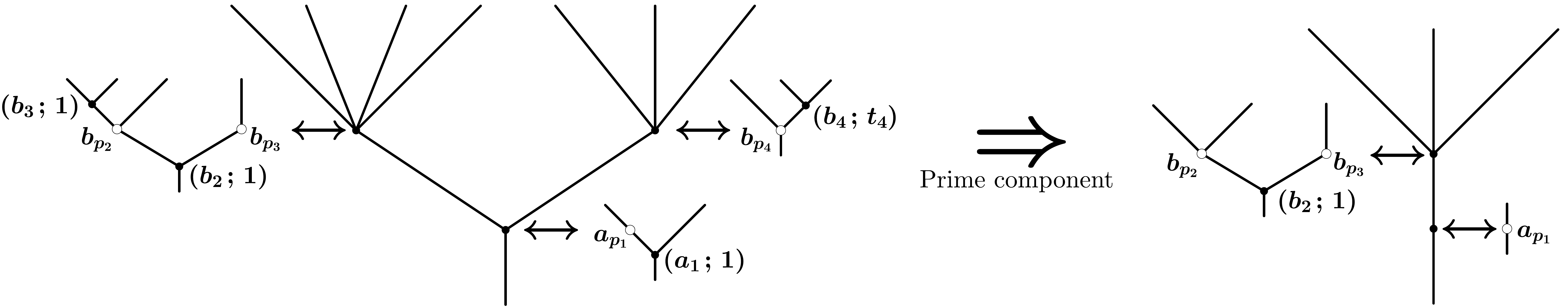}
\caption{Illustration of a composite point in $\overline{\mathcal{I}b}{}^\Sigma(O)$ together with its prime component.}\label{D8}\vspace{-5pt}
\end{figure}

A prime point $x$ is in the $k$-th filtration term  $\overline{\mathcal{I}b}{}_{k}^\Sigma(O)$ if its number of geometrical inputs (which is the arity of $x$ plus the total number of univalent vertices of all $b_i$) is smaller than $k$. Similarly, a composite point is in the $k$-th filtration term if its prime component is in  $\overline{\mathcal{I}b}{}_{k}^\Sigma(O)$. For instance, the composite point  in Figure \ref{D8} is an element in the third filtration term. Finally, $\overline{\mathcal{I}b}{}_{k}^\Sigma(O)$ is an infinitesimal bimodule over $O$ and one has the following filtration in  $\overline{\mathcal{I}b}{}^\Sigma(O)$:
\begin{equation}\label{E6}
\xymatrix{
\overline{\mathcal{I}b}{}_{0}^\Sigma(O) \ar[r] & \overline{\mathcal{I}b}{}_{1}^\Sigma(O)\ar[r] & \cdots \ar[r] & \overline{\mathcal{I}b}{}_{k-1}^\Sigma(O) \ar[r] & \overline{\mathcal{I}b}{}_{k}^\Sigma(O) \ar[r] & \cdots \ar[r] & \overline{\mathcal{I}b}{}^\Sigma(O)
}
\end{equation}

\begin{thm}\label{th:proj_bar_cof}
Assume that $O$ is reduced well-pointed and $\Sigma$-cofibrant. The objects $\overline{\mathcal{I}b}{}^\Sigma(O)$ and $\TT_{k}\overline{\mathcal{I}b}{}_{k}^\Sigma(O)$ are cofibrant replacements of $O$ and $\TT_{k}O$ in the categories $\Sigma\Ibimod_{O}$ and $\TT_{k}\Sigma\Ibimod_{O}$,
 respectively. 
\end{thm}

\begin{proof}
%
%
 We start by checking that the two objects are weakly equivalent to $O$ and $\TT_{k}O$ respectively. More precisely, we prove that the map $O(j)\to\overline{\mathcal{I}b}{}_k^\Sigma(O)(j)$  sending a point $\theta\in O$ to the point $[a\{ b_{1},\ldots,b_{|a|}\}\,;\,id]$, where $a$ is a pearled corolla labelled by $\theta$ while all $b_{i}=\ast_{1}^\calB$, is a deformation retract
 for any $k\geq j$. (Note that the image of this map lies in the $j$-th filtration term $O(j)\to\overline{\mathcal{I}b}{}_j^\Sigma(O)(j)$.)

For this purpose, we introduce a map of sequences $c\colon O\rightarrow \mathcal{B}^\Sigma(O)$ sending a point $\theta\in O$ to the point $[T\,;\,\{t_{v}\}\,;\,\{x_{v}\}]$ in which $T$ is a planar tree with section such that the root is indexed by the pair $(\theta\,;\,1)$ while the other vertices are bivalent pearls labelled by the unit $\ast_{1}\in O(1)$. We check that the map $c$ is a homotopy equivalence whose homotopy inverse is the bimodule map $\mu$, see~\eqref{G8}. We start by contracting the output edges of univalent vertices 
above the section using  the homotopy below. A vertex $v$ in a tree $T$ is said {\it vertically 
connected to a leaf} if there is a path from $v$ to a leaf that always goes upward.
$$
\begin{array}{cccl}\vspace{4pt}
h\colon&\mathcal{B}^\Sigma(O)\times [0\,;\,1] & \longrightarrow & \mathcal{B}^\Sigma(O); \\ 
&\left[ T\,;\,\{t_{v}\}\,;\,\{x_{v}\}\right]\,;\,t & \longmapsto & \left[ T\,;\,\{t\cdot t_{D(v)}+(1-t)\cdot t_{v}\}\,;\,\{x_{v}\}\right],
\end{array} 
$$
where $D(v)=v$ for a vertex below the section or a pearl. Otherwise, $D(v)$ is the first vertex 
(possibly itself)  vertically connected to a leaf in the path joining $v$ to its closest pearl. By convention, if such vertex does not exist, then $D(v)$ is  the closest pearl and $t_{D(v)}$ is fixed to be $0$. This homotopy pulls down and thus removes all the vertices above the section that are not vertically connected to a leaf. Note that this homotopy 
preserves filtration.


Applying relation $(v)$ of Construction~\ref{B0}, we obtain that the sequence $\mathcal{B}^\Sigma(O)$ is homotopy equivalent to a sub-sequence (sequence of subobjects)  $\mathcal{B}^\Lambda(O)$ formed by points without univalent vertices. Note that any such element of arity $i$ lies in the $i$-th filtration term, thus 
for the further retraction we do not need it to preserve filtration.\vspace{5pt}

Then, we contract all the vertices above the section using the homotopy which brings the parameters to $0$. Using the axiom $(iv)$ of Construction \ref{B0}, we get that the sequence $\mathcal{B}^\Sigma(O)$ is homotopy equivalent to the sub-sequence formed by points without vertices above the section and such that the pearls are indexed by the unit $*_1$. Finally, we bring the parameters below the section to $1$. This proves that $c$ is a homotopy equivalence. Furthermore, the reader can check that 
$\TT_{k}c\colon\TT_kO\to\TT_k\calB^\Sigma_k(O)$ is a well defined homotopy equivalence since we contract first the univalent vertices. We denote by $H$ the homotopy retraction so obtained.
\begin{figure}[!h]
\begin{center}
\includegraphics[scale=0.5]{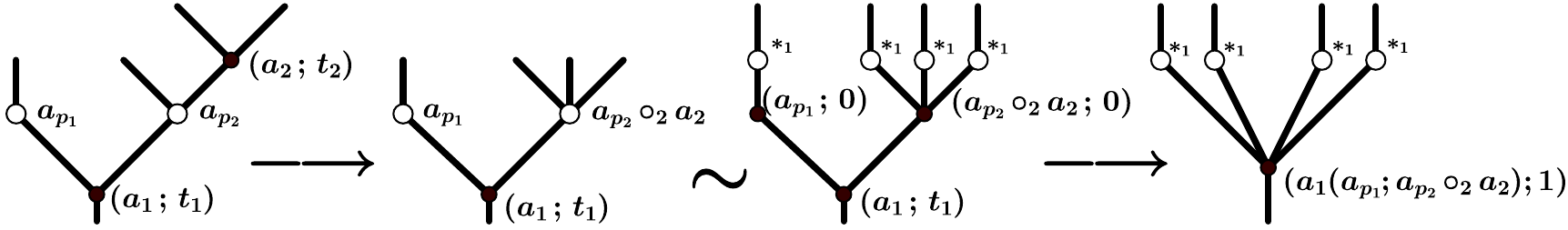}\vspace{-5pt}
\caption{Illustration of the homotopy retraction $H$.}\label{Fig:H}\vspace{-5pt}
\end{center}
\end{figure}

We extend the homotopy $H$ to the sequence $\overline{\mathcal{I}b}{}^\Sigma(O)$ by sending a pair 
$([a\{b_{1},\ldots,b_{|a|}\}\,;\,\sigma]\,;\,t)$ to the element $[ a\{H(b_{1}\,;\,t),\ldots,H(b_{|a|}\,;\,t)\,\}\,;\,\sigma]$. This proves that this sequence is weakly equivalence to the sub-sequence formed by points in which the elements $b_{i}$ are in the image of $c$. In other words, each $b_i$ is a  result of a left $O_{>0}$ action
on the elements $\ast_1^\calB$. Applying relation $(i)$ of Construction~\ref{D7}, this left action can be replaced to the right action on $a$.
Therefore,  the sequence $\overline{\mathcal{I}b}{}^\Sigma(O)$ is homotopy equivalent to the sub-sequence formed by points $[a\{b_{1},\ldots,b_{|a|}\}\,;\,id]$ in which the elements $b_{i}=\ast_1^\calB$ are corollas indexed by the unit $*_1$. Finally, we bring the parameters indexing the vertices of the pearled tree 
$a$ to $0$. We get  that $\overline{\mathcal{I}b}{}^\Sigma(O)$ is homotopy equivalent to $O$. The reader can check that the same argument works for the truncated case since the homotopy $H$ for each arity $j$ preserves the
$k$-th  filtration term, $k\geq j$.

The proof of the cofibrancy is obtained similarly by showing that each inclusion in the filtration~\eqref{E6} is a sequence of cell attachments, thus a 
cofibration, see Appendix~\ref{s:A1} for a proof of a similar statement.
\end{proof}

\subsubsection{Homeomorphism between $\mathcal{I}b^\Lambda(O)$ and $\overline{\mathcal{I}b}{}^\Lambda(O)$.}\label{ss:alt1}
Now we will look at the semi-direct product $\overline{\calI b}{}^\Lambda(O)=\calI b^\Lambda(O){\ltimes}\calB^\Lambda(O)$.
First thing to notice is that if in the general construction of $N{\ltimes}M$ one has $M(0)=*$, then $N{\circ_O} M$ as a symmetric sequence is the same as $N{\circ_{O_{>0}}} M_{>0}$. As a consequence, for the elements
$[a\{b_1,\ldots,b_{|a|}\}\, ;\, \sigma]\in  \overline{\calI b}{}^\Lambda(O)$, one can assume that all $b_i$
are of  positive arity. Note also that $\calB^\Lambda(O)$ can be viewed as an ($O{-}O_{>0}$)-subbimodule
of $\calB^\Sigma(O)$ (as being spanned by trees with section without univalent vertices). Thus,
$\overline{\calI b}{}^\Lambda(O)$ is also an infinitesimal $O_{>0}$-subbimodule of $\overline{\calI b}{}^\Sigma(O)$.
We consider the filtration in $\overline{\calI b}{}^\Lambda(O)$ induced by this inclusion and filtration~\eqref{E6}: 
 \begin{equation}\label{eq:filtr_bar_reedy}
\xymatrix{
\overline{\mathcal{I}b}{}_{0}^\Lambda(O) \ar[r] & \overline{\mathcal{I}b}{}_{1}^\Lambda(O)\ar[r] & \cdots \ar[r] & \overline{\mathcal{I}b}{}_{k-1}^\Lambda(O) \ar[r] & \overline{\mathcal{I}b}{}_{k}^\Lambda(O) \ar[r] & \cdots \ar[r] & \overline{\mathcal{I}b}{}^\Lambda(O).
}
\end{equation}
Theorem~\ref{th:proj_bar_cof} tells us that $\overline{\mathcal{I}b}{}^\Sigma(O)$ is
a cofibrant replacement of $O$ as $O$-Ibimodule. In the Reedy setting the same statement holds, 
moreover we actually have a homeomorphism.

\begin{thm}\label{H5}
For any reduced operad $O$, there  is a filtration preserving homeomorphism of $\calO$-Ibimodules $\gamma
\colon \mathcal{I}b^\Lambda(O)\rightarrow\overline{\mathcal{I}b}{}^\Lambda(O)$.  
\end{thm}

\begin{proof}
We will build the homeomorphism $\gamma$ explicitly. Let $[T\,;\, \{a_{v}\}\,;\, \{t_{v}\}]$ be a point in $\mathcal{I}b^\Lambda(O)$ and let 
$[\, a\{b_{1},\ldots,b_{|a|}\}\,;\,\sigma\,]$,  where $\sigma$ is the permutation labelling the leaves of $T$,  be its image under $\gamma$ that we are going to construct.  The elements $a$, $b_{1},\ldots, b_{|a|}$ are obtained by taking two cuts in $[T\,;\, \{a_{v}\}\,;\, \{t_{v}\}]$. 
 The part below the lower cut produces $a\in \mathcal{I}b^\Lambda(O)$, and the connected pieces above the lower cut produce $b_i$s.  The second cut determines the position of the sections in 
 $b_{i}\in \mathcal{B}^\Lambda(O)$.  Both cuts are horizontal in the sense that the  path from any leaf to the root must meet each cut exactly once.  The lower cut should go above the pearl and thus above all the vertices on the trunk (the path between 
 the pearl and the root). In general position   such cut  passes only through edges and does not pass through vertices. But if it does, then 
 the vertices on the lower cut can be both considered as ones below the cut (and thus belonging to $a$), or as ones above  it (and thus belonging
 to $b_i$). If we decide to put such vertex below the cut, it would mean that $a$ is an image of a non-trivial right $O$-action; if we put such vertex above,
 then the corresponding $b_i$ is an image of a non-trivial left $O$-action. The equalizing condition $(i)$ of Construction~\ref{D7} between the left and right actions, would
 imply  that either way produces the same element in $\overline{\mathcal{I}b}{}^\Lambda(O)$. For the upper cut, if it  goes through a vertex, then
 the vertex becomes a pearl of the corresponding $b_i$, if it goes through an edge, then in $b_i$  one replaces this edge by a pearled 1-corolla labelled by
 $\ast_1\in O(1)$. \vspace{-10pt}
 
 \begin{figure}[!h]
\begin{center}
\includegraphics[scale=0.5]{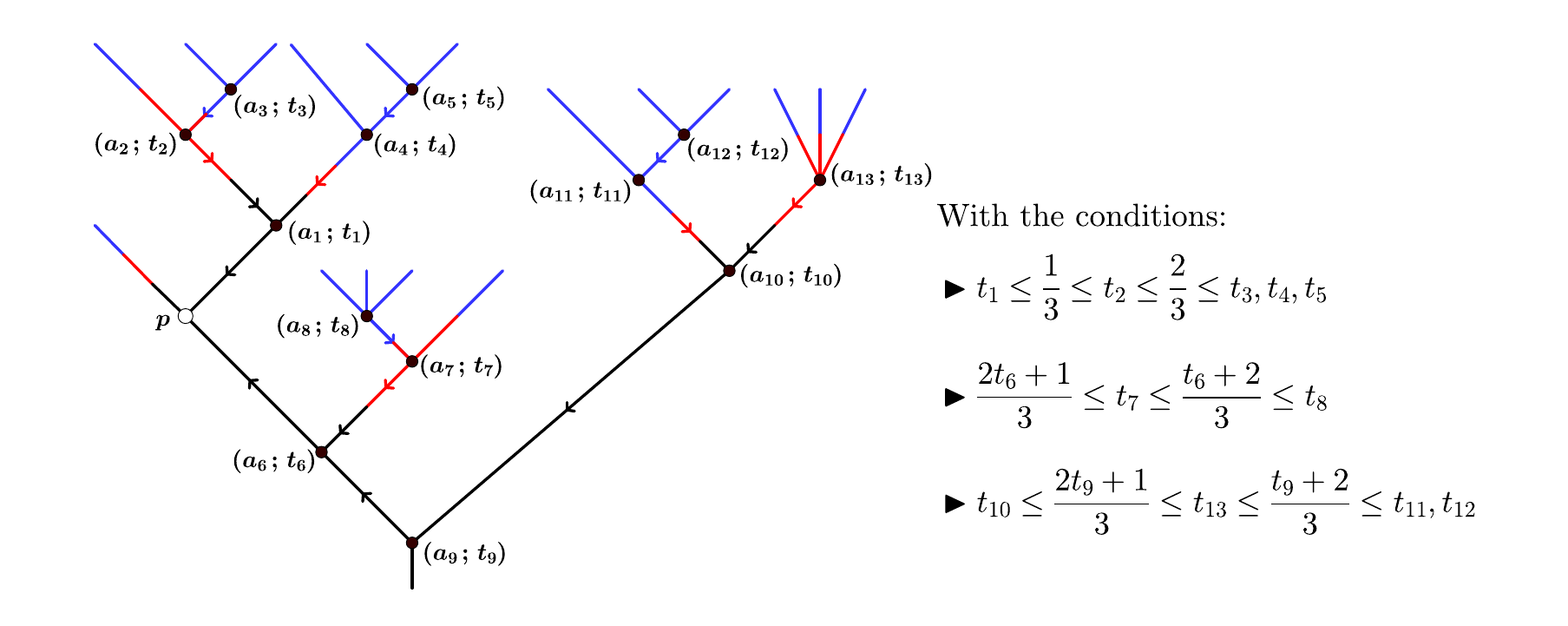}
\caption{Example of a point in $\mathcal{I}b^\Lambda(O)$.}\label{E0}\vspace{-30pt}
\end{center}
\end{figure} 

\newpage

  Let $v$ be a vertex of $T$ labelled by $(a_v,t_v)$. For the pearl $p$ in the formulas below, we  set~$t_p=0$. Let $j(v)$ be the
  closest to $v$  vertex on the trunk. We say that $v$ 
  \begin{itemize}
  \item is below the first cut if $t_v<\frac{2t_{j(v)}+1}{3}$,
  \item is on the first cut if $t_v=\frac{2t_{j(v)}+1}{3}$,
  \item is between the two cuts if $\frac{2t_{j(v)}+1}{3}<t_v<\frac{t_{j(v)}+2}{3}$,
  \item is on the second cut if $t_v=\frac{t_{j(v)}+2}{3}$,
  \item is above the second cut if $t_v>\frac{t_{j(v)}+2}{3}$.
  \end{itemize}
 For a vertex $j(v)$ on the trunk, we simply devide the segment
$[t_{j(v)},1]$ into 3 equal intervals to determine where the 2 cuts pass. Note in particular that these formulas ensure that  all the 
vertices of the trunk are below the first cut.  As illustration of this construction, we consider an example in Figure~\ref{E0}.\vspace{5pt}

  For every vertex on the first cut we decide whether it is above or below the cut, 
  and then
  move slightly the cut accordingly. As explained above, the resulting element is always the same regardless of our choice.  The point $a=[T_{a}\,;\,\{a_{v}^{0}\}\,;\, \{t_{v}^{0}\}]$ is given by the subtree $T_{a}$ of $T$ having the vertices $v$ below the first cut. 
The points $\{a_{v}^{0}\}$ labelling the vertices of $T_{a}$ are  the same and the set of real numbers $\{t_{v}^{0}\}$ indexing its  vertices  is defined by an increasing linear bijection $\left[t_{j(v)},\frac{2t_{j(v)}+1}{3}\right]\to \left[t_{j(v)},1\right]$. Explicitly,
   $
   t_v^0:=3t_v-2t_{j(v)}.
   $
   
 \begin{figure}[!h]
\begin{center}
\includegraphics[scale=0.4]{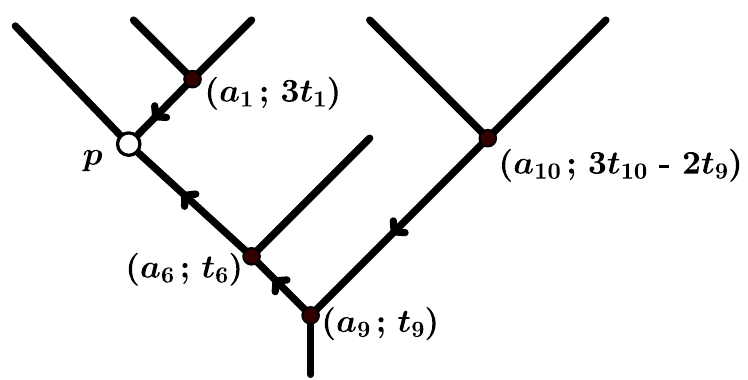}\vspace{-10pt}
\caption{The point $a$ associated to Figure \ref{E0}.}\vspace{-10pt}
\end{center}
\end{figure} 
   
   
   For the points $b_{j}=[T_{j}\,;\,\{b_{v}^{j}\}\,;\, \{t_{v}^{j}\}]\in \mathcal{B}^\Lambda(O)$, their corresponding  trees are the pieces above the first cut. The pearls of $T_j$ are the vertices on the upper cut. If the upper cut crosses an edge, as explained above, we create a pearl labelled by 
   $\ast_1\in O(1)$. Besides these new vertices that we label by $\ast_1$, all  the other vertices will be labelled by the same elements $b_v^j=a_v\in O$. 
   The numbers $t_v^j$ are determined as follows: if $v$ is between the two cuts, one uses  the decreasing linear bijection
   $\left[\frac{2t_{j(v)}+1}{3},\frac{t_{j(v)}+2}{3}\right]\to [0,1]$. If $v$ is above the second cut, one uses the increasing linear bijection
   $\left[\frac{t_{j(v)}+2}{3},1\right]\to [0,1]$. Explicitly,
  $$
t_{v}^{j}:=\left\{
\begin{array}{cl}\vspace{5pt}
\dfrac{3t_{v}-t_{j(v)}-2}{1-t_{j(v)}}, & \text{if the vertex }v\text{ is above the section}; \\  
\dfrac{3t_{v}-t_{j(v)}-2}{t_{j(v)}-1},  &  \text{if the vertex }v\text{ is below the section}.
\end{array} 
\right.
$$
  
\begin{figure}[!h]
\begin{center}
\includegraphics[scale=0.4]{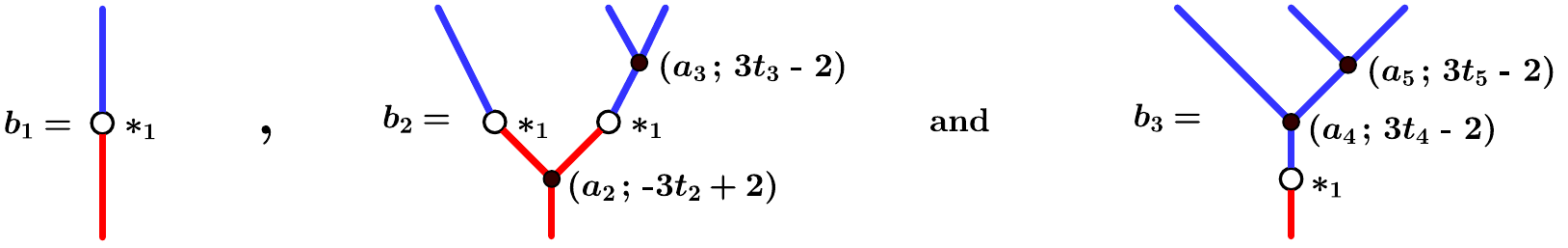}\vspace{-5pt}
\caption{The points $b_{1}$, $b_{2}$ and $b_{3}$ associated to Figure \ref{E0}.}\vspace{-25pt}
\end{center}
\end{figure} 

\newpage

 \begin{figure}[!h]
\begin{center}
\includegraphics[scale=0.4]{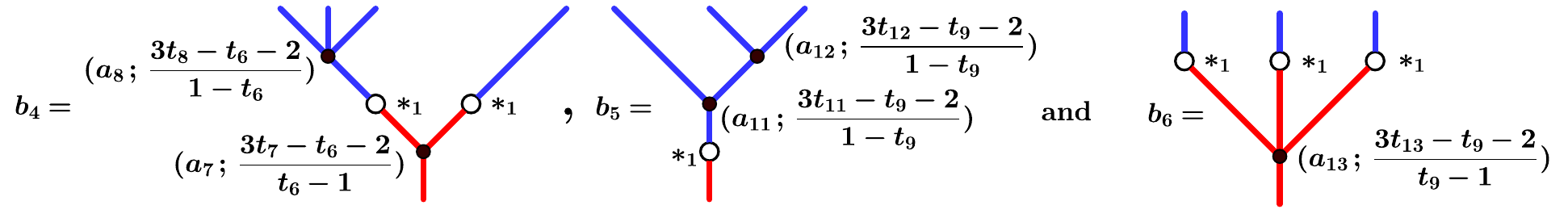}\vspace{-5pt}
\caption{The points $b_{4}$, $b_{5}$ and $b_{6}$ associated to Figure \ref{E0}.}\vspace{-5pt}
\end{center}
\end{figure}

 One can easily see that our map $\gamma$ is well defined and bijective. In particular, one would not need to divide by zero: because of the condition
 $t_{s(e)}\geq t_{t(e)}$ on the numbers labelling the vertices and relation~$(iii)$ of Construction~\ref{C8}, a vertex of the trunk labelled by~1 can
 only be connected to leaves or other vertices on the trunk. To check that this map is continuous, one can do it first for its inverse map $\gamma^{-1}$. 
 Note that its inverse can be similarly defined by \lq\lq{}forgetting\rq\rq{} the two sections -- instead of dividing by numbers that can tend to zero, we will be multiplying by such numbers. Thus the map 
 $\gamma^{-1}$ is continuous and bijective. We check first that it is a homeomorphism for the
 commutative operad $\Comm$. In the latter case $\calI b^\Lambda(\Comm)(k)$ is the space of non-planar pearled trees, whose vertices are labelled
 by numbers and non-pearled vertices have arity $\geq 2$. It is a cell complex, and the map $\gamma_\Comm^{-1}\colon\overline{\calI b}{}^\Lambda(\Comm)(k)\to\calI b^\Lambda(\Comm)(k)$ is a cellular refinement.
 On the other hand, both spaces $\overline{\calI b}{}^\Lambda(O)(k)$ and 
 $\calI b^\Lambda(O)(k)$ can be obtained as realizations of cellular cosheaves on
  $\overline{\calI b}{}^\Lambda(\Comm)(k)$ and $\calI b^\Lambda(\Comm)(k)$, respectively,
  see Appendix~\ref{s:A2}. Moreover, the first cosheaf is  homeomorphic to 
  the pullback of the second one obtained  along the homeomorphism  (refinement) $\gamma_\Comm^{-1}$. We remark, however, that in case $O(1)\neq *$, the value of the cocheaf on a cell of 
  $\calI b^\Lambda(\Comm)(k)$ is not merely given by the product of  $O(|v|)$s, but will also account for forgotten under the map
  $\calI b^\Lambda(O)(k)\to \calI b^\Lambda(\Comm)(k)$ bivalent vertices and labelled by the elements of $O(1)$. However, one can still apply Lemma~\ref{l:cellularsheaf} to conclude that the map $\gamma^{-1}$ 
  is
  a homeomorphism of realizations of cellular cosheaves.

\end{proof}

\begin{expl}\label{ex:2compl2}
Let $O$ be doubly reduced. Then $\overline{\mathcal{I}b}{}^\Lambda(O)(0)$ and $\overline{\mathcal{I}b}{}^\Lambda(O)(1)$
have exactly the same combinatorial decomposition as $\mathcal{I}b^\Lambda(O)(0)$ and $\mathcal{I}b^\Lambda(O)(1)$, respectively,
see Example~\ref{ex:2compl}.
This is so because all the corresponding pearled trees of arity $\leq 1$ have all their vertices on the trunk.  There are exactly two
pearled trees of arity~2 that have a vertex outside the trunk: \includegraphics[scale=0.06]{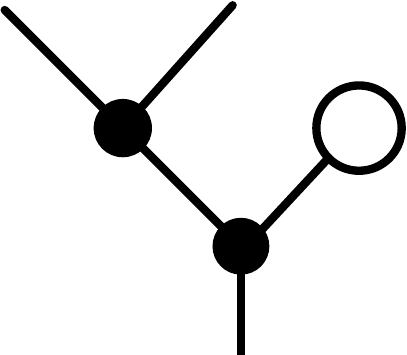}, 
\includegraphics[scale=0.06]{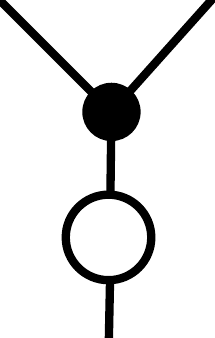}. The corresponding strata of  $\mathcal{I}b^\Lambda(O)(2)$ are subdivided in two
in  $\overline{\mathcal{I}b}{}^\Lambda(O)(2)$; the other strata in  $\overline{\mathcal{I}b}{}^\Lambda(O)(2)$ remain the same, see Figure~\ref{Fig:2compl2}.\vspace{-10pt}
\begin{figure}[!h]
\begin{center}
\includegraphics[scale=0.18]{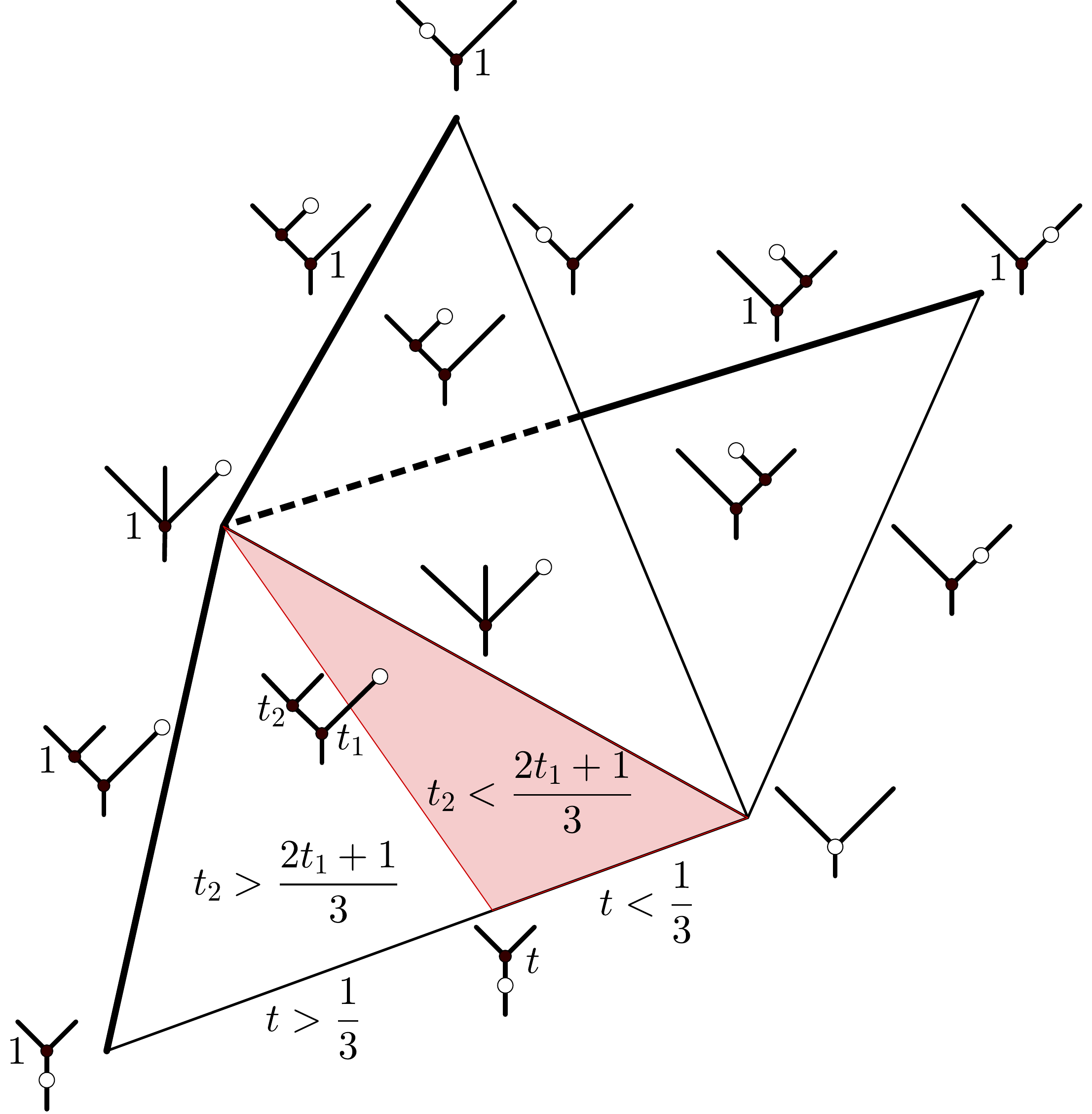}\vspace{-5pt}
\caption{Visualising the space of trees associated to $\overline{\mathcal{I}b}{}^\Lambda(O)(2)$.}\label{Fig:2compl2}\vspace{-25pt}
\end{center}
\end{figure}
\end{expl}

\subsection{Construction of the map}\label{s:map}

\begin{thm}\label{th:main_2reduced}
Given a map $\calO\to M$ of $O$-bimodules, where $O$ is a doubly reduced, $\Sigma$-cofibrant topological operad, one has 
a natural morphism between the towers
\begin{equation}\label{eq:equiv_general1_red}
\xi_k\colon  \Map_*\left(\Sigma O(2), \TT_k\Bimod_O^h\left(\TT_kO,\TT_kM
\right)\right) \to \TT_k\Ibimod_O^h\left(\TT_kO,\TT_kM\right),\, k\geq 0,
\end{equation}
producing  at the limit $k=\infty$:
\begin{equation}\label{eq:equiv_general2_red}
\xi\colon  \Map_*\left(\Sigma O(2), \Bimod_O^h\left(O,M
\right)\right) \to
\Ibimod_O^h\left(O,M\right),
\end{equation}
which is an equivalence of towers (respectively, equivalence up to stage $k$), provided $M(0)=*$ and $O$ is
coherent (respectively, $k$-coherent).
\end{thm}

In this subsection we describe this map $\xi_\bullet$  explicitly.  The theorem is proved in 
Subsection~\ref{ss:proof:2reduced}. 
 Let us start with a general construction.
\begin{const}\label{const:map}
Let $O$, $M$, and $N$ be as in Construction~\ref{D7} of $N{\ltimes}M$. Namely, $N$ is an $O$-Ibimodule, $M$ is an $O$-bimodule
endowed with a map $\mu\colon M\to O$ and a basepoint $\ast_1^M\in M(1)$ such that $\mu(\ast_1^M)=\ast_1\in O(1)$.  Let in addition assume that
$O$ is reduced ($O(0)=\{*_0\}$) and one is given a map $\mu'\colon N\to O$ of $O$-Ibimodules. Assume also that $K$ is an $O$-bimodule endowed with an
$O$-bimodules map $\eta\colon O\to K$. 

\noindent Define $\partial N(1)$ as the image of the infinitesimal left $O$-action $\circ_1\colon O(2)\times N(0)\to N(1)$.  
We will construct a map 
\begin{equation}\label{eq:xi}
\xi\colon \Map_*\left( N(1)/\partial N(1), \Bimod_O(M,K)\right) \to
\Ibimod_O(N{\ltimes}M,K).
\end{equation}
(Note that these are actual mapping spaces, not the derived ones.)  Here the basepoint for the space $\Bimod_O(M,K)$ is $\eta\circ \mu$. 

\noindent Consider the maps 
\[
\tau_i\colon N(n)\xrightarrow{j_i} N(1)\xrightarrow{q} N(1)/\partial N(1),\quad 1\leq i\leq n,
\]
where 
$$
\begin{array}{lccl}
j_{i}: & N(n) & \longrightarrow & N(1); \\ 
& x & \longmapsto & \big(\cdots\big(\,\big(\,\big(\cdots\big( x\circ^{n}\ast_0\big)\cdots\big)\circ^{i+1}\ast_0\big)\circ^{i-1}\ast_0\big)\cdots\big)\circ^{1}\ast_0.
\end{array} 
$$
Equivalently, $j_i$ is the $\Lambda$-map corresponding to the inclusion
\[
\underline{1}\hookrightarrow\underline{n},\quad 1\mapsto i.
\]
Let $g\in \Map_*\left( N(1)/\partial N(1), \Bimod_O(M,K)\right)$, for any element 
$[a\{b_1,\ldots,b_{|a|}\}\, ;\, \sigma]\in N{\ltimes}M$, we set
\begin{equation}\label{eq:xi2}
\xi(g)\left(\, [a\{b_1,\ldots,b_{|a|}\}\, ;\, \sigma]\,\right) =
\mu'(a)\bigl( g(\tau_1(a))(b_1),\ldots,g(\tau_{|a|}(a))(b_{|a|}) \bigr)\cdot \sigma.
\end{equation}

\end{const}

\begin{lmm}\label{l:xi}
The formula~\eqref{eq:xi2} produces a well-defined and continuous map~\eqref{eq:xi}.
\end{lmm}

\begin{proof}
Let us check first that~\eqref{eq:xi2} produces the same result for different representatives of the same element in $N{\ltimes}M$:
\begin{itemize}
\item \eqref{eq:xi2} respects the symmetric group action and the permutation of the $b_i$s as inputs in $a$ by construction;
\item \eqref{eq:xi2} respects the relation $(i)$ of Construction~\ref{D7} because $\mu'\colon N\to O$ respects the right $O$-action and $g(t)
\colon M\to K$ respects the left $O$-action for any $t\in N(1)/\partial N(1)$;
\item \eqref{eq:xi2} respects relation $(ii)$ of Construction~\ref{D7} because $j_i(a)\in\partial N(1)$ if $a$ is a result of an
infinitesimal left action by some $x\in O(\ell)$, $\ell\geq 2$, and $i$ labels a leaf connected directly to $x$. 
\end{itemize}

As a result $\xi(g)\colon N{\ltimes}M\to K$ is a well-defined map of $\Sigma$ sequences.  To see that it is an $O$-Ibimodules map,
we need to check that it respects the $O$-action:
\begin{itemize}
\item \eqref{eq:xi2} respects the right $O$-action, because  each $g(t)\colon M\to K$ respects the right action.
\item \eqref{eq:xi2} respects the infinitesimal left action by the construction of this action on $N{\ltimes}M$ and $K$ and also 
because we  have again $\tau_i(a)\in \partial N(1)$ for the corresponding inputs of $a$.
\end{itemize}

Now, let us check that $\xi$ is continuous. The space $\Ibimod_O(N{\ltimes}M,K)$ is topologized as a subspace in the product
\[
\prod_{k\geq 0} \Map(N{\ltimes}M(k),K(k)).
\]
Thus, it is enough to check that the projection to each factor is continuous. By Lemma~\ref{l:q_map},  for a quotient space $X/{\sim}$, the mapping space
$\Map(X/{\sim},Y)$ is a subspace of $\Map(X,Y)$. The space $N{\ltimes}M(k)$ is a quotient of $N{\circ} M(k)$. Thus we should check that
the corresponding map to $\Map(N{\circ}M(k),K(k))$ is continuous. On the other hand
\[
\Map\left(\coprod_\alpha X_\alpha,Y\right)=\prod_\alpha\Map(X_\alpha,Y),
\]
and $N{\circ}M(k)$ is a disjoint union of the components whose elements have the form $[a\{b_1,\ldots,b_{|a|}\}\, ;\, \sigma]$,
where the arities of $a$, $b_i$s, and the shuffle permutation $\sigma$ are fixed. Thus we are again reduced to show that the projection
to each corresponding factor is continuous. But for any such component $X_\alpha$ of $N{\circ}M(k)$, it is an easy exercise that the explicit
formula~\eqref{eq:xi2} defines a continuous map in the corresponding mapping space $\Map(X_\alpha,K(k))$. 

\end{proof}

%
%
%
%

Now, let $O$ and $M$ be as in Theorem~\ref{th:main_2reduced}. We apply Construction~\ref{const:map} to $\overline{\calI b}{}^\Sigma(O)= \calI b^\Lambda(O){\ltimes}\calB^\Sigma(O)$ and 
$\overline{\calI b}{}^\Lambda(O)= \calI b^\Lambda(O){\ltimes}\calB^\Lambda(O)$. Let also $M\rightarrow M^f$ be a Reedy fibrant 
replacement of $M$ as $O$-bimodule, which is automatically a Reedy fibrant replacement as $O$-Ibimodule. Recall from Example~\ref{D2}
that $\calI b^\Lambda(O)(1)=C(O(2))$. It is also easy to see that $\partial \calI b^\Lambda(O)(1)$ is the base
 of this cone. 

\begin{cordef}\label{cordef}
One has well-defined and continuous maps
\begin{equation}\label{eq:equiv_red_pr1}
\xi_k^\Sigma\colon  \Map_*\left(\Sigma O(2), \TT_k\Bimod_O\left(\TT_k \calB_k^\Sigma(O),\TT_kM
\right)\right) \to \TT_k\Ibimod_O\left(\TT_k \overline{\calI b}{}_k^\Sigma(O),\TT_kM\right),\, k\geq 0,
\end{equation}
producing  at the limit $k=\infty$:
\begin{equation}\label{eq:equiv_red_pr2}
\xi^\Sigma\colon  \Map_*\left(\Sigma O(2), \Bimod_O\left(\calB^\Sigma(O),M
\right)\right) \to
\Ibimod_O\left(\overline{\calI b}{}^\Sigma(O),M\right);
\end{equation}
and maps
\begin{equation}\label{eq:equiv_red_pr3}
\xi_k^\Lambda\colon  \Map_*\left(\Sigma O(2), \TT_k\Bimod_O\left(\TT_k \calB^\Lambda(O),\TT_kM^f
\right)\right) \to \TT_k\Ibimod_O\left(\TT_k \overline{\calI b}{}^\Lambda(O),\TT_kM^f\right),\, k\geq 0,
\end{equation}
producing  at the limit $k=\infty$:
\begin{equation}\label{eq:equiv_red_pr4}
\xi^\Lambda\colon  \Map_*\left(\Sigma O(2), \Bimod_O\left(\calB^\Lambda(O),M^f
\right)\right) \to
\Ibimod_O\left(\overline{\calI b}{}^\Lambda(O),M^f\right),
\end{equation}
which realize the maps~\eqref{eq:equiv_general1_red} and~\eqref{eq:equiv_general2_red}.
%
\end{cordef}

The maps $\xi_k^\Sigma$ and $\xi_k^\Lambda$ are defined as truncations of $\xi^\Sigma$ and $\xi^\Lambda$, the latter being produced by 
Construction~\ref{const:map}.

\subsection{Alternative proof of Main Theorem~\ref{th:delooping1}}\label{ss:FM5}
In this subsection we explain our initial geometric approach for the delooping. The main reason we 
present it
here, it gives some insight on  Constructions~\ref{D7} and~\ref{const:map} as well as on the proof
of Main Theorem~\ref{th:main} given in the next section.\vspace{5pt}

Recall  Subsections~\ref{ss:IF} and~\ref{ss:BF}, where  the infinitesimal bimodule $\calIF_m$ and the bimodule $\calBF_m$ are defined.
Thanks to the maps $\mu_I\colon\calIF_m\to\calF_m$ and $\mu_B\colon\calBF_m\to\calF_m$ of (infinitesimal) bimodules,
Constructions~\ref{D7} and~\ref{const:map} can be applied. Elements of the $\calF_m$-Ibimodule $\calIF_m{\ltimes}\calBF_m$ are
configurations of points in $\R^m$ labelled by elements from $\calBF_m$. When points collide, the labels get multiplied by means of the left 
$\calF_m$-action on $\calBF_m$. Points can escape to infinity, the labels in the latter case 
(by the last relation $(ii)$ of Construction~\ref{D7}) get shrunken in size 
by means of the projection $\mu_B\colon\calBF_m\to\calF_m$, which is a quotient by rescalings. The space $\calIF_m(1)\cong D^m$ 
is obtained as compactification of $C(1,\R^m)=\R^m$. Thus, $\calIF_m(1)/\partial\calIF_m(1)\cong\R^m\cup\{\infty\}\cong S^m$ 
is the one-point compactification of $\R^m$. The map $\xi$ has also a geometrical description
\begin{equation}\label{eq:xi_altern}
\hspace{-30pt}\begin{array}{l}\vspace{5pt}
\xi\colon\Omega^m\Bimod_{\calF_m}(\calBF_m,M^f)\cong
\Map_*\left(\R^m\cup\{\infty\}, \Bimod_{\calF_m}(\calBF_m,M^f)\right) \\ 
\hspace{240pt}\longrightarrow \Ibimod_{\calF_m}(\calIF_m{\ltimes}\calBF_m,M^f).
\end{array} 
\end{equation}
Let $g\in \Map_*\left(\R^m\cup\{\infty\}, \Bimod_{\calF_m}(\calBF_m,M^f)\right)$ and let $a\in\calIF_m(k)$ and $b_1,\ldots,b_k\in\calBF_m$.
For simplicity, let us assume that $a=(a_1,\ldots,a_k)\in C(k,\R^m)\subset\calIF_m(k)$ is in the interior. Then the map $\xi$ applies $g(a_i)$
to each $b_i$, and then get them multiplied by means of the left action by $\mu_I(a)\in\calF_m(k)$:
\[
\xi(g)\left(\, [a\{b_1,\ldots,b_{k}\}\, ;\, \sigma]\,\right) =
\mu_I(a)\bigl( g(a_1)(b_1),\ldots,g(a_k)(b_{k}) \bigr)\cdot \sigma.
\]
Our initial idea of the proof was to use the map~\eqref{eq:xi_altern} and  to show that it is an equivalene. 
One would need to show  that $\calIF_m{\ltimes}\calBF_m$ is homeomorphic to $\calIF_m$ as an $(\calF_m)_{>0}$-Ibimodule. One has a map $\zeta\colon\calIF_m(1)\times\calBF_m(k)
\to\calIF_m{\ltimes}\calBF_m(k)$, $(a,b)\mapsto a\{b\}$. Note that both the source and the target are manifolds 
of the same dimension $mk$. An important step in the proof is to show that $\calIF_m{\ltimes}\calBF_m(k)$
private the interior of the image of~$\zeta$ deformation retracts on the boundary $\partial\bigl(
\calIF_m{\ltimes}\calBF_m(k)\bigr)$.\vspace{5pt}

  Eventually we realized that this
geometrical construction of $\ltimes$ works in a more general context and in particular can be nicely applied to the Boardman-Vogt type
replacements. 
  We reiterate that we conjecture that
$\calIF_m$ and $\calBF_m$ are homeomorphic to $\calI b^\Lambda(\calF_m)$ and $\calB^\Lambda(\calF_m)$, respectively, as (infinitesimal)
$(\calF_m)_{>0}$-bimodules, and thus the 
construction that we finally use should not be much different from the initial geometrical one.

\section{Proof of Main Theorems~\ref{th:main} and~\ref{th:main3}}\label{s:proof}
In the first subsection we prove Main Theorem~\ref{th:main} in the special case of a strongly doubly reduced operad $O$, 
which is Theorem~\ref{th:main_2reduced}. In the second subsection, using homotopy theory methods, we complete the proof of our theorem
by generalizing it to the case of any weakly doubly reduced operad. In the third subsection we describe different explicit maps realizing the equivalence
of the theorem. In the last subsection we formulate and sketch a proof of a more general statement --
Main Theorem~\ref{th:main3}.

\subsection{Main Theorem~\ref{th:main}. Special case: $O$ is doubly reduced}\label{ss:proof:2reduced}
To prove Theorem~\ref{th:main_2reduced}, we constructed in Subsection~\ref{s:map} an explicit map of towers:
\begin{equation}\label{eq:equiv_red_pr1x}
\xi_k^\Lambda\colon  \Map_*\left(\Sigma O(2), \TT_k\Bimod_O\left(\TT_k \calB^\Lambda(O),\TT_kM^f
\right)\right) \to \TT_k\Ibimod_O\left(\TT_k \overline{\calI b}{}^\Lambda(O),\TT_kM^f\right),
\end{equation}
where $M\xrightarrow{\simeq}M^f$ is a Reedy fibrant replacement of $M$ as $O$-bimodule. In this subsection we show that if $O$ is $k$-coherent,
\eqref{eq:equiv_red_pr1x} is a weak equivalence. The proof is by induction over $k$. To simplify notation, the source of $\xi_k^\Lambda$
will be denoted by
\[
\TT_k\calB:= \Map_*\left(\Sigma O(2), \TT_k\Bimod_O\left(\TT_k \calB^\Lambda(O),\TT_kM^f
\right)\right),
\]
and the target by
\[
\TT_k\calI:= \TT_k\Ibimod_O\left(\TT_k \overline{\calI b}{}^\Lambda(O),\TT_kM^f\right).
\]
Without loss of generality and also to simplify notation, we will be assuming that $M$ is already Reedy fibrant $M=M^f$.

\vspace{3pt}

{\bf Initialization: $k=0$ and $1$.} For $k=0$, it is easy to see that $\TT_0\calB=\TT_0\calI=*$, because 
$\calB^\Lambda(O)(0)=\overline{\calI b}{}^\Lambda(O)(0)=M(0)=*$. For $k=1$, the map $\xi_1^\Lambda$ is always a homeomorphism,
that\rq{}s why the condition of 1-coherence is vacuous. Indeed, $\calB^\Lambda(O)(1)=*$, see Example~\ref{ex:1compl}, and the
space
\[
\TT_1\Bimod_O(T_1\calB^\Lambda(O),T_1M)\cong\Map(*,M(1))=M(1).
\]
Thus, $\TT_1\calB\cong \Map_*(\Sigma O(2),M(1))$, where $M(1)$ is pointed at $\ast_1^M:=\eta(\ast_1)$ (recall that one has
a map $\eta\colon O\to M$ of $O$-bimodules). One also has $\overline{\calI b}{}^\Lambda(O)(1)\cong C(O(2))$ is the cone with the base 
$\partial\overline{\calI b}{}^\Lambda(O)(1)\cong O(2)$, see Examples~\ref{D2} and~\ref{ex:2compl2}. Therefore, $\TT_1\calI$
is a space of maps
\[
\overline{\calI b}{}^\Lambda(O)(1)\to M(1)
\]
with prescribed behavior on $\partial \overline{\calI b}{}^\Lambda(O)(1)$. Let $\theta{\circ_1}\ast_0^\calI\in\partial \overline{\calI b}{}^\Lambda(O)(1)$, where $\theta\in O(2)$,
and $\ast_0^\calI$ is the only point of $\overline{\calI b}{}^\Lambda(O)(0)$. For any map of 1-truncated $O$-Ibinodules
$f\colon\TT_1 \overline{\calI b}{}^\Lambda(O)\to \TT_1M$, one has
\[
f(\theta{\circ_1}\ast_0^\calI)=\theta{\circ_1}f(\ast_0^\calI)=\theta{\circ_1}\ast_0^M=\theta{\circ_1}\eta(\ast_0)=
\theta(\eta(\ast_0),\eta(\ast_1))=
\eta(\theta(\ast_0,\ast_1))=\eta(\ast_1)=\ast_1^M.
\]
The first equality is because $f$ respects the infinitesimal left $O$ action by $\theta$; the second and third ones are because $M(0)$ has only one point
$\ast_0^M$; the fourth one is the definition of the infinitesimal left $O$ action on $M$, see Example~\ref{ex:ibimod2}; the fifth one is because $\eta$ respects the left $O$ action by $\theta$; the sixth  one is because $O(1)=*$.\vspace{5pt}

Thus, $\TT_1\calI$ is also homeomorphic to $\Map_*(\Sigma O(2),M(1))$.  Note that all elements of $\overline{\calI b}{}^\Lambda(O)(1)$ have 
the form $[\, a\{\ast_1^\calB\}\, , \, id\, ]= a\{\ast_1^\calB\}$, where $a\in \calI b^\Lambda(O)(1)=C(O(2))$ and $\ast_1^\calB$ is the only point
of $\calB^\Lambda(O)(1)$. Given $g\in\TT_1\calB\cong \Map_*(\Sigma O(2),M(1))$, one has that $\xi_1^\Lambda(g)\in\TT_1\calI$ sends
\[
a\{\ast_1^\calB\}\mapsto\mu'(a)\left(\, g(\tau_1(a))(\ast_1^\calB)\,\right)=g(a)(\ast_1^\calB).
\]
One has $\mu'(a)=\ast_1$ is the only element in $O(1)$, so we can ignore its action; and $\tau_1(a)=a$. Finally, $g(a)(\ast_1^\calB)$ is
exactly the image of the element $a\in \calI b^\Lambda(O)(1)=C(O(2))$ under the map
\[
C(O(2))\to\Sigma O(2)\xrightarrow{g}M(1).
\]
This completes the proof that $\xi_1^\Lambda$ is a homeomorphism.

\vspace{3pt}

{\bf Induction step.} Assuming that the map $\xi_{k-1}^\Lambda$ is an equivalence and that $O$ is $k$-coherent, we show that $\xi_k^\Lambda$
is a weak equivalence.

\begin{pro}\label{p:fibration}
The maps between the stages of the towers  $\pi_k^\calI\colon\TT_k\calI\to\TT_{k-1}\calI$, $\pi_k^\calB\colon \TT_k\calB\to\TT_{k-1}\calB$
are Serre fibrations.
\end{pro}

We prove this proposition at the end of this subsection.\vspace{5pt}

One has a morphism of fiber sequences
\begin{equation}
\xymatrix{
F_1\ar[d]\ar[r]&\TT_k\calB\ar[r]^{\pi_k^\calB}\ar[d]^{\xi_k^\Lambda}&\TT_{k-1}\calB\ar[d]^{\xi_{k-1}^\Lambda}_{\simeq}\\
F\ar[r]&\TT_k\calI\ar[r]^{\pi_k^\calI}&\TT_{k-1}\calI
},
\label{eq:morph_fibr}
\end{equation}
where $F_1$ is the fiber over some $g\in\TT_{k-1}\calB$, and $F$ is the  fiber over $\xi_{k-1}^\Lambda(g)$. Since $\xi_{k-1}^\Lambda$ is an equivalence by induction hypothesis,
to prove that $\xi_k^\Lambda$ is so, it is enough to show that the map between the fibers $F_1\to F$ is an equivalence for any choice
of~$g$. We will describe the fibers $F$ and $F_1$ as certain spaces of section extensions.\vspace{5pt}

For a Serre fibration $\pi\colon E\to B$, denote by $\Gamma(\pi;B)$ the space of global sections of $\pi$. For $A\subset B$, and a partial section
$s\colon A\to E$, denote by $\Gamma_s(\pi;B,A)$ the space of sections $f\colon B\to E$, which extend $s$, i.e.
$f|_A=s$. \vspace{5pt}

In order to avoid heavy notation, below, every map induced (or determined) by $g$ is denoted by~$g_*$. It should not be understood
as a name of a map, but rather as a label \lq\lq{}made by~$g$\rq\rq{}. We will be always explicit a map between which spaces is discussed
to avoid any confusion. \vspace{5pt}

The fiber $F$ is a space of $\Sigma_k$-equivariant maps
\begin{equation}\label{eq:f}
f\colon \overline{\calI b}{}^\Lambda(O)(k)\to M(k)
\end{equation}
satisfying two conditions. First, it is determined by $g$ on all non-prime elements, i.e. on $\partial \overline{\calI b}{}^\Lambda(O)(k)$ ---
one applies $\xi_{k-1}^\Lambda(g)$ on their prime component and then the $O$-action to see what their image in $M(k)$ should be. 
Secondly, it should respect the $\Lambda$-structure, i.e. the right action by $O(0)$. The first condition means that the upper triangle
in the square below must commute:
\begin{equation}\label{eq:sq_2tr_I}
\xymatrix{
\partial \overline{\calI b}{}^\Lambda(O)(k)\ar[r]^{g_*}\ar@{^{(}->}[d]& M(k)\ar@{->>}[d]\\
 \overline{\calI b}{}^\Lambda(O)(k)\ar[r]^{g_*}\ar@{.>}[ru]^{f}&\MM(M)(k)
 }.
 \end{equation}
 And the second condition translates that the lower triangle in~\eqref{eq:sq_2tr_I} must commute. Here, the right arrow is the matching map,
 which is a fibration as we assumed that $M$ is Reedy fibrant. The lower map is a composition
 \[
  \overline{\calI b}{}^\Lambda(O)(k)\to \MM(\overline{\calI b}{}^\Lambda(O))(k)\xrightarrow{g_*}\MM(M)(k),
  \]
  where the first arrow is the matching map, and the second one is $\MM(g)(k)$. \vspace{5pt}
  
  Consider the pullback of the right map of~\eqref{eq:sq_2tr_I} along its lower one. We get a fibration of $\Sigma_k$-spaces
  \[
  \pi\colon E\to \overline{\calI b}{}^\Lambda(O)(k).
  \]
  Since $\overline{\calI b}{}^\Lambda(O)(k)$ is $\Sigma_k$-cofibrant, and the property of being Serre fibration is local,
  \[
  \pi/\Sigma_k\colon E/\Sigma_k\to \overline{\calI b}{}^\Lambda(O)(k)/\Sigma_k
  \]
  is still a Serre fibration. It is easy to see that $F$ is the space of section extensions:
  \begin{equation}\label{eq:F}
  F=\Gamma_{g_*}\left(\pi/\Sigma_k; \overline{\calI b}{}^\Lambda(O)(k)/\Sigma_k, \partial\overline{\calI b}{}^\Lambda(O)(k)/\Sigma_k\right),
  \end{equation}
  where by $g_*$ we denote the section $\partial \overline{\calI b}{}^\Lambda(O)(k)/\Sigma_k\to E/\Sigma_k$ induced by the upper
   map of~\eqref{eq:sq_2tr_I}.\vspace{5pt}
   
   The points of $F_1$ can similarly be described as the space of $\Sigma_k$-equivariant maps 
   \begin{equation}\label{eq:f1}
   f_1\colon \calI b^\Lambda(O)(1){\times}\calB^\Lambda(O)(k)\to M(k)
   \end{equation}
   making the two triangles in the diagram 
   \begin{equation}\label{eq:sq_2tr_B}
\xymatrix{
\partial \bigl( \calI b^\Lambda(O)(1){\times}\calB^\Lambda(O)(k)\bigr)\ar[r]^(0.66){g_*}\ar@{^{(}->}[d]& M(k)\ar@{->>}[d]\\
 \calI b^\Lambda(O)(1){\times}\calB^\Lambda(O)(k)\ar[r]^(0.66){g_*}\ar@{.>}[ru]^{f_1}&\MM(M)(k)
 }
 \end{equation}
commute, where
\[
\partial \bigl( \calI b^\Lambda(O)(1){\times}\calB^\Lambda(O)(k)\bigr) =
\partial\calI b^\Lambda(O)(1){\times}\calB^\Lambda(O)(k)\bigcup
\calI b^\Lambda(O)(1){\times}\partial\calB^\Lambda(O)(k).
\]
We use this notation later in this subsection. Here, $\partial\calB^\Lambda(O)(k)=\calB^\Lambda_{k-1}(O)(k)
\subset\calB^\Lambda(O)(k)$ is the subspace of non-prime elements.\vspace{5pt}

On the first part of the boundary the upper map of~\eqref{eq:sq_2tr_B} is the composition
\[
\partial \calI b^\Lambda(O)(1){\times}\calB^\Lambda(O)(k)\xrightarrow{pr_2} \calB^\Lambda(O)(k)\xrightarrow{\mu}
O(k)\xrightarrow{\eta} M(k).
\]
On the second part, this map
\[
\calI b^\Lambda(O)(1){\times}\partial\calB^\Lambda(O)(k)\xrightarrow{g_*}M(k)
\]
is determined by the fact that for any $x\in \calI b^\Lambda(O)(1)$, the map $g(x)$ is a map of $(k-1)$-truncated $O$-bimodules and thus
can be uniquely extended to a map of $O$-bimodules $\calB_{k-1}^\Lambda(O)\to M$, see~\eqref{eq:der_bimod_vs_tr2}. \vspace{5pt}

Now let $\widetilde{g}\in F_1$. Its arity $\leq (k-1)$ part is $g$, and the arity $k$ part is $f_1$ which makes commute~\eqref{eq:sq_2tr_B}.
Define the subspace $\partial_1 \overline{\calI b}{}^\Lambda(O)(k)\subset \overline{\calI b}{}^\Lambda(O)(k)$ to contain 
$\partial \overline{\calI b}{}^\Lambda(O)(k)$ together with all the elements of the form $[\, a\{b_{1},\ldots,b_{|a|}\}\,;\,\sigma\,]$, $|a|\geq 2$.
We claim that the arity~$k$ part of $\xi_k^\Lambda(\widetilde{g})\in F$ is uniquely determined by~$g$ on 
$\partial_1 \overline{\calI b}{}^\Lambda(O)(k)$.  Indeed, $\xi_k^\Lambda(\widetilde{g})$ sends $[\, a\{b_{1},\ldots,b_{|a|}\}\,;\,\sigma\,]$ to 
\begin{equation}\label{eq:gtilde}
\mu'(a)\bigl( \widetilde{g}(\tau_1(a))(b_1),\ldots,\widetilde{g}(\tau_{|a|}(a))(b_{|a|}) \bigr)\cdot \sigma
=
\begin{cases}
\mu'(a)\bigl( g(\tau_1(a))(b_1),\ldots,g(\tau_{|a|}(a))(b_{|a|}) \bigr)\cdot \sigma& |a|\geq 2;\\
f_1(a,b_1)& |a|=1.
\end{cases}
\end{equation}
 by the fact that $|b_i|\leq k-1$ for all $i=1\ldots |a|$ 
if $|a|\geq 2$, and in case $|a|=1$ one has $\mu'(a)=\ast_1$ and thus its action can be ignored. \vspace{5pt}

Thus $\xi_k^\Lambda$ sends $F_1$ onto the subspace of $F$ described as the space of $\Sigma_k$-equivariant maps~\eqref{eq:f}
making the diagram
\begin{equation}\label{eq:sq_2tr_II}
\xymatrix{
\partial_1 \overline{\calI b}{}^\Lambda(O)(k)\ar[r]^{g_*}\ar@{^{(}->}[d]& M(k)\ar@{->>}[d]\\
 \overline{\calI b}{}^\Lambda(O)(k)\ar[r]^{g_*}\ar@{.>}[ru]^{f}&\MM(M)(k)
 }
 \end{equation}
 commute. We claim that this subspace of $F$ is exactly the image of $F_1$ and moreover $\xi_k^\Lambda|_{F_1}$ is a homeomorphism on its image.
 To prove it we construct the inverse.  Let $\zeta$ be the map
 \[
 \begin{array}{rccc}
 \zeta\colon&  \calI b^\Lambda(O)(1){\times}\calB^\Lambda(O)(k)&\longrightarrow & \overline{\calI b}{}^\Lambda(O)(k);\\
 & (a,b)&\mapsto& a\{b\}.
 \end{array}
 \]
 Let us check that 
 \begin{equation}\label{eq:subset}
 \zeta\bigl( \,\partial\bigl( \calI b^\Lambda(O)(1){\times}\calB^\Lambda(O)(k)\bigr)\, \bigr)\subset
 \partial_1 \overline{\calI b}{}^\Lambda(O)(k).
 \end{equation}
 Recall that $\partial\overline{\calI b}{}^\Lambda(O)(k)$ consists of elements $[\, a\{b_{1},\ldots,b_{|a|}\}\,;\,\sigma\,]$, for which either $a$
 is the image of the arity $\geq 2$ infinitesimal  left $O$-action, or one of $b_i$ is the image of the arity $\geq 2$ right $O$-action. As a 
 consequence
 \[
  \zeta\bigl(\, \partial \calI b^\Lambda(O)(1){\times}\calB^\Lambda(O)(k)\, \bigr)\subset  \partial \overline{\calI b}{}^\Lambda(O)(k)
 \subset\partial_1 \overline{\calI b}{}^\Lambda(O)(k).
 \]
 One has 
 \[
 \partial\calB^\Lambda(O)(k) = \partial_l\calB^\Lambda(O)(k)\cup \partial_r\calB^\Lambda(O)(k),
 \]
 where $\partial_l$, respectively $\partial_r$, is the image  of the arity $\geq 2$ left, respectively right, $O$-action. One similarly has
 \[
  \zeta\bigl(\,  \calI b^\Lambda(O)(1){\times}\partial_r\calB^\Lambda(O)(k)\, \bigr)\subset  \partial \overline{\calI b}{}^\Lambda(O)(k).
  \]
  Using relation~$(i)$ of Construction~\ref{D7} (right-left action equalizing), we get
  \[
   \zeta\bigl(\,  \calI b^\Lambda(O)(1){\times}\partial_l\calB^\Lambda(O)(k)\, \bigr)\subset  \partial_1 \overline{\calI b}{}^\Lambda(O)(k).
  \]
Given $f$ that makes commute~\eqref{eq:sq_2tr_II}, we define $f_1:=f\circ\zeta$. Using~\eqref{eq:gtilde} for the case $|a|=1$ 
and~\eqref{eq:subset}, one easily checks that this assignment produces the inverse of $\xi_k^\Lambda|_{F_1}$. \vspace{5pt}

We get that $F_1\cong\xi_k^\Lambda(F_1)\subset F$ can also be described as a space of section extensions similarly to~\eqref{eq:F},
and the inclusion $F_1\to F$ in these terms is described as
\begin{equation}\label{eq:F-F1}
\Gamma_{g_*}\left(\pi/\Sigma_k; \overline{\calI b}{}^\Lambda(O)(k)/\Sigma_k, \partial_1\overline{\calI b}{}^\Lambda(O)(k)/\Sigma_k\right)
\hookrightarrow
\Gamma_{g_*}\left(\pi/\Sigma_k; \overline{\calI b}{}^\Lambda(O)(k)/\Sigma_k, \partial\overline{\calI b}{}^\Lambda(O)(k)/\Sigma_k\right).
  \end{equation}
The theorem now follows from a lemma and a proposition.

\begin{lmm}\label{l:sections}
Given a Serre fibration $\pi\colon E\to B$, cofibrant inclusions $X\subset Y\subset Z$ with $X$ cofibrant, and with $X\subset Y$ a weak equivalence,
for any partial section $s\colon Y\to E$, the natural inclusion
\[
\Gamma_s(\pi;Z,Y)\hookrightarrow \Gamma_s(\pi;Z,X)
\]
is a weak equivalence.
\end{lmm}

\begin{pro}\label{p:d1}
If $O$ is a doubly reduced and $k$-coherent operad, the inclusion $\partial\overline{\calI b}{}^\Lambda(O)(k)
\hookrightarrow \partial_1\overline{\calI b}{}^\Lambda(O)(k)$
is a weak equivalence.
\end{pro}

To finish the proof of the theorem we notice that both spaces $\partial\overline{\calI b}{}^\Lambda(O)(k)$ and $\partial_1\overline{\calI b}{}^\Lambda(O)(k)$
are $\Sigma_k$-cofibrant. Therefore, their quotients by $\Sigma_k$ are equivalent to their homotopy quotients and by Proposition~\ref{p:d1}
these quotients are equivalent to each other. We then apply Lemma~\ref{l:sections} to conclude that the map $F_1\to F$ is an equivalence.
\hfill $\square$.

\begin{proof}[Proof of Lemma~\ref{l:sections}]
One has a morphism of fiber sequences:
\begin{equation}
\xymatrix{
\Gamma_s(\pi;Z,Y)\ar[d]\ar[r]&\Gamma(\pi;Z)\ar[r]\ar@{=}[d]&\Gamma(\pi;Y)\ar[d]^{\simeq}\\
\Gamma_s(\pi;Z,X)\ar[r]&\Gamma(\pi;Z)\ar[r]&\Gamma(\pi;X)
}.
\end{equation}
Since $X\hookrightarrow Y$ is a trivial cofibration, the right arrow is a trivial fibration. As a consequence, the left arrow, i.e. map between the fibers,
is also an equivalence.
\end{proof}

\begin{proof}[Proof of Proposition~\ref{p:d1}]
Instead of the initial Definition~\ref{d:coherent} of $k$-coherence, we will be using equivalent to it condition that the maps
\[
\partial\calI b^\Lambda(O)(i)\hookrightarrow\partial' \calI b^\Lambda(O)(i),\quad 2\leq i\leq k,
\]
are equivalences, see Remark~\ref{r:hocolim}. Recall also  the homeomorphism
\begin{equation}\label{eq:homeo}
\gamma_k\colon \calI b^\Lambda(O)(k)\xrightarrow{\cong} \overline{\calI b}{}^\Lambda(O)(k)
\end{equation}
constructed in Subsubsection~\ref{ss:alt1}. Since $\gamma_k$ is filtration preserving, it sends 
$\partial \calI b^\Lambda(O)(k)$ homeomorphically onto $\partial \overline{\calI b}{}^\Lambda(O)(k)$.
Denote by $\partial_1 \calI b^\Lambda(O)(k):=\gamma_k^{-1}\left( \partial_1 \overline{\calI b}{}^\Lambda(O)(k)\right)$. 
We claim that $\partial' \calI b^\Lambda(O)(k)$ is a deformation retract of $\partial_1 \calI b^\Lambda(O)(k)$, 
which implies the statement of our proposition. To recall, $\partial' \calI b^\Lambda(O)(k)$ consists of two
(non-disjoint) types of elements:
\begin{itemize}
\item elements in $\partial \calI b^\Lambda(O)(k)$;
\item elements labelled by pearled trees from $\Psi_k^U$
\end{itemize}

Recall Definitions~\ref{d:trunk_arity}, \ref{d:ud}. The trees from $\Psi_k^U$ are exactly those whose trunk has arity $\geq 2$. The trees from $\Psi_k\setminus\Psi_k^U$, i.e. trees
whose trunk has arity one, are of two types: either its root is pearled and has arity one, or its root has arity two and it is connected to
a univalent pearl:

\[
\includegraphics[scale=0.25]{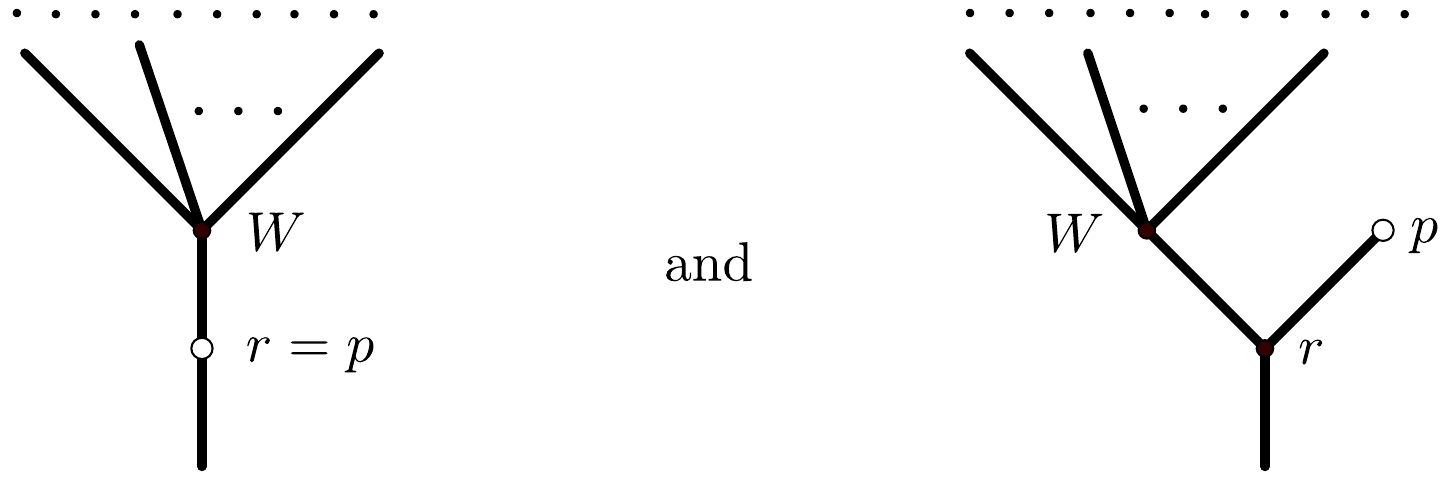}
\]

In both cases there is only one non-pearl  vertex, that we denote by $W$, connected to the root. 

The elements of $\partial_1 \calI b^\Lambda(O)(k)$ are of three (non-disjoint) types:
\begin{itemize}
\item elements in $\partial \calI b^\Lambda(O)(k)$;
\item elements labelled by pearled trees from $\Psi_k^U$;
\item elements labelled by pearled trees from $\Psi_k\setminus\Psi_k^U$, and whose only non-pearl vertex $W$ connected to the root $r$
is labelled by $t_W\leq \frac{2t_r+1}{3}$ (as before, if $r=p$, we set $t_r=0$).
\end{itemize}
Indeed, elements of $\partial_1 \calI b^\Lambda(O)(k)$ are either from $\partial \calI b^\Lambda(O)(k)$ (first type) or are
preimages under $\gamma_k$ of $[\, a\{b_{1},\ldots,b_{|a|}\}\,;\,\sigma\,]$, with $|a|\geq 2$. The latter type of elements 
are those in $\calI b^\Lambda(O)(k)$ for which the lower cut of the homeomorphism $\gamma_k$ either crosses
at least two edges or, as a limit case, at least one vertex. If an element in $\calI b^\Lambda(O)(k)$  is labelled by a tree
$T\in \Psi_k^U$ whose trunk has arity $\geq 2$, this condition is always satisfied. If it is labelled by a tree $T\in\Psi_k\setminus\Psi_k^U$
whose trunk has arity one, this condition is satisfied only if the lower cut goes through or above $W$, which is exactly given by the inequality
of the type~3 points. 

\begin{figure}[!h]
\begin{center}
\includegraphics[scale=0.09]{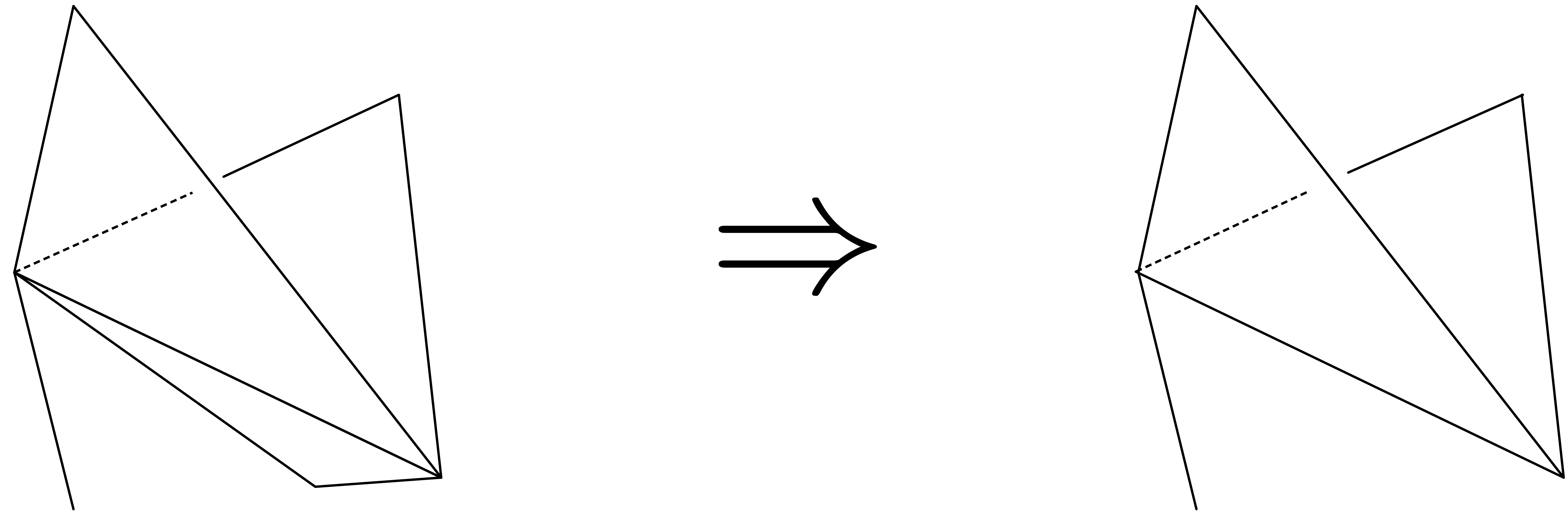}\vspace{-5pt}
\caption{Deformation retraction of $\partial_1 \calI b^\Lambda(O)(2)$ onto $\partial' \calI b^\Lambda(O)(2)$.}\label{Fig:Homot}\vspace{-15pt}
\end{center}
\end{figure}

We define a deformation retraction $H$ of $\partial_1 \calI b^\Lambda(O)(k)$ onto $\partial' \calI b^\Lambda(O)(k)$ as follows. 
$H$ is set to be identity on the elements of $\partial_1 \calI b^\Lambda(O)(k)$ of the first and second type, and it is defined as
\[
H[T\,;\,\{a_{v}\}\,;\,\{t_{v}\}]= [T\,;\,\{a_{v}\}\,;\,\{h(t,t_{v})\}]
\]
on the elements of the third type. 
   Let
\[
t_{max}:=\underset{v\in V(T)}{Max}(t_v).
\]


\noindent one has $t_r\leq t_W\leq t_{max}\leq 1$. In particular if $t_{max}=1$
or $t_W=t_r$ the homotopy $h(t,-)$ is the constant identy since such element
lies already in $\partial' \calI b^\Lambda(O)(k)$. In general, for any $t\in[0,1]$, $h(t,-)$  is linear on $[t_W,t_{max}]$ and $h(t,t_r)=t_r$.
It starts by pulling the left end $t_W$ toward the root:
\[
h(t,t_W)=t_W+t(t_r-t_W),
\]
while expanding the length of the image of $[t_W,t_{max}]$: 
\[
h(t,t_{max})-h(t,t_W)= (1+t)(t_{max}-t_W);
\]
in case $h(t,t_{max})$ reaches 1, the homotopy stops. It is easy to compute that the latter case happens when $t_{max}-t_W\geq \frac{1-t_r}{2}$, and
this happens at 
\[
t=\frac{1-t_{max}}{t_{max}-2t_W+t_r}.
\]
Otherwise the homotopy continues until $t=1$, in which case the edge $(r,W)$ gets contracted. 
The condition $t_W\leq \frac{2t_r+1}{3}$ implies $1-t_W\geq \frac 23(1-t_r)>\frac{1-t_r}2$. Thus when $t_{max}$
approaches~1, the homotopy gets closer and closer to the constant identity. We conclude that the homotopy $H$ is continuously defined.
%
%
  Figure~\ref{Fig:Homot} shows how the retraction works for $k=2$, compare with Figure~\ref{Fig:2compl2}.

\end{proof}

\begin{proof}[Proof of Proposition~\ref{p:fibration}]
We need a lemma which is an easy exercise. (Hint: use  
Lemma~\ref{l:pushprod}.)

\begin{lmm}\label{l:greg}
Let $G$ be a finite group. Consider the projective model structure in $\Topo^G$. 
If $X\hookrightarrow Y$ is a cofibration in $\Topo^G$ with $X$ cofibrant, and $E\to B$ is a fibration in $\Topo^G$, then 
the following map is a Serre fibration
\[
\Map_G(Y,E)\to \Map_G(Y,B)\bigtimes_{\Map_G(X,B)}\Map_G(X,E).
\]
\end{lmm}

One has two pullback squares:

\[
\xymatrix{
\TT_k\calI\ar[r]\ar[d] & \Map_{\Sigma_k} \left( \overline{\calI b}{}^\Lambda, M\right)\ar[d]\\
\TT_{k-1}\calI\ar[r] &  \Map_{\Sigma_k} \left( \overline{\calI b}{}^\Lambda, \MM(M)\right)
\underset{\Map_{\Sigma_k} \left( \partial\overline{\calI b}{}^\Lambda, \MM(M)\right)}{\scalebox{2.0}{$\bigtimes$}}
\Map_{\Sigma_k} \left( \partial\overline{\calI b}{}^\Lambda, M\right)
}
\]

$$
\xymatrix{
\TT_k\calB\ar[rr]\ar[d] & & \Map_{\Sigma_k}\left(   \calI b^\Lambda(1){\times}\calB^\Lambda, M \right)\ar[d] & \\
\TT_{k-1}\calB\ar[r] & \Map_{\Sigma_k}\left(   \calI b^\Lambda(1){\times}\calB^\Lambda, \MM(M) \right)\hspace{-30pt} & \underset{\Map_{\Sigma_k}\left(  \partial\left( \calI b^\Lambda(1){\times}\calB^\Lambda\right), \MM(M) \right)}{\scalebox{2.0}{$\bigtimes$}}  & \hspace{-30pt}\Map_{\Sigma_k}\left(  \partial\left( \calI b^\Lambda(1){\times}\calB^\Lambda\right), M\right)
}
$$

%
%
Here for shortness, $\overline{\calI b}{}^\Lambda$, $\calB^\Lambda$, $\calI b^\Lambda(1)$, $M$, $\MM(M)$ stay for
$\overline{\calI b}{}^\Lambda(O)(k)$, $\calB^\Lambda(O)(k)$, $\calI b^\Lambda(O)(1)$, $M(k)$, $\MM(M)(k)$, respectively.
By Lemma~\ref{l:greg}, the right arrows in the squares are Serre fibrations. Since fibrations are preserved by pullbacks, the left arrows
are also fibrations.
\end{proof}

\subsection{Main Theorem~\ref{th:main}. General case: $O$ is weakly doubly reduced}\label{ss:proof:gen}
We need to recall some general facts  on the homotopy theory of undercategories. We thank B.~Fresse for prividing us  the three
lemmas below. The last Lemma~\ref{l:fresse3}, which follows from Lemmas~\ref{l:fresse1} and~\ref{l:fresse2},
is crucial for the proof of Main Theorem~\ref{th:main}.\vspace{5pt} 

    Given a model category $C$, and $c\in C$, the undercategory $(c\downarrow C)$ is endowed with a model structure, in which
a morphism is a weak equivalence, fibration, or cofibration if it is one in~$C$, see~\cite[Section~7.6.4]{Hirschhorn03}. Given a morphism 
$f\colon c_1\to c_2$, one has an adjunction
\begin{equation}\label{eq:f_adj}
f^!\colon(c_1\downarrow C)\leftrightarrows (c_2\downarrow C)\colon f^*,
\end{equation}
where $f^*$ is the precomposition with $f$, and $f^!$ sends $c_1\to x$ to $c_2\to c_2\sqcup_{c_1}x$. The following lemma is well-known, see for example~\cite[Proposition 2.5]{Rezk02} or~\cite[Proposition~16.2.4]{MayPonto}. 

\begin{lmm}\label{l:fresse1}
In case $C$ is a left proper model category, and $f\colon c_1\to c_2$ is a weak equivalence, the adjunction~\eqref{eq:f_adj}
is a Quillen equivalence. 
\end{lmm}

Now, given a Quillen adjunction
\begin{equation}\label{eq:adj_LR}
L\colon C\leftrightarrows D\colon R
\end{equation}
between two model categories, $c\in C$, $d\in D$, and a map
\begin{equation}\label{eq:phi} 
\phi\colon c\to Rd,
\end{equation}
one has an adjunction
\begin{equation}\label{eq:adj_LRhat}
\hat L\colon (c\downarrow C)\leftrightarrows(d\downarrow D)\colon\hat R,
\end{equation}
where $\hat R$ sends $d\to d_1$ to the composition $c\xrightarrow{\phi} Rd\to Rd_1$, and $\hat L$ sends $c\to c_1$ to the
pushout map $d\to d\sqcup_{Lc}Lc_1,$ where $Lc\to d$ is the adjunct of $\phi$. 

\begin{lmm}\label{l:fresse2}
If~\eqref{eq:adj_LR} is a Quillen equivalence, $c$ is cofibrant, $d$ is fibrant,~\eqref{eq:phi} is a weak equivalence,   and $D$ is left proper, then
the adjunction~\eqref{eq:adj_LRhat} is a Quillen equivalence.
\end{lmm}

\begin{proof}
First, note that since $c$ is cofibrant, $d$ is fibrant, and~\eqref{eq:adj_LR} is a Quillen equivalence, the adjunct of~\eqref{eq:phi},
i.e. the map $Lc\to d$ is a weak equivalence. Let  $S$ be the fibrant replacement functor in $D$. We need to show two conditions. The first one is that for any cofibration $c\to c_1$,
the map
$
c_1\to RS\left(d\sqcup_{Lc}Lc_1\right)
$
is a weak equivalence. (The second condition is formulated and checked at the end.)
This map factors as
\[
c_1\to RSLc_1\to RS\left(d\sqcup_{Lc}Lc_1\right),
\]
where the first map is an equivalence because $c_1$ is cofibrant and~\eqref{eq:adj_LR} is a Quillen equivalence; the second map is 
a weak equivalence, because $R$ preserves equivalences between fibrant objects (being right adjoint
in a Quillen adjunction) and the map
\[
Lc_1\to d\sqcup_{Lc} Lc_1
\]
is an equivalence being a pushout of  a weak equivalence $Lc\to d$ along a cofibration $Lc\to Lc_1$
(using left properness of $D$).\vspace{5pt}

For any map $c\to c_1$, let $c\to\hat{Q}c_1\xrightarrow{\sim}c_1$ denote its canonical factorization
into a cofibration and a trivial fibration. In other words, $\hat Q$ is the cofibrant replacement functor in
$c\downarrow C$. 
The second condition that we need to check is that for any morphism $d\to d_1$ with fibrant target, the map
$d\sqcup_{Lc}L\hat{Q}Rd_1\to d_1$ is an equivalence. Consider the composition
\[
L\hat{Q}Rd_1\to d\sqcup_{Lc}L\hat{Q}Rd_1\to d_1.
\]
The first map here is an equivalence, because it is a pushout of a weak equivalence $Lc\to d$ along a
 cofibration $Lc\to L\hat{Q}Rd_1$ (and $D$ is left proper). The composition of the two $L\hat{Q}Rd_1\to d_1$ is an equivalence,
 because $\hat{Q}Rd_1$ is cofibrant, $d_1$ is fibrant, $\hat{Q}Rd_1\to Rd_1$ is an equivalence
  (and~\eqref{eq:adj_LR} is a Quillen equivalence).   By two out of three axiom, the map 
  $d\sqcup_{Lc}L\hat{Q}Rd_1\to d_1$ is also an equivalence.

\end{proof}

\begin{lmm}\label{l:fresse3}
In case~\eqref{eq:adj_LR} is a Quillen equivalence between left proper model categories, $R$ preserves weak equivalences,
and~\eqref{eq:phi} is a weak equivalence, then for any $c\in C$, $d\in D$,  \eqref{eq:adj_LRhat} is a Quillen equivalence.
\end{lmm}

\begin{proof}
Let $Q$ denote the cofibrant replacement functor in $C$ and $S$ denote the fibrant replacement functor in $D$. 
One has a composition of adjunctions:
\[
(Qc\downarrow C)\leftrightarrows (c\downarrow C)\leftrightarrows (d\downarrow D)\leftrightarrows (Sd\downarrow D).
\]
The left and right ones are Quillen equivalences by Lemma~\ref{l:fresse1}. The composition of the three is a Quillen equivalence by Lemma~\ref{l:fresse2}
(here we use that $R$ preserves weak equivalences as one needs $Rd\to RSd$ to be a weak equivalence). 
Applying twice the two out of three property of Quillen equivalences, we get that the middle adjunction is also a Quillen equivalence. 
\end{proof}



\begin{pro}\label{p:M1}
 Let $\phi\colon W\xrightarrow{\simeq}W_1$ be a weak equivalence of Reedy cofibrant operads.
 
(i) The adjunction 
 \begin{equation}\label{eq:W_W1_adj_under}
\hat\phi_B^!\colon(W\downarrow\Lambda\Bimod_W)\leftrightarrows(W_1\downarrow\Lambda\Bimod_{W_1})\colon\hat\phi_B^*
\end{equation}
is a Quillen equivalence.

(ii)  Let
$\eta\colon W\to M$ be a cofibration in $\Lambda\Bimod_W$, then there exists a $W_1$-bimodule $M_1$
endowed with a map $\eta_1\colon W_1\to M_1$ of $W_1$-bimodules, and an equivalence $M\xrightarrow{\simeq} M_1$
of $W$-bimodules, such that the square
\begin{equation}\label{eq:M1}
\xymatrix{
W\ar[r]^\eta\ar[d]^\phi_\simeq&M\ar[d]^\simeq\\
W_1\ar[r]^{\eta_1}&M_1
}
\end{equation}
commutes.
\end{pro}

\begin{proof}
(i) The adjunction~\eqref{eq:W_W1_adj_under} is induced by the map $\phi$ and the restriction-induction adjunction 
\[
\phi_B^!\colon\Lambda\Bimod_W\leftrightarrows\Lambda\Bimod_{W_1}\colon\phi_B^*,
\]
which is a Quillen equivalence by Theorem~\ref{th:bimod_ind_restr}.
By Theorem~\ref{th:bim_reedy_model}~(iii), the categories $\Lambda\Bimod_W$ and $\Lambda\Bimod_{W_1}$ are left proper relative to $\Sigma$-cofibrant objects. For simplicity
Lemmas~\ref{l:fresse1}-\ref{l:fresse3} were formulated for left proper model categories, but similar statements
hold also in the relative case. To apply Lemma~\ref{l:fresse1} to the case $C$ is 
$\Lambda\Bimod_W$ or $\Lambda\Bimod_{W_1}$, one needs to assume that $c_1$ and $c_2$ are 
$\Sigma$-cofibrant. To apply Lemma~\ref{l:fresse2} to the adjunction~\eqref{eq:W_W1_adj_under} one needs $Lc$ and $d$ to be $\Sigma$-cofibrant. 
As a consequence, to apply an analogue of Lemma~\ref{l:fresse3} to the adjunction~\eqref{eq:W_W1_adj_under}, one needs to assume that $c$, $Qc$, $LQc$, $d$, and $Sd$
are $\Sigma$-cofibrant. One has $c=W$ and $d=W_1$ as cofibrant operads they are also
$\Sigma$-cofibrant by~\cite[Proposition~4.3]{Berger03}. We choose 
for $Qc$ the Boardman-Vogt resolution $\calB_W^\Lambda(W)$ of $W$ as a bimodule 
over itself. It is $\Sigma$-cofibrant by the same argument as the proof of Lemma~\ref{l:W1}
-- one  takes a filtration by the number of vertices. Similarly one proves that
$LQc=\phi_B^!\left(\calB_W^\Lambda(W)\right)$ is $\Sigma$-cofibrant. Elements of $\phi_B^!\left(\calB_W^\Lambda(W)\right)$ are trees with vertices labelled by $W$ and $W_1$ and 
by numbers in $[0,1]$. Only vertices assigned~1 are labelled by $W_1$. To prove the $\Sigma$-cofibrancy,
one should consider a double filtration by the total number of vertices and also by the number of 
vertices that are not labelled by~1. Finally for $Sd$ one can choose the functorial Reedy fibrant replacement
$W_1^f$ of $W_1$ as an operad. It has automatically the structure of a Reedy fibrant $W_1$-bimodule.
Since $W_1\to W_1^f$ is a cofibration of operads, $W_1^f$ is a cofibrant operad and therefore
is $\Sigma$-cofibrant.\vspace{5pt}

%
%

(ii) One can take $M_1:=\hat\phi_B^!(M)$. 
The map $M\to \hat\phi_B^* (\hat\phi_B^!(M))=M_1$
is a weak equivalence, because $W\to M$ is a cofibrant object in $W\downarrow\Lambda\Bimod_W$
and the adjunction~\eqref{eq:W_W1_adj_under} is a Quillen equivalence.
The square~\eqref{eq:M1} commutes  because the natural map $M\to M_1$ is a morphism in the under category $W\downarrow\Lambda\Bimod_W$. 
\end{proof}

We are now ready to finish the proof of Main Theorem~\ref{th:main}. For a weakly doubly reduced operad~$O$, one has a zigzag of equivalences
\[
O\xleftarrow{\simeq}WO\xrightarrow{\simeq}W_1O
\]
from Theorem~\ref{th:zigzag} in which $WO$ and $W_1O$ are Reedy cofibrant. 
Given a map $\eta\colon O\to M$ of $O$-bimodules, one has a  zigzag of derived mapping spaces of bimodules:
\begin{multline}\label{eq:zigzag1}
\Bimod_O^h(O,M)\simeq\Bimod_{WO}^h(O,M)\simeq\Bimod_{WO}^h(WO,M)
\simeq\\
\Bimod_{WO}^h(WO,M^c)\simeq\Bimod_{WO}^h(W_1O,M_1^c)\simeq\Bimod_{W_1O}^h(W_1O,M_1^c),
\end{multline}
where $M^c$ is obtained by taking the canonical factorization $WO\to M^c\xrightarrow{\simeq}M$
(in the category of $WO$-bimodules) into a cofibration and a trivial fibration of the
map $WO\to M$, and $M_1^c$ is obtained from $M^c$ using Proposition~\ref{p:M1}(ii). The first and the last equivalences are
 by Theorem~\ref{th:bimod_ind_restr}; the second, third, and fourth by replacing the source and/or target by an equivalent object. One has 
 a similar zigzag for Ibimodules mapping spaces:
 \begin{multline}\label{eq:zigzag2}
\Ibimod_O^h(O,M)\simeq\Ibimod_{WO}^h(O,M)\simeq\Ibimod_{WO}^h(WO,M)
\simeq\\
\Ibimod_{WO}^h(WO,M^c)\simeq\Ibimod_{WO}^h(W_1O,M_1^c)\simeq\Ibimod_{W_1O}^h(W_1O,M_1^c).
\end{multline}
By Theorem~\ref{th:main_2reduced}, the generalized delooping~\eqref{eq:equiv_general2_red} holds for the right-most mapping spaces
of~\eqref{eq:zigzag1} and~\eqref{eq:zigzag2}. Therefore it holds for the left-most mapping spaces as well. The truncated case
is obtained by the same zigzag.\hfill$\square$

\subsection{Explicit maps}\label{ss:proof:expl}
In this subsection we want to emphasize the fact that there are explicit maps realizing the equivalences of Main Theorem~\ref{th:main}. \vspace{5pt}

In case of a doubly reduced operad $O$,
in Subsection~\ref{s:map} besides the maps $\xi_\bullet^\Lambda$ that realize the equivalences~(\ref{eq:equiv_general1_red}-\ref{eq:equiv_general2_red})
of Theorem~\ref{th:main_2reduced}, we also constructed maps $\xi_\bullet^\Sigma$, see~(\ref{eq:equiv_red_pr1}-\ref{eq:equiv_red_pr2}).
The fact that $\xi_\bullet^\Sigma$ are equivalences follows easily from Theorems~\ref{th:bimod_maps} and~\ref{th:id_ibim_equiv}~(ii) and
 the fact that $\xi_\bullet^\Lambda$  are ones. This can also be shown directly, but the proof is more tedious than in the case of
 $\xi_\bullet^\Lambda$. According to Examples~\ref{ex:BprojO} and~\ref{ex:Ibproj} one would need to consider a double filtration
 in $\calB^\Sigma(O)$ and $\overline{\calI b}{}^\Sigma(O)$, as the inclusion from the $(k-1)$st filtration term to the $k$-th one for both of them is a sequence of $(k+1)$ cell attachments, and then to compare the mapping spaces for each term of this double filtration. This has been done for the associative operad
 by the second author in~\cite[Part~II]{Tourtchine10}.\vspace{5pt} 
 
 The maps $\xi_\bullet^\Lambda$ and $\xi_\bullet^\Sigma$ can also be generalized to the case  of any weakly doubly reduced operad,
 but in the latter case one would need to replace $\Sigma O(2)$ by the homotopy equivalent space $\calI b^\Lambda(O)(1)/\partial\calI b^\Lambda(O)(1)$.
 Figure~\ref{Fig:Ibim_weak} shows  how typical elements in $\calI b^\Lambda(O)(1)$ and $\partial\calI b^\Lambda(O)(1)$ look like.
  
 \begin{figure}[!h]
\begin{center}
\includegraphics[scale=0.3]{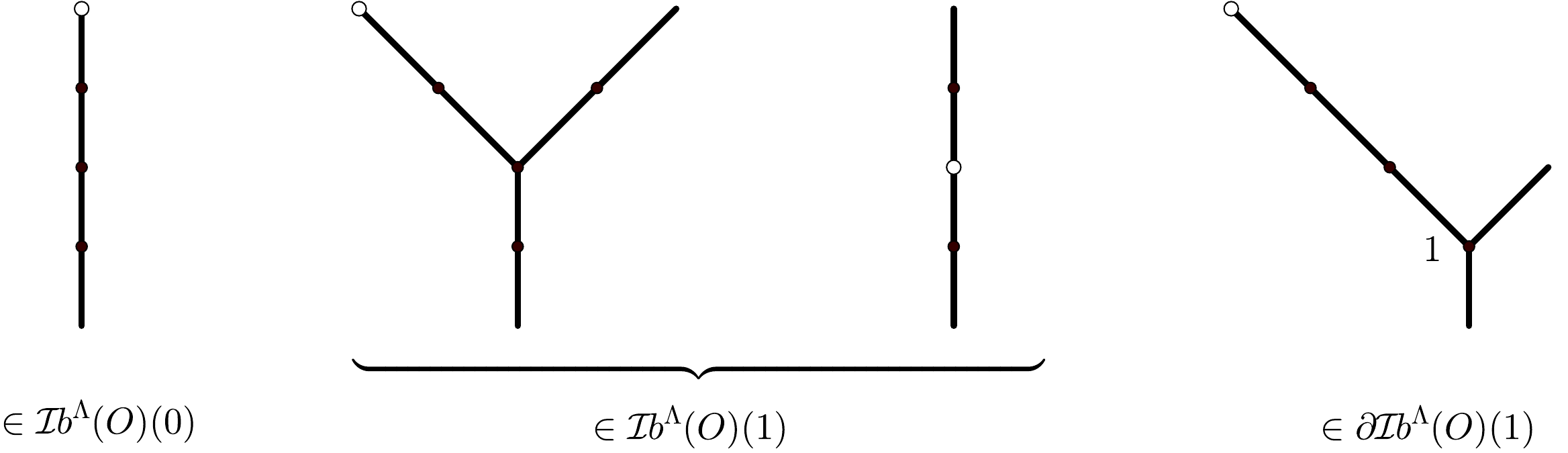}
\caption{Examples of elements in $\calI b^\Lambda(O)(0)$, $\calI b^\Lambda(O)(1)$, $\partial\calI b^\Lambda(O)(1)$.}\label{Fig:Ibim_weak}\vspace{-10pt}
\end{center}
\end{figure}


 We claim that the obtained maps 
\[
\hspace{-30pt}\begin{array}{l}\vspace{5pt}
\xi_k^\Sigma\colon  \Map_*\left(\calI b^\Lambda(O)(1)/\partial\calI b^\Lambda(O)(1), \TT_k\Bimod_O\left(\TT_k \calB_k^\Sigma(O),\TT_kM
\right)\right) \\ 
\hspace{230pt}\longrightarrow \TT_k\Ibimod_O\left(\TT_k \overline{\calI b}{}_k^\Sigma(O),\TT_kM\right),\, k\geq 0,
\end{array} 
\]
producing  at the limit $k=\infty$:
\[
\xi^\Sigma\colon  \Map_*\left(\calI b^\Lambda(O)(1)/\partial\calI b^\Lambda(O)(1), \Bimod_O\left(\calB^\Sigma(O),M
\right)\right) \longrightarrow
\Ibimod_O\left(\overline{\calI b}{}^\Sigma(O),M\right);
\]
and maps
\[
\hspace{-20pt}\begin{array}{l}\vspace{5pt}
\xi_k^\Lambda\colon  \Map_*\left(\calI b^\Lambda(O)(1)/\partial\calI b^\Lambda(O)(1), \TT_k\Bimod_O\left(\TT_k \calB^\Lambda(O),\TT_kM^f
\right)\right) \\ 
\hspace{240pt}\longrightarrow \TT_k\Ibimod_O\left(\TT_k \overline{\calI b}{}^\Lambda(O),\TT_kM^f\right),\, k\geq 0,
\end{array} 
\]
producing  at the limit $k=\infty$:
\[
\xi^\Lambda\colon  \Map_*\left(\calI b^\Lambda(O)(1)/\partial\calI b^\Lambda(O)(1), \Bimod_O\left(\calB^\Lambda(O),M^f
\right)\right) \longrightarrow
\Ibimod_O\left(\overline{\calI b}{}^\Lambda(O),M^f\right)
\]
 realize the equivalences of Main Theorem~\ref{th:main}. We failed to find a direct proof of this statement,
  but it follows easily from the
 fact that the maps $\xi_\bullet^\Lambda$ and $\xi_\bullet^\Sigma$ are functorially defined, and thus can be shown to commute 
 with the zigzags~\eqref{eq:zigzag1} and~\eqref{eq:zigzag2} from the previous subsection.

\subsection{Main Theorem~\ref{th:main3}}\label{ss:last}
Even though we constructed an explicit map that proves equivalence~\eqref{eq:equiv_general2}
of Main Theorem~\ref{th:main}, this map still seems mysterious. On the other hand,
the result of this theorem and its underlying construction has an easy generalization that we
present below. It  shed some light on the algebraic nature of this construction. \vspace{5pt}

Recall Definitions~\ref{d:rho},~\ref{d:coherent}, and~\ref{d:scoherent}.

\begin{defi}\label{d:coherent2}
An infinitesimal bimodule $N$ over a weakly doubly reduced operad $O$ is called {\it (strongly) coherent} if each diagram $\rho_i^N$, $i\geq 2$, is (strongly) coherent.
\end{defi}

Let $C_N$ denote the mapping cone of the infinitesimal left action map $\circ_1\colon O(2)\times N(0)\to N(1)$ pointed in the 
apex of the cone.
Note that $C_N$ is determined by the 1-truncation $\TT_1N$ of $N$.

\begin{mainthm}\label{th:main3}
Let $O$ be a weakly doubly reduced, $\Sigma-cofibrant$, and well-pointed topological operad,
let $\mu\colon N\to O$ be a map of $O$-Ibimodules, $\eta\colon O\to M$ be a map of
$O$-bimodules, and $M(0)=*$. If $N$ is coherent, then one has an equivalence of  towers
\begin{equation}\label{eq:equiv_general3}
\TT_k\Ibimod_O^h\left(\TT_kN,\TT_kM\right)\simeq \Map_*\left(C_N, \TT_k\Bimod_O^h\left(\TT_kO,\TT_kM
\right)\right),\, k\geq 0,
\end{equation}
implying at the limit $k=\infty$:
\begin{equation}\label{eq:equiv_general4}
\Ibimod_O^h\left(N,M\right)\simeq \Map_*\left(C_N, \Bimod_O^h\left(O,M
\right)\right).
\end{equation}
\end{mainthm}

The left-hand sides in \eqref{eq:equiv_general4} and~\eqref{eq:equiv_general3} are pointed in
$\eta\circ\mu$ and $\TT_k(\eta\circ\mu)$.

\begin{proof}[Sketch of the proof]
First prove it in the case of a stritctly doubly reduced operad $O$, then generalize it to the
weakly doubly reduced case. (One has to use the fact that the model category $\Lambda\Ibimod_O$
is right proper for any $O$, which is a consequence of the fact that $\Topo$ is so and pullbacks
in $\Lambda\Ibimod_O$ are obtained objectwise.)\vspace{5pt}

For the strictly doubly reduced case, one has $\calI b^\Lambda(N)(1)/\partial
\calI b^\Lambda(N)(1)=C_N$ (compare with Example~\ref{D2}). On the other hand, one can show that $\overline{\calI b}{}^\Lambda(N):=
\calI b^\Lambda(N){\ltimes}\calB^\Lambda(O)$ is homeomorphic to $\calI b^\Lambda(N)$ 
(which is analogous to Theorem~\ref{H5}).  By the general Construction~\ref{const:map}, one gets a map
\[
\Map_*\left(C_N, \Bimod_O\left(\calB^\Lambda(O),M^f
\right)\right) \to
\Ibimod_O\left(\overline{\calI b}{}^\Lambda(N),M^f\right),
\]
which similarly to the argument in Subsection~\ref{ss:proof:2reduced} is proved to be an equivalence 
provided $N$ is coherent. 
\end{proof}

This result  implies Main Theorem~\ref{th:main} by taking $N=O$. 
  Note  that  in case $N$ is strongly coherent, each $N(k)$, $k\geq 2$, can be expressed
 from $N(0),\ldots,N(k-1)$ by means of a homotopy colimit~\eqref{eq:scoherent}. 
However, the structure of a $k$-truncated infinitesimal bimodule can not be recovered from the $(k-1)$-truncation. It would be interesting to find more examples of (strongly) coherent infinitesimal bimodules,
and to explore further their connection with the manifold functor calculus.

\appendix

\section{Appendix}

\subsection{A convenient category of topological spaces}\label{s:A0}

Below we recall  basic properties of the category $\Topo$ of $k$-spaces 
 from~\cite[Appendix~A]{Lewis78} and~\cite[Section~2.4]{Hovey99}, where this category is denoted by~$\textbf{K}$.\vspace{5pt}

For any topological space $X$, its kelleyfication $kX$  is the refinement space of $X$  in which
$U\subset X$ is open if and only if for any continuous map $f\colon K\to X$ with a compact source,
$f^{-1}(U)$ is open in $K$. The (identity) map $kX\to X$ is continuous  and induces a bijective map of  the sets 
of singular chains, therefore it is a weak homotopy equivalence. A space $X$ is called {\it $k$-space}
if $kX=X$. One has an adjunction
\begin{equation}\label{eq:k_adj}
i\colon\Topo\leftrightarrows\AllTop\colon k
\end{equation}
between the category $\Topo$ of $k$-spaces and the category $\AllTop$ of all topological spaces.
(Here $i$ is the inclusion functor.)  Both 
categories admit a cofibrantly generated model structures, where in both cases equivalences are
weak homotopy equivalences, the set of generating cofibrations is $\{S^{n-1}\hookrightarrow D^n\}_{n\geq 0}$ and that of generating acyclic cofibrations is $\{D^n\hookrightarrow D^n\times [0,1]\}_{n\geq 0}$.
With respect to these model structures, the adjunction~\eqref{eq:k_adj} is a Quillen equivalence~\cite[Theorem~2.4.23]{Hovey99}. Limits in $\Topo$ are kelleyfication of the usual limits. Colimits are the 
usual colimits since any usual colimit of $k$-spaces is always a $k$-space.
It follows from this adjunction, that a map $X\to kY$ from a $k$-space $X$ to the kelleyfication of another space is
continuous if and only if the composition $X\to kY\to Y$ is so.\vspace{5pt}

The category $\Topo$ is a so called {\it convenient} topological category~\cite{Steenrod67,Vogt71}: it contains all $CW$-complexes, it is complete and cocomplete,
and 
it is cartesian closed meaning that for any $X,Y,Z\in\Topo$ one has homeomorphisms:
\begin{equation}\label{eq:exp2}
\Map(X,\Map(Y,Z))\cong\Map(X\times Y,Z)\cong \Map(Y,\Map(X,Z)).
\end{equation}
From~\eqref{eq:exp2} we get that for any $Y,Z\in\Topo$ one has adjunctions
\begin{equation}\label{eq:adj_top1}
(-)\times Y\colon\Topo\leftrightarrows\Topo\colon\Map(Y,-);
\end{equation}
\begin{equation}\label{eq:adj_top2}
\Map(-,Z)\colon\Topo\leftrightarrows\Topo^{op}\colon\Map(-,Z).
\end{equation}
The first adjunction~\eqref{eq:adj_top1} implies that the product $(-)\times Y$ distributes
over colimits:
\[
\left(\underset{I}{\operatorname{colim}}\,X_i\right)\times Y\cong \underset{I}{\operatorname{colim}}
\left(X_i\times Y\right),
\]
which is a prerequisite to define the theory of operads on the monoidal category $(\Topo,\times,*)$;
and one also gets that $\Map(Y,-)$ preserves limits:
\[
\Map\left(Y,\underset{I}{\operatorname{lim}}\, Z_i\right)\cong\underset{I}{\operatorname{lim}}\,\Map\left(Y,Z_i\right).
\]
From the second adjunction~\eqref{eq:adj_top2}, $\Map(-,Z)$ converts colimits to limits:
\begin{equation}\label{eq:colim_lim}
\Map\left(\underset{I}{\operatorname{colim}}\,X_i,Z\right)\cong
\underset{I}{\operatorname{lim}}\, \Map\left(X_i,Z\right).
\end{equation}

\begin{lmm}\label{l:q_map}
For any $k$-spaces $X$ and $Y$, and any equivalence relation $\sim$ on $X$, the natural inclusion map
$\Map(X/{\sim},Y)\hookrightarrow \Map(X,Y)$ is a homeomorphism on its image.
\end{lmm}
\begin{proof}
Consider the pushout diagram 
\begin{equation}\label{eq:q_map}
(X/{\sim})^\delta\leftarrow X^\delta\rightarrow X,
\end{equation}
 where $(-)^\delta$ is the discretization functor. The colimit of~\eqref{eq:q_map} is $X/{\sim}$. 
 Applying~\eqref{eq:colim_lim} to this diagram, the left-hand side is $\Map(X/{\sim},Y)$,
 while the right-hand side is exactly the image subspace in $\Map(X,Y)$.
 \end{proof}

 Recall again that the mapping spaces and subspaces in this lemma (and everywhere throughout the paper)
 are taken with the kelleyfication of the standard compact-open and subspace topologies. In $\AllTop$
 the inclusion map $\Map(X/{\sim},Y)\hookrightarrow \Map(X,Y)$ may not even be a weak 
 homotopy equivalence on its image, see example~\cite{mathoverflow}.
 
%

\subsection{Proof of Theorem~\ref{th:zigzag}}\label{s:A1}
Without loss of generality, one can assume that the operad $O$ is well-pointed and $\Sigma$-cofibrant. (If 
not, one can replace $O$ by $|C_\bullet O|\times\calB_\infty$, where $C_\bullet O$ is the simplicial 
operad of singular chains on $O$, and $|-|$ denotes the realization functor. The operad $|C_\bullet O|$ 
is well pointed and has cofibrant components. By taking the objectwise product with $\calB_\infty$,
the resulting operad is $\Sigma$-cofibrant.) Define an operad $WO$ as the Boardman-Vogt resolution of 
$O_{>0}$, see~\cite{Boardman68}, in positive arities and as a point in arity~0. Composition with
$*_0^W\in WO(0)$ is defined in the same way as the right action of $*_0$ on $\calB^\Lambda(M)$ 
and on $\calI b^\Lambda(M)$, see Section~\ref{s:cof}. \vspace{5pt}

The operad $W_1O$ is obtained from $WO$ by collapsing the arity one component $WO(1)$
to a point and quotienting out the other components by the induced relations.\vspace{5pt} 

Recall that $WO(k)$ is the space
\begin{equation}\label{eq:WO}
WO(k)=\left.\left(\coprod_{T\in \textbf{tree}_k^{\geq 1}} [0,1]^{E^{int}(T)}\times 
\prod_{v\in V(T)} O(|v|)\right)\right/\sim.
\end{equation}
 The relations are: relation~(ii) of Construction~\ref{B0}; analogue of 
relation~(i) of Construction~\ref{B0}:

\begin{center}
\includegraphics[scale=0.19]{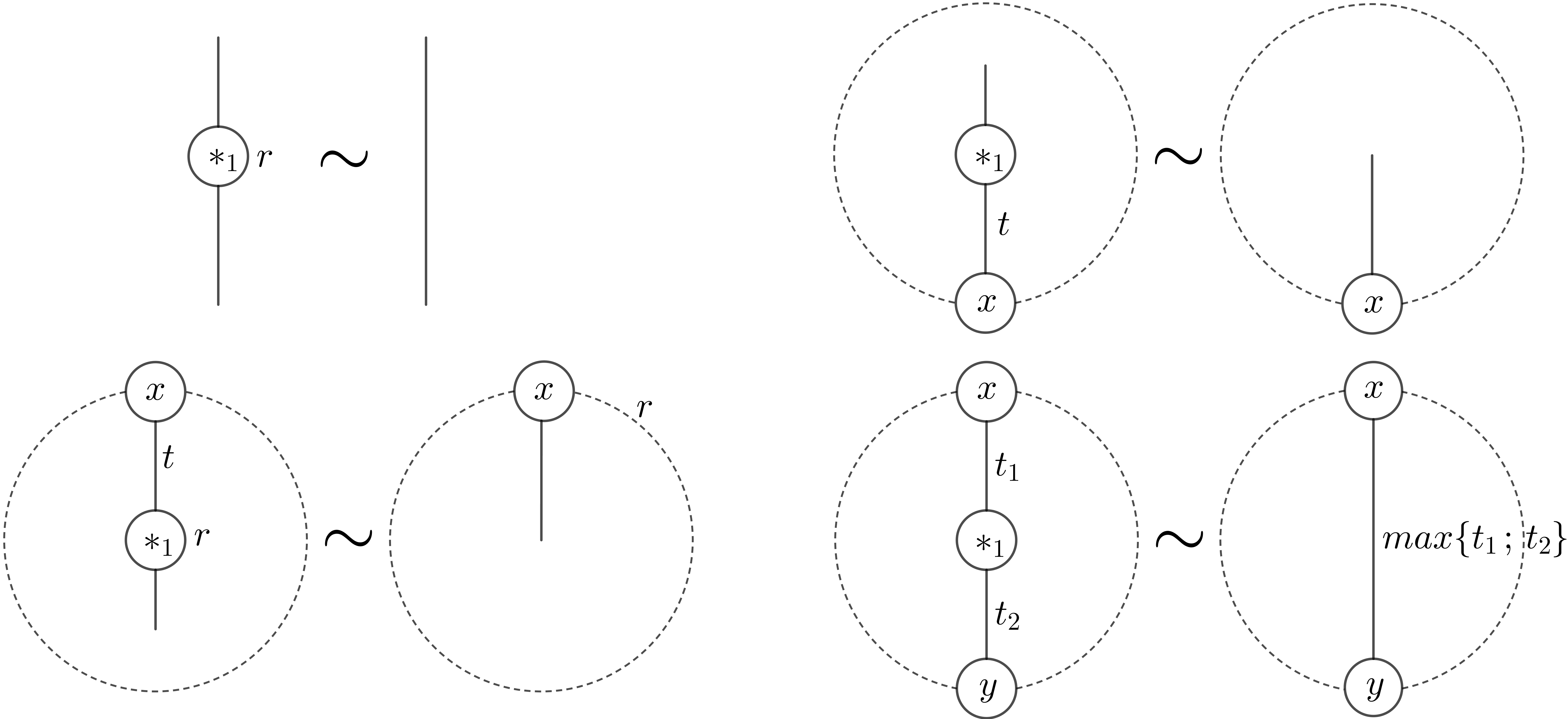}\vspace{-2pt}
\end{center}

\noindent (the relations with dotted circles are \lq\lq{}local\rq\rq{} in the sense that the trees outside 
the dotted circles are the same for left and right sides); and composition relations:

\begin{center}
\includegraphics[scale=0.21]{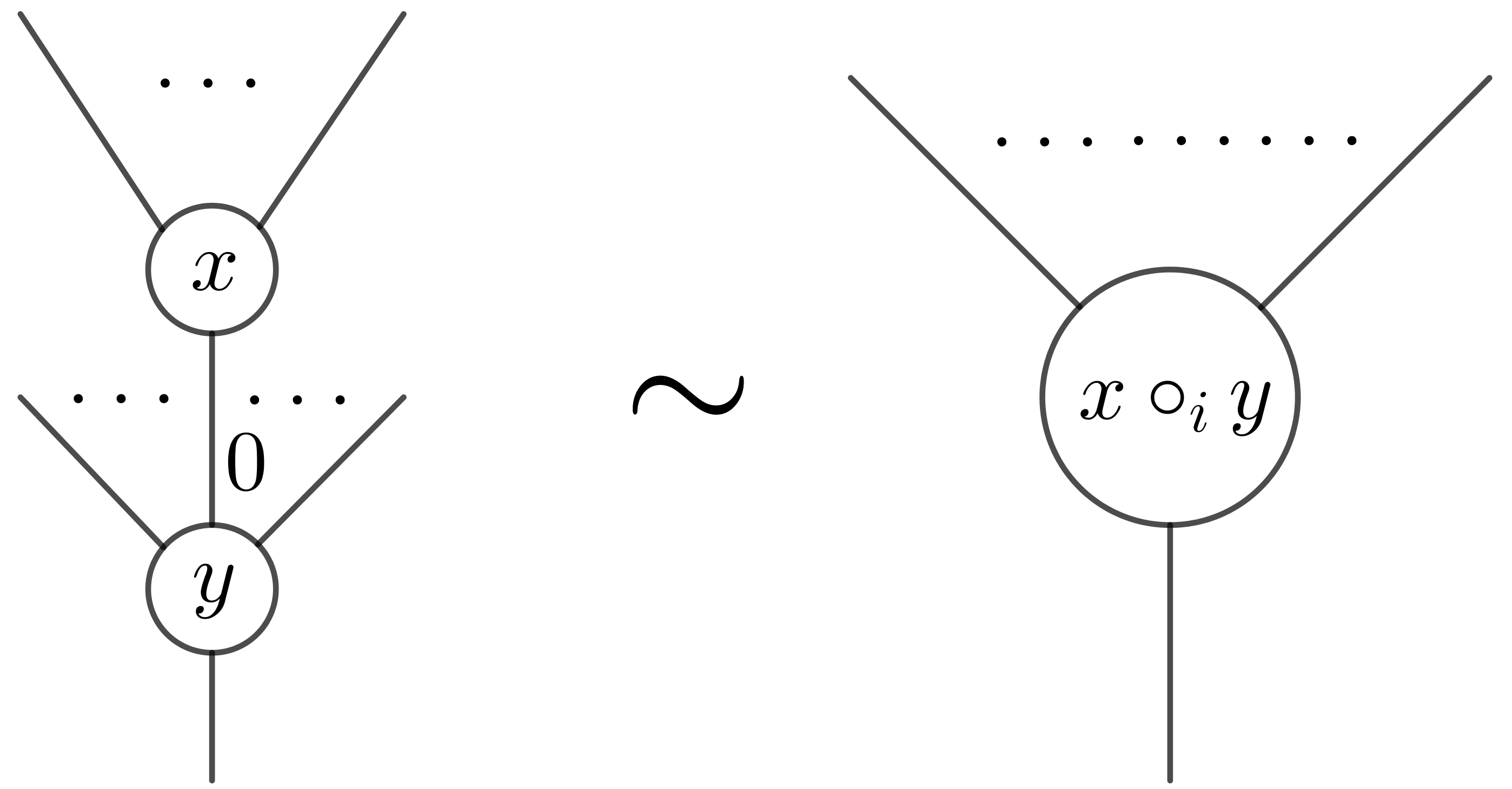}\vspace{-2pt}
\end{center}

To pass to $W_1O$, one adds the local relations from Figure~\ref{Fig:W1}.\vspace{-10pt}

\begin{figure}[!h]
\begin{center}
\includegraphics[scale=0.14]{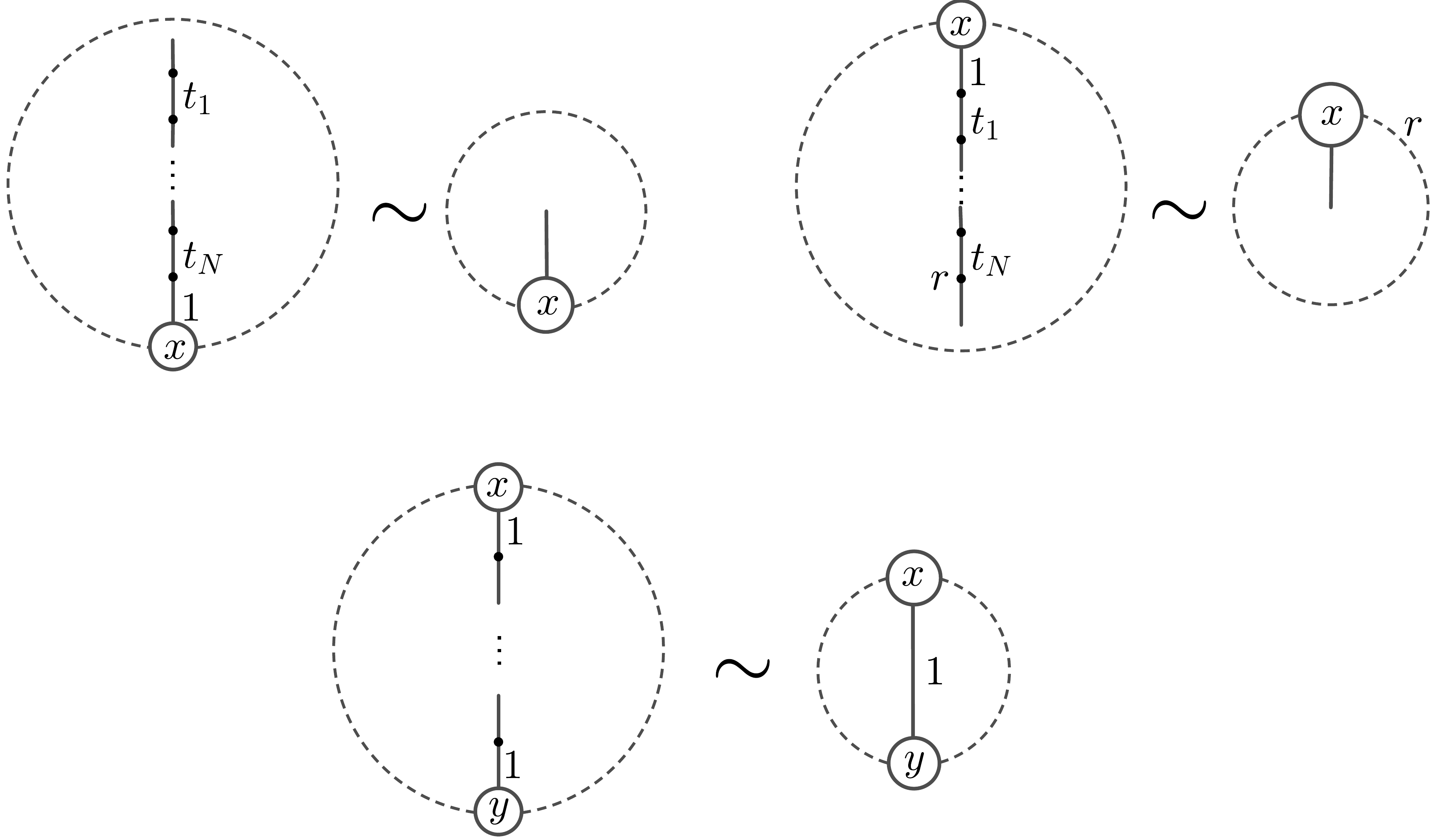}\vspace{-13pt}
\caption{Additional relations in $W_1O$.}\label{Fig:W1}\vspace{-42pt}
\end{center}
\end{figure}

\newpage

We need to show that the quotient map $WO(k)\to W_1O(k)$ is a weak equivalence. For $k=0$, both
spaces are singletons. For $k=1$, the second space is a point, while the first one is homotopy equivalent 
to $O(1)$ therefore is also contractible. We concentrate below on the case $k\geq 2$. We also
need to show that $WO$ and $W_1O$ are Reedy cofibrant operads, or, equivalently,  that $WO_{>0}$ and
$W_1O_{>0}$ are cofibrant in the projective model structure, see Subsection~\ref{ss:ht2} 
or~\cite[Theorem II.8.4.12]{Fresse17}. The cofibrancy of $WO_{>0}$ is guaranteed by the fact that $O_{>0}$
is well-pointed and $\Sigma$-cofibrant, see~\cite{Boardman68}. \vspace{5pt}

Consider an increasing filtration $F_iWO(k)$, respectively $F_iW_1O(k)$, in $WO(k)$, respecticely $W_1O(k)$,
by the number $i$ of vertices of the trees in $\textbf{tree}_k^{\geq 1}$. The map
\begin{equation}\label{eq:W-W1}
WO(k)\to W_1O(k)
\end{equation}
preserves this filtration. One also has $F_1WO(k)=F_1W_1O(k)=O(k)$, $k\geq 2$. The following lemma
implies that~\eqref{eq:W-W1} is a weak homotopy equivalence.

\begin{lmm}\label{l:W1}
For any weakly doubly reduced, well-pointed, and $\Sigma$-cofibrant operad $O$, the inclusions
\begin{equation}\label{eq:F_iWO}
F_{i-1}WO(k) \to F_iWO(k)\,\, \text{ and }\,\, F_{i-1}W_1O(k) \to F_iW_1O(k)
\end{equation}
are trivial $\Sigma_k$-cofibrations, $i\geq 2$, $k\geq 2$. 
\end{lmm}

\begin{proof}
This fact for the Boardman-Vogt resolutions (first inclusion in~\eqref{eq:F_iWO}) was proved by 
Berger-Moerdijk in~\cite{Berger06}. Below we recall their argument and then adapt it to the case
of the second inclusion in~\eqref{eq:F_iWO}. \vspace{5pt}

For any tree $(T,id)\in\textbf{tree}_k^{\geq 1}$ with exactly $i$ vertices, denote by $Aut(T)\subset \Sigma_k$
the group of automorphisms of $T$. Define $O(T):=\prod_{v\in V(T)}O(|v|)$ and $H(T):=[0,1]^{E^{int}(T)}$.
Define also $O^-(T)\subset O(T)$ as the subset of elements for which at least one coordinate
is $*_1\in O(1)$; and $H^-(T)\subset H(T)$ as the subset for which at least one coordinate is $0\in[0,1]$. 

\begin{lmm}[Equivariant pushout-product lemma - Lemma 2.5.2 in~\cite{Berger06}]
\label{l:pushprod}
Let $\Gamma$ be a discrete group. Consider the projective model structure on $\Topo^\Gamma$. Let $A\hookrightarrow B$ and $X\hookrightarrow Y$ be $\Gamma$-equivariant cofibrations. If one of them is
a $\Gamma$-cofibration, the pushout-product map $A\times Y\bigcup B\times X
\hookrightarrow B\times Y$ is a $\Gamma$-cofibration. Moreover, the latter is trivial if 
$A\hookrightarrow B$ or $X\hookrightarrow Y$ is.
\end{lmm}

Berger-Moerdijk showed that for $O$ well-pointed and $\Sigma$-cofibrant,
the inclusion $O^-(T)\subset O(T)$ is an $Aut(T)$-cofibration and $H^-(T)\subset H(T)$ 
is a trivial $Aut(T)$-equivariant cofibration. From Lemma~\ref{l:pushprod} it follows that the
inclusion
\[
(H(T)\times O(T))^-:= H^-(T)\times O(T)\bigcup H(T)\times O^-(T)
\subset H(T)\times O(T)
\]
is a trivial $Aut(T)$-cofibration. As a consequence, the inclusion 
\begin{equation}\label{eq:OH}
(H(T)\times O(T))^-\times_{Aut(T)}\Sigma_k\subset (H(T)\times O(T))\times_{Aut(T)}\Sigma_k
\end{equation}
is a trivial $\Sigma_k$-cofibration. On the other hand, it is easy to see that $F_iWO(k)$ is obtained
from $F_{i-1}WO(k)$ by taking pushout along cofibrations~\eqref{eq:OH} over all (non-planar)
trees $T$ with $k$ leaves and $i$ vertices (all of arity $\geq 1$). \vspace{5pt}

Let us adapt this argument to the second inclusion in~\eqref{eq:F_iWO}. For any tree $T$
as above, define $\widetilde{H}{}^-(T)\subset H(T)$ as the set of points having at least one coordinate $0
\in [0,1]$ plus the points (markings of the edges) that would appear in relations of Figure~\ref{Fig:W1}. 
It is easy to see that $\widetilde{H}{}^-(T)$ is a union of 
some codimension~0 and~1 faces of the cube $H(T)$. Therefore  $\widetilde{H}{}^-(T)\subset H(T)$
is an $Aut(T)$-equivariant cofibration, trivial in case all vertices of~$T$ have arities $\geq 2$. 
On the other hand, by the usual pushout product lemma ($\Gamma=1$ in Lemma~\ref{l:pushprod}),
the $Aut(T)$-cofibration $O^-(T)\subset O(T)$  is trivial in case $T$ has at least one vertex of arity~1. 
Thus every inclusion
\[
 \widetilde{H}{}^-(T)\times O(T)\bigcup H(T)\times O^-(T)
\subset H(T)\times O(T)
\]
is a trivial $Aut(T)$-cofibration and we conclude similarly that $F_{i-1}W_1O(k)\hookrightarrow
F_iW_1O(k)$ is a trivial $\Sigma_k$-cofibration.
\end{proof}

To prove that the operad $W_1O_{>0}$ is cofibrant, we consider a different filtration in $W_1O_{>0}$. 
We say that an element $x\in W_1O_{>0}$ is prime if all edges of $x$ are assigned length $<1$.
To any $x\in W_1O_{>0}$ one can assign the set of its prime components by splitting the underline tree
into subtrees cutting all edges of length~1. We say that $x$ is in the $i$-th filtration term $W_1^{(i)}O_{>0}$
if all its prime components have $\leq i$ vertices. Each term $W_1^{(i)}O_{>0}$ is an operad. Moreover,
each inclusion $W_1^{(i-1)}O_{>0}\subset W_1^{(i)}O_{>0}$ is an operadic cofibration. Indeed,
it is obtained as a sequence of free (operadic) cell attachments along $\Sigma_k$-cofibrations:
\[
(\partial H(T)\times O(T)\bigcup H(T)\times O^-(T))\times_{Aut(T)}\Sigma_k
\subset (H(T)\times O(T))\times_{Aut(T)}\Sigma_k,
\]
where $\partial H(T)$ is the boundary of the cube $H(T)$, and $T$ is a tree with $i$ internal vertices
(each of arity $\geq 1$) and with $k\geq 2$ leaves. (It is easy to see that by the same argument
 $WO_{>0}$ is obtained as a sequence of cell attachments. Moreover $W_1O_{>0}$ is obtained
 from $WO_{>0}$ by collapsing to a point all its arity $k=1$ cells.)

\subsection{Cellular cosheaves}\label{s:A2}
In this section we explain what we mean by a {\it cellular cosheaf} and its {\it realization}.  
Given a CW-complex $X$, we consider the poset structure on the set 
$I$ of its cells generated by relations $i<j$ if the interior of a cell $i$ has a non-empty intersection with
the boundary of a cell $j$. Denote by $d\colon I\to {\mathbb N}$ the associated dimension function,
which is obviously order-preserving. One has 
\[
X=\left.\left(\coprod_{i\in I} D^{d(i)}\right)\right/\sim,
\]
where $\sim$ says how the boundary of each disc $D^{d(i)}$ is attached to the $(d(i)-1)$-skeleton
of $X$. \vspace{5pt}

Viewing $I$ as a category, we say that a {\it cellular cosheaf} on $X$ is a contravariant functor
${\mathcal F}\colon I\to\Topo$. Thus for any $i<j$, one gets a map ${\mathcal F}(j)\to{\mathcal F}(i)$.

By {\it realization} of a cellular cosheaf ${\mathcal F}$ we understand the space
\[
|{\mathcal F}|:=\left.\left(\coprod_{i\in I}{\mathcal F}(i)\times D^{d(i)}\right)\right/\sim,
\]
where $\sim$ is generated by $(a_1,x_1)\sim (a_2,x_2)$ with $x_1\in D^{d(i)}$, $x_2\in \partial D^{d(j)}$,
$a_1\in {\mathcal F}(i)$, $a_2\in {\mathcal F}(j)$, $i<j$, $x_1\sim x_2$, and $a_2\mapsto a_1$ 
by the map ${\mathcal F}(i<j)\colon {\mathcal F}(j)\to {\mathcal F}(i)$.\vspace{5pt}

In case a CW-complex $Y$ is a cellular refinement of $X$ along a homeomorphism 
$f\colon Y\to X$, the poset $J$ of cells of $Y$ admits a poset map $f_*\colon J\to I$, which assigns
to a cell of $Y$ the cell of $X$ in which interior it lies.  One can also define the cellular cosheaf
$f^*{\mathcal F}$ on $Y$ as the composition $f^*{\mathcal F}\colon J\xrightarrow{f_*}I \xrightarrow{\mathcal F}Top$. \vspace{5pt}

The following lemma is easy to check.

\begin{lmm}\label{l:cellularsheaf}
If $f\colon Y\xrightarrow{\cong} X$ is a cellular refinement, ${\mathcal F}$ is a cellular cosheaf on $X$,
then one has a homeomorphism $|f^*{\mathcal F}|\cong|{\mathcal F}|$.
\end{lmm}


\bibliographystyle{amsalpha}
\bibliography{bibliography2}


%

\vspace{20pt}

\noindent Julien Ducoulombier: Department of Mathematics, ETH Zurich, Ramistrasse 101, Zurich, Switzerland\\
\textit{E-mail address: } \href{mailto:julien.ducoulombier@math.ethz.ch}{julien.ducoulombier@math.ethz.ch}

\vspace{20pt}

\noindent Victor Turchin: Department of Mathematics, Kansas State University, Manhattan, KS 66506, USA\\
\textit{E-mail address: } \href{mailto:turchin@ksu.edu}{turchin@ksu.edu}

\end{document}